\documentclass[a4paper,10pt,reqno]{amsart}
\usepackage{amscd,amssymb,graphicx,color,epigraph}

\usepackage[shortalphabetic,initials]{amsrefs}

\newcommand{\red}[1]{{\color{red} #1}}

\usepackage{enumitem}
\setlist[itemize]{leftmargin=18pt}
\setlist[enumerate]{leftmargin=18pt}

\usepackage{amssymb}
\usepackage{amsthm}
\usepackage{bm}
\usepackage{latexsym}
\usepackage{amsmath}
\usepackage{hyperref}
\usepackage{eufrak}
\usepackage{mathrsfs}
\usepackage{amscd}
\usepackage[all,cmtip]{xy}
\usepackage{color}
\usepackage{amscd}
\usepackage{listings}
\usepackage[center]{caption}
\theoremstyle{plain}

 \numberwithin{equation}{section}

\newtheorem{theorem}{Theorem}[section]
\newtheorem{proposition}[theorem]{Proposition}

\newtheorem{lemma}[theorem]{Lemma}
\newtheorem{lemdef}[theorem]{Lemma-Definition}
\newtheorem{corollary}[theorem]{Corollary}

\theoremstyle{definition}

\newcommand{\appsection}[1]{\let\oldthesection\thesection
\renewcommand{\thesection}{Appendix \oldthesection}
\section{#1}\let\thesection\oldthesection}

\newtheorem{definition}[theorem]{Definition}
\newtheorem{notation}[theorem]{Notation}
\newtheorem{Assumption}[theorem]{Assumption}

\newtheorem{remark}[theorem]{Remark}
\newtheorem{example}[theorem]{Example}

\DeclareMathOperator{\RHom}{RHom}
\DeclareMathOperator{\Qcoh}{Qcoh}

\DeclareMathOperator{\Mod}{Mod}
\DeclareMathOperator{\mmod}{mod}
\DeclareMathOperator{\Cl}{Cl}

\DeclareMathOperator{\Pic}{Pic}
\DeclareMathOperator{\ch}{ch}
\DeclareMathOperator{\perf}{perf}

\DeclareMathOperator{\ext}{ext}
\DeclareMathOperator{\depth}{depth}

\DeclareMathOperator{\End}{End}
\DeclareMathOperator{\Mat}{Mat}

\DeclareMathOperator{\Coker}{Coker}
\DeclareMathOperator{\rk}{rk}

\DeclareMathOperator{\rank}{rank}

\def\bB{{\Bbb B}}
\def\bR{{\Bbb R}}

\def\D{{\mathbb{D}}}

\def\Z{{\mathbb{Z}}}

\def\bE{{\mathbb{E}}}

\def\F{{\mathbb{F}}}
\def\Q{{\mathbb{Q}}}
\def\C{{\mathbb{C}}}
\def\P{{\mathbb{P}}}

\def\B{{\mathbb{B}}}

\def\cB{{\mathcal{B}}}
\def\cA{{\mathcal{A}}}
\def\cO{{\mathcal{O}}}
\def\cD{{\mathcal{D}}}

\def\cY{{\mathcal{Y}}}
\def\cHom{{\mathcal{H}om}}
\def\RcHom{R{\mathcal{H}om}}
\def\cExt{{\mathcal{E}xt}}

\def\cC{{\mathcal{C}}}
\def\cE{{\mathcal{E}}}

\def\cF{{\mathcal{F}}}
\def\cZ{{\mathcal{Z}}}

\def\O{{\mathcal{O}}}

\def\cX{{\mathcal{X}}}
\def\cT{{\mathcal{T}}}
\def\cW{{\mathcal{W}}}
\def\ocW{\overline{\mathcal{W}}}

\def\W{{\mathcal{W}}}
\def\Z{{\mathbb{Z}}}

\def\oR{{\overline{R}}}
\def\oY{{\overline{Y}}}
\def\oW{{\overline{W}}}

\def\epsilon{\varepsilon}

\DeclareMathOperator{\Def}{Def}
\DeclareMathOperator{\Bl}{Bl}

\DeclareMathOperator{\Hom}{Hom}
\DeclareMathOperator{\Tor}{Tor}
\DeclareMathOperator{\Ext}{Ext}

\DeclareMathOperator{\Spec}{Spec}

\makeatletter
\providecommand{\leftsquigarrow}{%
  \mathrel{\mathpalette\reflect@squig\relax}%
}
\newcommand{\reflect@squig}[2]{%
  \reflectbox{$\m@th#1\rightsquigarrow$}%
}
\makeatother

\title[Categorical aspects of the Koll\'ar--Shepherd-Barron correspondence]{Categorical aspects of\\ the Koll\'ar--Shepherd-Barron correspondence}

\author{Jenia Tevelev}
\address{Department of Mathematics \& Statistics, University of Massachusetts, Amherst, USA. 
}
\email{tevelev@cns.umass.edu}

\author{Giancarlo Urz\'ua}
\address{Facultad de Matem\'aticas,
Pontificia Universidad Cat\'olica de Chile, Santiago, Chile.}
\email{urzua@mat.uc.cl}

\date{\today}

\begin{document}


\begin{abstract}
It is well known that a $2$-dimensional cyclic quotient singularity $\oW$ has the same singularity category as a finite dimensional associative algebra $\oR$ introduced by Kalck and Karmazyn. We study the deformations of the algebra $\oR$ induced by the deformations of the surface $\oW$ to a smooth surface. We show that they are Morita--equivalent to path algebras $\hat R$ of  acyclic  quivers for general smoothings within each irreducible component of the versal deformation space of $\oW$ (as described by Koll\'ar and Shepherd-Barron). Furthermore, $\hat R$ is semi-simple if and only if the smoothing is $\Q$-Gorenstein (one direction is due to Kawamata).
We provide many applications. For example, we describe strong exceptional collections of length $10$ on all Dolgachev surfaces and classify admissible embeddings of derived categories of quivers into derived categories of rational surfaces. 
\end{abstract}

\dedicatory{To Professor Yujiro Kawamata on the occasion of his 70th birthday}

\maketitle

\section{Introduction}
In the classical papers from 1974, Pinkham \cite{Pi74} and Gabriel \cite{Gab74} studied deformations of varieties with ${\mathbb G}_m$-action and finite dimensional associative algebras, respectively. These papers included famous examples of reducible versal deformation spaces of the cone $\oW$ over a rational normal curve in $\P^4$ (Figure~\ref{1974figs}, left) and of the $4$-dimensional algebra $\oR=\C[x,y,z]/(x,y,z)^2$ (Figure~\ref{1974figs}, right), respectively.

\begin{figure}[h]
\center
\includegraphics[width=0.8\textwidth]{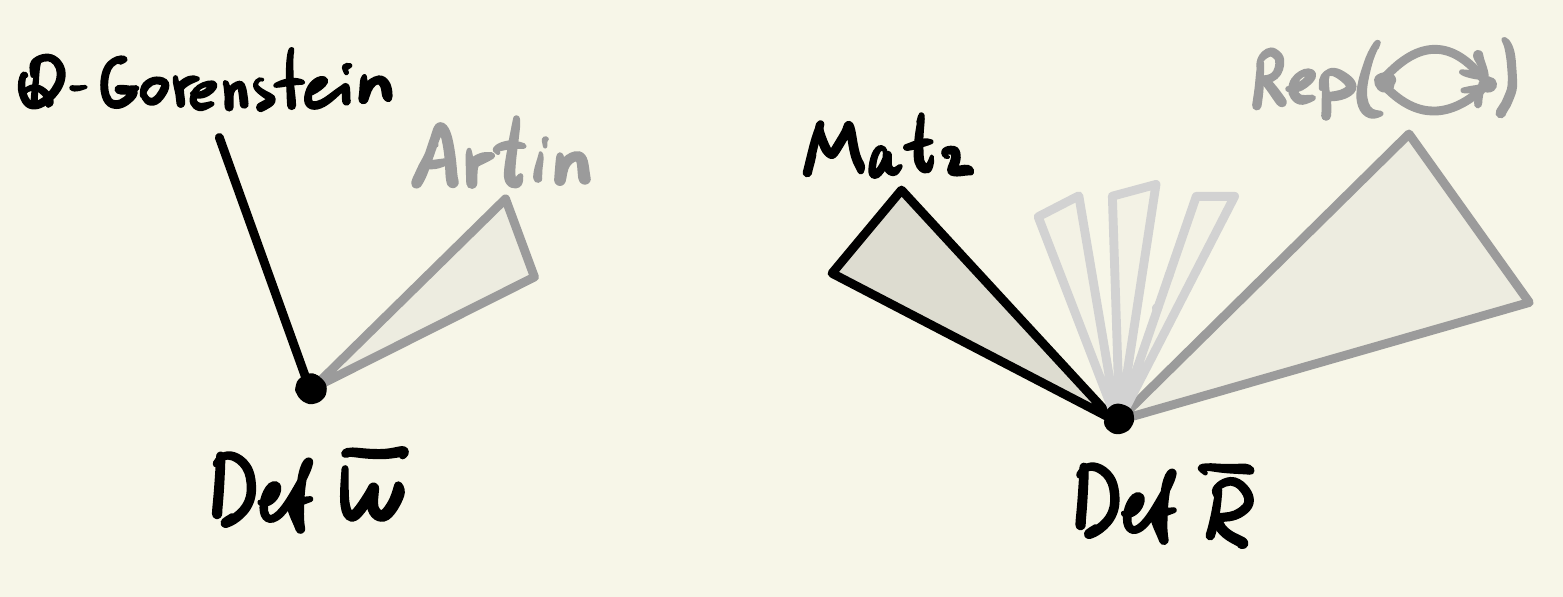}
\caption{Versal deformation spaces of $\oW$ and $\oR$}
\label{1974figs}
\end{figure}

We observe that there is a remarkable embedding $\psi \colon \Def\oW\to\Def\oR$.
The~Artin component of $\Def\oW$ (that parametrizes deformations induced by deformations of the minimal resolution of $\oW$) maps to the deformations of $\oR$ to the path algebra of the Kronecker quiver, whereas the $\Q$-Gorenstein component  
(that parametrizes deformations of $\oW$ such that the relative canonical divisor is $\Q$-Cartier) maps to the deformations of $\oR$ to the matrix algebra $\Mat_2(\C)$. 
We will construct the map~ $\psi$ 
for all cyclic quotient singularities $\oW$ and show that $\psi$ sends each irreducible component of $\Def\oW$  to a uniquely determined irreducible component of $\Def\oR$.

Throughout the paper 
we fix a cyclic quotient singularity (c.q.s.) of type $\frac{1}{\Delta}(1,\Omega)$. This is a surface germ  $P \in \overline{W}$
\'etale-locally isomorphic to the germ $(0\in \C^2/\mu_\Delta)$, where a generator $\zeta\in\mu_\Delta$ acts on $\C^2$ with weights $\zeta, \zeta^\Omega$ for coprime $0<\Omega<\Delta$. Equivalently, a c.q.s.~ is a $2$-dimensional toroidal singularity, where the toric boundary divisors are 
the images of the coordinate axes in~$\C^2$.
 The exceptional divisor of the minimal resolution of $\oW$ is a chain of rational curves with self-intersections $-e_1,\ldots, -e_l$ determined by the Hirzebruch--Jung continued fraction of $\Delta/\Omega$.

It is  possible
to embed the surface germ $P\in\overline W$ into a projective surface $\overline W$ that satisfies several technical assumptions~\ref{assume}.
The~surface $\overline W$ carries a vector bundle $\bar F$ of rank~$\Delta$, see Section~\ref{s2} or  \cite[Prop.~6.7]{KKS}. 
We call $\bar F$ a {\em Kawamata vector bundle}. It is the maximal iterated extension \cite{Kaw} of the ideal sheaf $\cO_{\overline{W}}(-\bar A)$ of one of the toric boundaries of $\oW$. Concretely, we choose $\bar A$ so that its proper transform in the minimal resolution of $\overline W$ intersects the exceptional curve of self intersection $-e_1$. 

The {\em Kalck-Karmazyn algebra} $\oR$ of the algebra $R:=\C[[x,y]]^{\mu_{\Delta}}=\hat{\O}_{P,\oW}$ is a $\Delta$-dimensional associative algebra which induces an equivalence between the singularity categories of $R$ and $\oR$ \cite{KK17}. It was originally introduced by explicit generators and relations. We will use its description as 
the endomorphism algebra $$\oR=\End_{\overline W}(\bar F)$$
of the Kawamata bundle.
It is  non-commutative  unless  $\Omega=1$ or $\Delta-1$.
For example, if $\oW$ 
the cone over the rational normal curve in $\P^4$ then 
$\oR=\C[x,y,z]/(x,y,z)^2$.
In this case $\oW$ is of type $\frac{1}{4}(1,1)$. The first non-commutative Kalck--Karmazyn algebra is $\oR=\C\langle x,y\rangle/(x^2,y^3,xy,y^2x)$ for the singularity$\oW$ of type $\frac{1}{5}(1,2)$. 

Our results can be informally summarized as follows.

\begin{corollary}[\em Theorem~\ref{main}]
There is a natural map of versal deformation spaces 
$$\psi:\,\Def_{P \in \oW}\to\Def_{\oR}.$$
A general deformation of $\oW$ within a given irreducible component of $ \Def_{P \in \oW}$ induces a deformation of $\oR$ to a $\Delta$--dimensional hereditary algebra,
which is Morita--equivalent to the path algebra $\hat R$ of an acyclic quiver without relations.
\end{corollary}

Finite-dimensional hereditary  algebras are rigid,  and  each of them corresponds to a  dense open subset of an irreducible component  of $\Def_{\oR}$ \cite{Gab74}. It follows that  $\psi$ induces a map between the sets of irreducible components of 
$\Def_{P \in \oW}$ and $\Def_{\oR}$.

In order to study how the Kawamata vector bundle $\bar F$ and the Kalck--Karmazyn algebra $\oR$
deform under deformations of  $\overline W$ to a smooth surface $Y$, we use an interplay of two techniques, one coming from the study of semi-orthogonal decompositions (s.o.d.) of derived categories and another from birational geometry.
By \cite{KKS}, the surface $\oW$ admits a s.o.d. $D^b(\overline{W})=\langle \cA^{\overline{W}},\cB^{\overline{W}}\rangle$, where $\cA^{\overline{W}}\simeq D^b(\oR\hbox{\rm -mod})$.  We~show that this s.o.d. deforms to 
$D^b(Y)=\langle \cA^{Y},\cB^{Y}\rangle$,
where $\cA^{Y}\simeq D^b(\hat R\hbox{\rm -mod})$.

Incidentally, this gives a large amount of admissible embeddings of derived categories $D^b(\hat R\hbox{\rm -mod})$ of  acyclic quivers without relations into derived categories of smooth projective surfaces $Y$ (which one can choose to be rational). 
While Orlov proved \cite{Or16} that the embedding always exists if $\dim Y$ is sufficiently large, there are strong restrictions in the case of surfaces. In fact, very few examples were known before our work.
In particular, Belmans and Raedschelders \cite[Sect.4]{BR} ask whether there are bounds on the lengths of paths of realizable quivers and which acyclic quivers 
$Q_{a,b,c}$ with 3 vertices, where $a,b,c$ are the number of arrows between them, are realizable. Our results show that lengths of paths are unbounded, and we have the following partial answer for the 3 vertices quiver's question.

\begin{corollary} (Prop.~ \ref{Q4answered})
The quiver $Q_{a,b,c}$ is realizable by the algebra $\hat R$ if and only if there exists an extremal P-resolution 
(see Definition~\ref{srbsbsr})
with Wahl singularities of indices $a$ and $b$ and with $\delta=c$. See~Lemma~\ref{pun} for the list of possible $c$ for fixed values of $a,b$. 
\end{corollary}

Our second tool is birational geometry.
Let $\ocW$ be the total space of a smoothing of $\oW$ to $Y$ over a smooth curve. Special fibers of small birational models $\cW\to\ocW$
provide partial resolutions of singularities $W\to \oW$ that can be deformed to $Y$ via a $\Q$-Gorenstein deformation. For example,
irreducible components of $\Def_{P\in\oW}$
are parametrized by P-resolutions $W^+ \to \oW$ of  Koll\'ar and Shepherd-Barron \cite{KSB}.
The algebra $\hat R$ is the endomorphism algebra of a strong exceptional collection of vector bundles associated with a smoothing of another partial resolution $W^- \to \oW$, which we call the {\em N-resolution}. It is the negative analog of the  P-resolution. 

Geometric applications can be obtained by considering normal projective surfaces $W$ with $p_g(W)=q(W)=0$ which contain an N-resolution that contracts to some c.q.s. $P \in \oW$. Assume in addition that $\oW$ is unobstructed in deformations, and that $\oW \setminus \{P\}$ is simply-connected. These surfaces are abundant, and their smoothings $Y$ could be: rational surfaces, Enriques surfaces, proper elliptic surfaces, and surfaces of general type (see e.g. \cites{LP,HP,U16}). In Section~\ref{s6}, we consider applications to the following proper elliptic surfaces.  

\begin{definition}
A {\em Dolgachev surface} $D_{p,q}$ is a minimal elliptic fibration $Y \to \P^1$ 
with $H^2(Y,\cO_Y)=H_1(Y,\Z)=0$ and
with exactly two multiple fibers of coprime multiplicities $p$ and $q$.
It is simply-connected and has Kodaira dimension $1$.
\label{dolga}
\end{definition}

In \cite{CL18}, Lee and Cho construct an exceptional collection of maximum length $12$ of line bundles on Dolgachev surfaces $D_{2,3}$. Other Dolgachev surfaces cannot have exceptional collections of length $12$ even numerically by results of Vial \cite{Vial}. On the other hand, our exceptional collections never have full length because they only categorify the Milnor fiber of the smoothing.
Our results imply the following.

\begin{corollary} (Theorem \ref{I93I1})
Let $p,q \geq 2$ be coprime integers.  Dolgachev surfaces $D_{p,q}$ carry a strong exceptional collection $\bar E_{9},\ldots,\bar E_0$ associated with an N-resolution, where

\begin{enumerate}
    \item $\hat R=\End(\bar E_{9}\oplus\ldots\oplus\bar E_0)$ is the endomorphism algebra of the quiver with vertices $\bar P_0,\ldots,\bar P_9$ and with $pq-p-q$ arrows connecting each $\bar P_i$ to $\bar P_9$ for $i=0,\ldots,8$. 

    
    \item The semi-orthogonal complement of 
    $\langle\bar E_{9},\ldots,\bar E_0\rangle$ in $D^b(D_{p,q})$ has Mukai lattice $\Z^2$ with Euler pairing given by the Gram matrix $\left[\begin{matrix}
-1 & 3(pq-p-q) \cr
0  & -1\cr
\end{matrix}\right]$.
    This lattice has a full numerical exceptional collection if and only if $p=3$, $q=2$.
    \end{enumerate}
\end{corollary}

We now provide the definitions and the notation that will be used throughout this paper. We work  with schemes of finite type over $\Spec\C$. We use notation $\cW\to B$ for a flat deformation of normal surfaces  with the special fiber~$W$ and the total space~$\cW$. We also use notation $Y \rightsquigarrow W$ for a smoothing over a
smooth curve germ $0\in B$ with general fiber $Y$.
A deformation  over a smooth curve  is called $\Q$-Gorenstein if $K_\cW$ is $\Q$-Cartier. We~refer to \cite[S.~3]{H04} for a general theory of $\Q$-Gorenstein deformations. 
Quotient singularities of dimension $2$ admitting a $\Q$-Gorenstein smoothing are called T-singularities \cite{KSB}. They are either Du Val singularities or c.q.s.~of the form $\frac{1}{dn^2}(1,dna-1)$ for $0<a<n$ coprime \cite[Prop.~3.10]{KSB}. 
The special case is a {\em Wahl singularity}
$\frac{1}{n^2}(1,na-1)$.

\begin{definition}\label{sRGsrgrH}
A {\em c.q.s.~surface} $(\Gamma_1\cup\ldots\cup\Gamma_r\subset W)$ 
is a surface germ that contains a chain of smooth projective
rational curves $\Gamma_1,\ldots,\Gamma_r$ that are
toric boundary divisors at c.q.s. $P_0,\ldots,P_r$,
the surface is smooth elsewhere
(we also allow $P_i$ to be smooth points). 
We choose a toric boundary divisor germ $\Gamma_0$ at $P_0$ complementary to ~$\Gamma_1$
and 
$\Gamma_{r+1}$ at $P_r$ complementary to ~$\Gamma_r$.
A {\em c.q.s. resolution} of a c.q.s. $P \in \overline{W}$ is a c.q.s.~ surface that admits a contraction
$(\Gamma_1\cup\ldots\cup\Gamma_r\subset W)\to(P \in \overline{W})$. 
 A~c.q.s.~surface is called a {\em Wahl surface} if 
$P_0,\ldots,P_r$ are Wahl singularities.
A {\em Wahl resolution} is a c.q.s.~resolution $W\to\oW$ such that $W$ is a Wahl surface. In addition, we impose the following minimality assumption: a Wahl resolution $W$ should admit a $\Q$-Gorenstein smoothing $Y\rightsquigarrow W$ that blows down to a smoothing $Y\rightsquigarrow \oW$. 
\end{definition}

\begin{figure}[h]
\center
\includegraphics[width=0.8\textwidth]{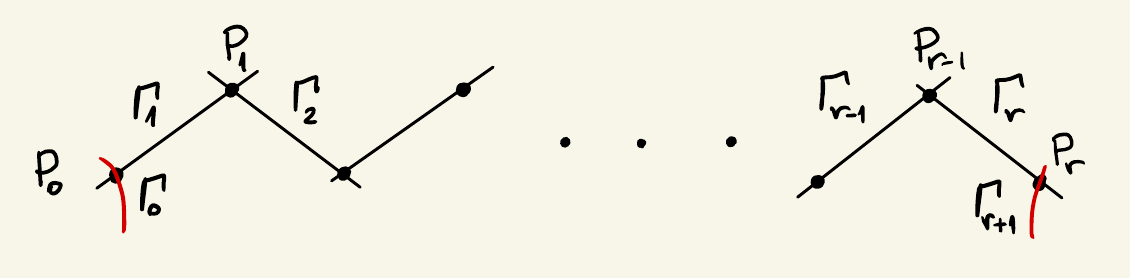}
\caption{Chain of curves in a Wahl (or c.q.s.) surface $W$}
\label{f1}
\end{figure}

Koll\'ar and Shepherd-Barron proved that irreducible components $\cC$ of the reduced versal deformation space $\Def_{P \in \overline{W}}$ 
are in a one-to-one  correspondence with {\em $P$-resolutions} \cite[Th.~ 3.9]{KSB} of $\oW$, i.e.~ c.q.s.~resolutions $W\to\oW$ with T-singularities 
and relatively ample canonical class. Concretely, a smoothing $Y \rightsquigarrow \oW$ from $\cC$ with total space $\ocW$ lifts to a smoothing $Y \rightsquigarrow W$ with total space $\cW$ given by  the relative canonical model of ~$\ocW$. We will use an analogous  one-to-one correspondence with {\em M-resolutions} 
of
Behnke and Christophersen~
\cite{BC}. 
An  M-resolution is a Wahl resolution $W^+\to\oW$ 
such that 
$K_{W^+}$ is relatively~ nef. 
The versal $\Q$-Gorenstein deformation space $\Def^{{\Q}G}_{\Gamma_1\cup\ldots\cup\Gamma_r\subset W^+}$ 
of an M-resolution
is smooth. 
Blowing down deformations \cite{Wahl} gives 
a map  $\Def^{{\Q}G}_{\Gamma_1\cup\ldots\cup\Gamma_r\subset W^+}\to \Def_{P \in \overline{W}}$, which is a Galois covering of an irreducible component $\cC$ and the Galois group is a reflection group.  
In particular, $\cC$ is also smooth and every  deformation  in $\cC$ is the blow-down of a $\Q$-Gorenstein deformation in
$\Def^{{\Q}G}_{\Gamma_1\cup\ldots\cup\Gamma_r\subset W^+}$ after a finite base change.

\begin{example}
The {\em minimal resolution} of singularities of $P \in \overline{W}$ is an example of an M-resolution. The corresponding component 
of the versal deformation space $\Def_{P \in \overline{W}}$ is called the {\em Artin component}. It parametrizes deformations of $\oW$ that admit a simultaneous resolution of singularities after a finite base change. 
\end{example}

\begin{notation} \label{wahlres}
Every Wahl surface $W$, including an M-resolution $W^+$, has the following numerical invariants $n_i,a_i,\delta_i$.
For $i=0,\ldots,r$, the Wahl singularity $P_i\in W$ has type $\frac{1}{n_i^2}(1,n_i a_i-1)$, where the Hirzebruch-Jung continued fraction of $\frac{n_i^2}{n_i a_i-1}$ goes in the direction from $\Gamma_i$ to $\Gamma_{i+1}$. For~smooth points, $n_i=a_i=1$. 
For~$i=1,\ldots,r$,
let  $ \delta_i := n_{i-1} n_{i} |K_W \cdot \Gamma_i|$ (a non-negative integer).
\end{notation} 


\begin{definition}\label{nres}
Let $W^+$ be an M-resolution of a c.q.s.~$P\in\overline{W}$ with invariants $n_i,a_i,\delta_i$ as in Notation~\ref{wahlres}. The corresponding  {\em N-resolution} $W^-$  
is a Wahl resolution   of $P\in\overline{W}$
with curves $\bar\Gamma_i$ and  singularities $\bar P_i$ of type ${1\over {\bar n}_i^2}(1,{\bar n}_i {\bar a}_i-1)$ 
such that $-K_{W^-}$ is relatively nef, i.e., $ K_{W^-} \cdot \bar\Gamma_i \leq 0$ for  $i=1,\ldots,r$, and

\begin{itemize}
    \item[(1)] The singularity $\bar P_r$ is the same as $P_0$. 
    Moreover, for every $i=1,\ldots,r$, the contraction of the chain $\bar \Gamma_{r-i+1}\cup\ldots\cup \bar \Gamma_r \subset W^-$ is the same  c.q.s as the contraction of the chain $\Gamma_1\cup\ldots\cup\Gamma_i \subset W^+$. We denote that c.q.s. by ${1\over \Delta_i}(1,\Omega_i)$.  
    \item[(2)] $\bar \delta_{r-i+1}=\delta_i$ for $i=1,\ldots,r$.
\end{itemize}
 The N-resolution associated with the M-resolution exists ans is unique (Section~\ref{s1}).\end{definition}

\begin{example}
By \cite[Ex. 3.15]{KSB}, the c.q.s.~ $\frac{1}{19}(1,7)$ admits three M-resolutions (for the notation see Section \ref{s1}), where the first M-resolution is the minimal resolution: $$(3)-(4)-(2) \ \ \ \ \ \ \ \ \ \ \ \ [{2 \choose 1}]-(1)-[{3 \choose 1}] \ \ \ \ \ \ \ \ \ \ \ \ (3)-[{2 \choose 1}]-(2).$$ 
We list the corresponding 
N-resolutions and the quivers for the path algebra $\hat R$. All this information can be computed using the program \cite{Z}.
$$[{ 8 \choose 3}]-(1)-[{ 8 \choose 3}]-(1)-[{2 \choose 1}]-(1) \ \ \ \ \ \ [{ 5 \choose 2}]-(1)-[{ 2 \choose 1}] \ \ \ \ \ \ [{ 8 \choose 3}]-(1)-[{5 \choose 2}]-(1)$$  
\begin{figure}[h]
\center
\includegraphics[width=10cm]{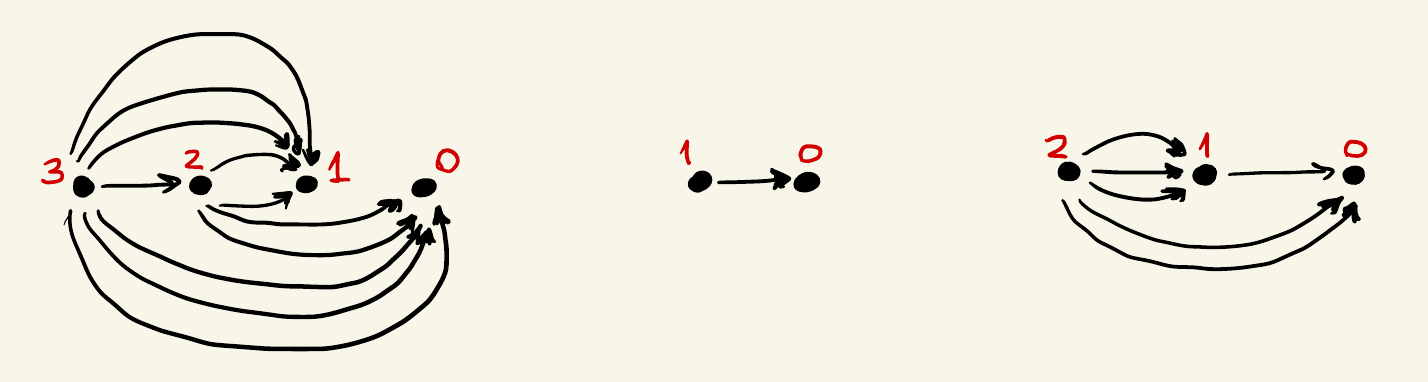}
\label{quivers}
\end{figure}
\label{quiver}
\end{example}

\begin{Assumption}\label{assume}
For technical reasons, we need to compactify all surfaces, which requires imposing the following assumptions throughout the paper.

\noindent
\begin{enumerate}
\item $W$ is a normal, projective c.q.s.~ surface 
smooth outside of $\{P_0,\ldots,P_r\}$.\break 
If $\Gamma_1 \cup \ldots \cup \Gamma_r\subset W$ is a c.q.s.~ resolution of $P\in\overline W$, then the surface $\overline{W}$ is determined as the contraction of the chain $\Gamma_1,\ldots,\Gamma_r$ to the point $P \in \overline{W}$.

\item $H^1(W,\O_W)=H^2(W,\O_W)=0$. Since 
c.q.s.~are rational singularities, if $W$ is a c.q.s.~resolution of $\oW$ then (2) is equivalent to the same vanishing on $\oW$ or, equivalently,  on the minimal resolution of $\oW$. For example, $\oW$ can be a rational surface. Since rational singularities are Du Bois, (2)~is~equivalent to the same vanishing on any projective deformation $Y$ of~$W$.

\item 
There is a Weil divisor $\bar A$ on $\oW$ 
that generates the local class group $\Cl(P \in \oW)$. By Lemma~\ref{goodA}, we can  choose effective smooth divisors $\bar A, \tilde{\bar A}\subset \oW$ such that 
the germ $P\in (\bar A\cup\tilde{\bar A})\subset\oW$ is et\'ale-locally isomorphic to
$0\in(x=0)\cup(y=0)\subset{1\over\Delta}(1,\Omega)$.
Proper transforms $\Gamma_0$ (resp.~$\Gamma_{r+1}$) of $\bar A$ 
(resp.~$\tilde{\bar A}$) in a c.q.s. resolution $W$ of $\oW$
intersect the chain $\Gamma_1\cup\ldots\cup\Gamma_r$
only at  $P_0$ (resp.~$P_r$),
where they are equivalent to  toric boundaries
opposite to $\Gamma_1$ (resp.~$\Gamma_r$) as in Figure~\ref{f1}.

\item  $H^2(\overline{W},T_{\overline{W}})=0$.
By Lemma~\ref{NO}, there are
no local-to-global obstructions to $\Q$-Gorenstein deformations of a Wahl resolution $W$ of $\oW$ or the pair $(W,\Delta)$ where $\Delta=\Gamma_0+\Gamma_1+\ldots+\Gamma_r+\Gamma_{r+1}$ if (2) and (3) also hold.
A~general example satisfying (4) is any surface $\overline{W}$ such that $-K_{\overline{W}}$ is big \cite[Prop. 3.1]{HP}. 

\end{enumerate}
\end{Assumption}


\begin{definition}
Let $Y \rightsquigarrow W$ be a projective $\Q$-Gorenstein smoothing of a Wahl surface $W$ over a smooth curve germ $0\in B$ satisfying Assumption~\ref{assume}. 
By~\cite{H13}, 
for each $P_i\in W$ and after shrinking $B$, we have an associated exceptional vector bundle $E_i$ of rank $n_i$ on $Y$, which we call a {\em Hacking vector bundle}. 
Bundles $E_r,\ldots, E_0$ 
form a {\em Hacking exceptional collection} 
on $Y$, see Section~\ref{AEFefEEfe}. 
\end{definition}

Our main theorem is the following (see Section~\ref{s4} for more detailed results).

\begin{theorem}\label{weFwetwET}\label{main}
Let $W^+$ be an M-resolution of  $P\in\overline W$
satisfying Assumption~\ref{assume}. Fix a projective $\Q$-Gorenstein smoothing $Y \rightsquigarrow W^+$ which is sufficiently general in its irreducible component of the versal deformation space of $\oW$. 
This component also contains a $\Q$-Gorenstein smoothing $Y \rightsquigarrow W^-$, where $W^-$ 
is the N-resolution associated to~$W^+$. 
\begin{enumerate}
    \item 
Let 
$\bar E_r,\ldots, \bar E_0$
be a  Hacking exceptional collection  on $Y$ associated with the\break N-resolution~ $W^-$. This  collection is strong: $\Ext^k(\bar E_i, \bar E_j)=0$ for $k>0$ and $i>j$.
\item
In contrast, let $E_r,\ldots, E_0$ be a Hacking exceptional collection on $Y$ associated with the M-resolution~$W^+$. Then we have  
$\Ext^k(E_i, E_j)=0$ for $k\ne1$ and $i>j$. \item
For $i=1,\ldots, r$, we have
$\Hom(\bar E_{r+1-i}, \bar E_{r-i})\simeq
\Ext^1(E_{i}, E_{i-1})^\vee\simeq
\C^{\delta_i}$.
\item
The Kawamata  bundle $\bar F$  on $\oW$ deforms  to a vector bundle $F\simeq \bigoplus\limits_{i=0}^r \bar E_i^{n_{r-i}}$ on~$Y$. Since $F$ has rank $\Delta$, we note that  $\Delta=n_0 \bar n_r + n_1 \bar n_{r-1} +\ldots + n_r \bar n_0$.
\item
The Kalck--Karmazyn algebra $\bar R=\End(\bar F)$ deforms to the algebra $\End(F)$, which is hereditary and
Morita-equivalent to the path algebra
$\hat R=\End(\bar E_r\oplus\ldots\oplus\bar E_0)$. 
\end{enumerate}
\end{theorem}


\begin{remark} 
The algebra $\hat R$ is a path algebra of a quiver with vertices $\bar E_r,\ldots,\bar E_0$ and with arrows connecting $\bar E_i$ to $\bar E_j$ for $i>j$ so that the total number of paths connecting $\bar E_i$ to $\bar E_j$ is equal to  (see Lemma~\ref{sdtheth}) \begin{equation}\label{afbzdfbdfna}\hom(\bar E_i,\bar E_j)=\bar n_j \bar a_i - \bar n_i \bar a_j=  \bar n_i \bar n_j \Big( \frac{ \bar \delta_{j+1}}{\bar n_j \bar n_{j+1}} + \dots +\frac{ \bar \delta_i}{\bar n_{i-1} \bar n_i} \Big). \end{equation} 

It follows from this formula that the quiver is  connected unless the algebra is semisimple, in which case the N-resolution is the M-resolution of a T-singularity. 
\end{remark}

\begin{remark}
One of the ingredients of the proof 
is a result of  Kawamata \cite{K21}, where 
Theorem~\ref{weFwetwET} was proved in the case  when
$Y \rightsquigarrow \overline W$
is a $\Q$-Gorenstein smoothing of a cyclic T-singularity ${1\over dn^2}(1,dna-1)$ (see also \cite{C20}). In this case 
all $\delta_i=0$,  the M-resolution is equal to the N-resolution,
Hacking vector bundles are pairwise orthogonal, and 
$\End(F)$ is a direct sum of $d$ copies of 
$\Mat_n(\C)$.
\end{remark}

In Section~\ref{s5} we illustrate these results in the simplest case of the Artin component, when $W^+\to \oW$ is a minimal resolution of singularities. In this case only, the exceptional collection $E_r,\ldots,E_0$ on $Y$ is a collection of line bundles, which is a deformation of an exceptional collection of line bundles on $W^+$. However, while the latter has both $\Hom$ and $\Ext^1$ in the forward direction, the former has $\Ext^1$ only. As in the general case, the dual  collection $\bar E_r,\ldots,\bar E_0$ 
associated with the N-resolution $W^-\to\oW$ is a  strong exceptional collection of vector bundles.

\begin{figure}[htbp]
\centering
\includegraphics[width=9cm]{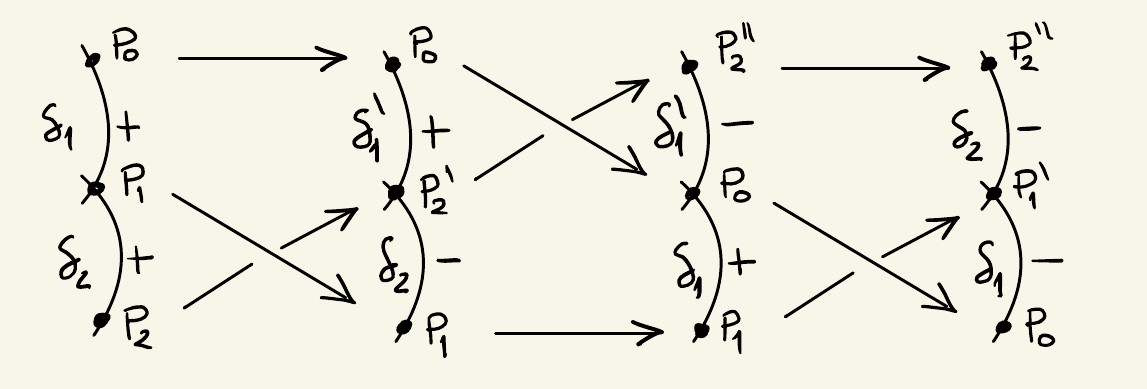}
\centering
\caption{Braiding from $W^+$ to $W^-$ for three Wahl singularities}
\label{f0}
\end{figure}

\begin{remark}
Antiflips of \cite{HTU17} are the
main  tool in the proof of Theorem~\ref{main}.
We use them to produce  $\Q$-Gorenstein smoothings $Y\rightsquigarrow W$ of different Wahl resolutions~$W\to\overline{W}$ over different curve germs in a given irreducible component of the versal deformation space of $\oW$, see Definition~\ref{sGshsRHWRH}.
Antiflips generate a ``geometric'' braid group action on the infinite set of all Wahl resolutions $W\to\oW$ compatible with the ``categorical'' braid group action by mutations of exceptional collections on $Y$, see Theorem~\ref{braidrelations}.
After applying finitely many  antiflips to the
$\Q$-Gorenstein smoothing $Y \rightsquigarrow W^+$ of the $M$-resolution, we obtain a $\Q$-Gorenstein smoothing $Y \rightsquigarrow W^-$ of 
the N-resolution associated to $W^+$. 
Going back from  $W^-$ to  $W^+$ can be done over the same curve $B$. This decomposes a birational map of total spaces of deformations $\cW^-\dasharrow\cW^+$
into a sequence of $r(r+1)\over2$ flips and flops.
\end{remark}

We will also describe the derived category of the total space~$\cW$ of a deformation.
We recall that a {\em semi-orthogonal decomposition} (s.o.d)  
of a triangulated category $\cT$ is a sequence of full triangulated subcategories
$\langle\cA_0,\ldots,\cA_r\rangle$ satisfying two conditions: 
(1)  $\Hom(\cA_j,\cA_i)=0$ for $j>i$, and 
(2) for every object $T\in\cT$,
there exist morphisms $0=T_r\to\ldots\to T_0=T$ such that the cone $A_i$ of $T_{i}\to T_{i-1}$ belongs to~$\cA_i$. 
The objects $A_i$ are functorial in $T$, i.e. we have projection functors $\cT\to\cA_i$.

\begin{theorem}\label{mainKawa}
Let $Y \rightsquigarrow W$ be a projective $\Q$-Gorenstein smoothing of a Wahl surface~$W$ satisfying Assumption \ref{assume} (1), (2), (3).
After possibly shrinking $B$, $D^b(\cW)$ admits a $B$-linear\footnote{I.e.~ preserved by tensoring with a pullback of any object $T\in D^{\perf}(B)$.} s.o.d.~$\langle  \cA^{\cW}_r,\ldots,\cA^{\cW}_0,\cB^{\cW} \rangle$
compatible with respect to restrictions to $W$ and $Y$
$$\begin{CD}
\langle \cA^W_r,\ldots,\cA^W_0,\cB^W \rangle  @<Li_W^*<<  \langle  \cA^{\cW}_r,\ldots,\cA^{\cW}_0,\cB^{\cW} \rangle @>Li_Y^*>> 
\langle \cA^Y_r,\ldots,\cA^Y_0, \cB^Y \rangle.\\
\end{CD}
$$
Each
$\cA^Y_i$ is  generated by the Hacking bundle $E_i$ and 
and each
$\cA^W_i\simeq D^b(R_i\hbox{\rm -mod})$, where 
$R_i$ is the Kalck-Karmazyn algebra associated to $P_i \in W$. Furthermore, 
$\cB^{\cW}\subset D^{\perf}(\cW)$.
\end{theorem}

See Section~\ref{AEFefEEfe} for results about deformations of a  c.q.s.~surface over any smooth base~$B$, for example for the whole versal $\Q$-Gorenstein deformation of ~$W$.
The  categories $\cA^W_i$, $\cA^Y_i$ and $\cA^{\cW}_i$  categorify Wahl singularities 
$P_i\in W$, their~Milnor fibers in~$Y$ and terminal singularities $P_i\in\cW$, respectively. 
The  categories $\cB^W$ and $\cB^Y$  categorify the complement of the chain  $\Gamma_1\cup\ldots\cup\Gamma_r\subset W$
and the complement of its Milnor fiber in $Y$, 
which are topologically equivalent.







\setcounter{tocdepth}{1}
\subsection*{Acknowledgements}

We are grateful to Paul Hacking and Alexander Kuznetsov for useful comments and discussions and especially to Martin Kalck for a suggestion to apply our methods to open questions raised in  \cite{BR} (see Prop.~\ref{Q4answered} and Example~~\ref{Q5answered}). We thank Juan Pablo Z\'u\~niga for writing the computer program \cite{Z} that finds all M- and N-resolutions. The first author was supported by the NSF grant DMS-2101726 and Simons Fellowship. The second author was supported by the FONDECYT regular grant 1190066.

\tableofcontents

\section{N-resolution associated with an M-resolution} \label{s1}

Let $0 < \Omega < \Delta$ be coprime integers, and let $P \in \overline{W}$ be a c.q.s. of type ${1\over \Delta}(1,\Omega)$. We~fix an M-resolution $W^+ \to \overline{W}$. We will construct its N-resolution $W^- \to \overline{W}$  in Lemma~\ref{existencenres}
and prove its 
uniqueness in Corollary~\ref{uniquenres}.
In this section we do not consider deformations of surfaces or derived categories. 
We write
$$\frac{\Delta}{\Omega}=[e_1, \ldots ,e_{\ell}]=e_1-\cfrac{1}{\ddots-\cfrac{1}{e_l}}.$$
If $\frac{\Delta}{\Omega} = [e_1, \ldots ,e_{\ell}]$ and $\frac{\Delta}{\Delta-\Omega} = [b_1,\ldots,b_s]$, then the Hirzebruch--Jung continued fraction $[b_s,\ldots,b_1,1,e_1,\ldots,e_{\ell}]$ is equal to $0$
(and in particular is well-defined, i.e.~there is no division by $0$).
This follows from the Riemenschneider's diagram \cite{Ri} and can be interpreted as the consecutive contraction of $(-1)$-curves in the chain of $\P^1$'s of self-intersections $-b_s,\ldots,-b_1,-1,-e_1,\ldots,-e_{\ell}$.
The contraction process terminates with a single  $\P^1$ of self-intersection~$0$. Since $\frac{\Delta}{\Omega'} = [e_{\ell}, \ldots ,e_{1}]$ implies $\Omega' \Omega \equiv 1(\text{mod} \Delta)$, then we also have $[e_1,\ldots,e_{\ell},1,b_s,\ldots,b_1]=0$. 
 
\begin{notation}
Various operations with Hirzebruch--Jung continued fractions will include the notation $[{n \choose a}]$, which abbreviates the Hirzebruch--Jung continued fraction of  $\frac{n^2}{na-1}$
of a Wahl singularity. For~example, for our M-resolution we have
\begin{equation}\label{argwgwrgwR}
[b_s,\ldots,b_1]-(1)-[{n_0 \choose a_0}]-(c_1)-[{n_1 \choose  a_1}]-(c_2)-\ldots-(c_r)-[{ n_r \choose a_r}]=0.
\end{equation}
Here $[{n_i \choose a_i}]$ represents the Wahl singularity $P_i$ and  $(c_i)$ represents the curve~$\Gamma_i$, so that its proper transform in the minimal resolution
of $W^+$ has self-intersection $-c_i$. 
\end{notation}

By \cite{Ch, Stevens}, there is a 
bijection between  P-resolutions of 
$P\in\overline{W}$ and the following set of zero continued fractions: $K(\Delta / \Omega)=\{[k_1,\ldots,k_s]=0\ \colon 1 \leq k_i \leq b_i\}$.  
The  M-resolution \cite{BC} is constructed by resolving all Du Val singularities
of the P-resolution
(they are of type $A_m$ for some $m$'s), and partially resolving each T-singularity $\frac{1}{dn^2}(1,dn a -1)$ with $d>1$ by its crepant Wahl resolution, which has $d-1$ rational curves and $d$ Wahl singularities of type $\frac{1}{n^2}(1,na-1)$.

\begin{notation}
Suppose the M-resolution $W^+ \to \overline{W}$ corresponds to a zero-fraction $[k_1,\ldots,k_s] \in K(\Delta / \Omega)$. 
As in Notation \ref{wahlres}, the surface $W^+$ contains curves $\Gamma_1, \allowbreak  \ldots, \allowbreak  \Gamma_r$ and Wahl singularities at $P_i$ of type $\frac{1}{n_i^2}(1,n_i a_i -1)$. We have $$\delta_i =n_{i-1} n_i K_{W^+} \cdot \Gamma_i\geq 0$$ for all $i=1,\ldots,r$. Let $d_i:=b_i-k_i \geq 0$. Furthermore, $d_1+\ldots+d_s=r+1$. Let~ $d_{i_1},\ldots,d_{i_e}$ be the set of nonzero $d_i$ with $i_1<i_2<\ldots<i_e$.
\end{notation}

%

\begin{proposition}
 $\delta_1,\ldots,\delta_r$ can be computed as follows: for $k=1,\ldots,e-1$,
$$\frac{\delta_{d_{i_1}+\ldots+d_{i_{k}}}}{ \epsilon_{d_{i_1}+\ldots+d_{i_{k}}}} =[b_{i_{k}+1},\ldots,b_{i_{k+1}-1}]$$ if $i_{k+1}>i_{k}+1$, or $\delta_{d_{i_1}+\ldots+d_{i_{k}}}=1$ if $i_{k+1}=i_{k}+1$.
All other $\delta_i$ are equal to $0$.
\label{hotdog}
\end{proposition}

\begin{proof}
This is in the algorithm \cite[Cor.~10.1]{PPSU} for a P-resolution, adapted to its M-resolution. See also \cite[Prop.~4.1]{HTU17}. 
\end{proof}

\begin{example}
Consider $\frac{\Delta}{\Omega}=\frac{89}{33}=[3,4,2,2,4]$. Take the P-resolution $W^+$ given by $$[{2 \choose 1}]-(1)-[{3 \choose 1}]-(2)-[{2 \choose 1}]=[4]-(1)-[5,2]-(2)-[4]\quad \to \quad [3,4,2,2,4].$$ We have $\frac{\Delta}{\Delta-\Omega}=[2,3,2,5,2,2]$.
The element in $K(\Delta/\Omega)$ corresponding to $W^+$ is $[2,2,1,5,1,2]=0$. Thus $d_1=0$, $d_2=1$, $d_3=1$, $d_4=0$, $d_5=1$, and $d_6=0$. Therefore, by Proposition \ref{hotdog}, we have $\delta_1=1$ and $\delta_2=5$.
\label{exampleI}
\end{example}


\begin{lemma}
Let $0<B<A$ be integers with gcd$(A,B)=1$. Let $$\frac{A}{A-B}=[x_1,\dotsc,x_p] \ \ \text{and}  \ \ \frac{A}{B}=[y_1, \dotsc, y_q].$$ Then

\begin{enumerate}
    \item $\frac{A^2}{AB-1} = [y_1, \dotsc, y_{q-1}, y_q+x_p, x_{p-1}, \dotsc, x_1]$,
    \item $\frac{A^2}{A^2-(AB-1)}=[x_1,\ldots,x_p,2,y_q,\ldots,y_1]$, and
    \item $[x_p,\dotsc,x_1]-(1)-[y_1, \dotsc, y_{q-1}, y_q+x_p, x_{p-1}, \dotsc, x_1]$ contracts to $[x_p,\dotsc,x_1]$.
\end{enumerate}
\label{merken}
\end{lemma}

\begin{proof}
See \cite[Lem.8.5]{HP} or \cite[Cor.2.1 and 2.2]{PSU}.
\end{proof}

We use the geometric procedure in \cite[Cor.10.1]{PPSU}, which  interprets the zero continued fraction of the Wahl resolution \eqref{argwgwrgwR} 
as follows: 

\begin{itemize}
    \item[(1)] 
     At the beginning of \eqref{argwgwrgwR}
     we have $d_{i_1}$ Wahl chains $[{n_0 \choose a_0}]$ as follows: 
    $$[b_s,\ldots,b_1]-(1)-\underbrace{[{n_0 \choose a_0}]-(1)-[{n_0 \choose  a_0}]-(1)-\ldots-(1)-[{ n_0 \choose a_0}]}_{d_{i_1}}-(c_{d_{i_1}})-\ldots.$$
    We can blow-down the $(-1)$-curves and new $(-1)$-curves consecutively until we obtain the new chain 
    $$[b_{s},\ldots, b_{i_1+1},b_{i_1}-d_{i_1},b_{i_1-1},\ldots,b_1]-(c_{d_{i_1}})-[{ n_{d_{i_1}} \choose a_{d_{i_1}}}]-\ldots-(c_r)-[{ n_r \choose a_r}].$$ 
    \item[(2)] If $b_{i_1}-d_{i_1}=1$, then we contract this $(-1)$-curve and all new $(-1)$-curves in the subchain $[b_{s},\ldots, b_{i_1+1},b_{i_1}-d_{i_1},b_{i_1-1},\ldots,b_1]$ until there are none. 
    \item[(3)] Then the original $(-c_{d_{i_1}})$-curve becomes a $(-1)$-curve, and we have $$ \frac{n_{d_{i_1}}}{n_{d_{i_1}} - a_{d_{i_1}}}=[b_{1},\ldots, b_{i_1-1},b_{i_1}-d_{i_1},b_{i_1+1},\ldots,b_{i_2-1}].$$
    \item[(4)] We now repeat starting in (1) with the $d_{i_2}$. 
    \item[(5)] We end with $[\ldots,b_{i_e}-d_{i_e},\ldots,b_{i_1}-d_{i_1},\ldots]=0$, which is the zero continued fraction corresponding to the M-resolution.
\end{itemize}

\begin{proposition}
For every $j=0,\ldots,r$, we have that $$[{n_0 \choose a_0}]-(c_1)-[{n_1 \choose  a_1}]-(c_2)-\ldots-(c_j)-[{ n_j \choose a_j}]-(1)- \ \ \ \ \ \ \ \ \ \ \ \ \ \ \ \ \ \ \ \ \ \ \ \ \ \ \ \ \ \ \ \ \ \ \ \ \ \ \ $$ $$ \ \ \ \ \ \ \ \ \ \ \ \ \ \ \ \ \ \ \ \ \ \ \ \ \ \ \ \ \ \ \ \ \ \ \ \ \ \ [\ldots,b_{i_e}-d_{i_e},\ldots,b_{i_p}-d_{i_p}, \ldots,b_{i_{p-1}}-k, \ldots,b_{i_{p-2}}, \ldots b_{i_1},\ldots]=0,$$ for some $p$ and some $k$ depending on $j$.
\label{partialCQS}
\end{proposition}

\begin{proof}
As explained above, the blowing-down process of the M-resolution produces the zero continued fraction $[\ldots,b_{i_e}-d_{i_e}, \ldots,b_{i_1}-d_{i_1},\ldots]$ 
by subtracting $1$ from the $b_{i_p}$ for each of the Wahl singularities from $n_0,a_0$ to $n_r,a_r$. Let us consider $$
[{n_0 \choose a_0}]-(c_1)-[{n_1 \choose  a_1}]-(c_2)-\ldots-(c_r)-[{ n_r \choose a_r}]-(1)-[\ldots,b_{i_e},\ldots,b_{i_1},\ldots]=0.$$ Then we do the same process of subtracting $1$ but now we start with $b_{i_e}$ and we finish with $b_{i_1}$. This proves the claim via stopping at the Wahl singularity $[{ n_j \choose a_j}]$ during this process.  
\end{proof}

\begin{lemma}\label{existencenres}
An N-resolution 
$W^- \to \overline{W}$ 
associated to the M-resolution $W^+ \to \overline{W}$
can be constructed as follows.
It has Wahl singularities $\bar P_i$ of type $\frac{1}{\bar n_i^2}(1,\bar n_i \bar a_i-1)$ for $i=0,\ldots,r$, which we will describe from the bottom up via $\tilde n_{p}, \tilde a_{p}$ 
such that 
$\bar n_{r-i}=\tilde n_i$, $\bar a_{r-i}=\tilde a_i$ for $i=0,\ldots,r$. 
The algorithm is as follows. 
\begin{itemize}
    \item If $i_1=1$ (i.e. $d_1 \neq 0$), then $\tilde n_p= \tilde a_p=1$ for $p=0,\ldots,d_1-1$. In other words, we start with $d_1$ smooth points.

\item If $i_1>1$, then $\frac{\tilde n_p}{\tilde n_p - \tilde a_p} =[b_1,\ldots,b_{i_1-1}]$ for $p=0,\ldots, d_{i_1}-1$. 

\item Let $q= \sum\limits_{j=1}^k d_{i_j}$. Then $\frac{\tilde n_p}{\tilde n_p - \tilde a_p} =[b_1,\ldots,b_{i_{k+1}-1}]$ for $p=q,\ldots,q+d_{i_{k+1}}-1$. 
\end{itemize}

The curves $\bar \Gamma_i$ for $i=1,\ldots,r$ are 
as follows. If $\bar \Gamma_i$ passes through one or two Wahl singularities, then its proper transform in the minimal resolution is a $(-1)$-curve. Otherwise (i.e. no Wahl singularities) it is a $(-2)$-curve. 
\end{lemma}

\begin{example}
In Example \ref{exampleI}, 
$\bar n_0=35$, $\bar n_1=5$, $\bar n_2=2$. The N-resolution is
$$[{ 35 \choose 13}]-(1)-[{5 \choose 2}]-(1)-[{2 \choose 1}]=[3,4,2,2,7,2,3,2]-1-[3,5,2]-1-[4].$$
\label{exampleII}
\end{example}

Before proving Lemma \ref{existencenres} in general, let us do it just for extremal M-resolutions. 

\begin{definition}\label{srbsbsr}
A Wahl resolution $W \to \overline{W}$
of $P \in \overline{W}$ is called {\em extremal} if the exceptional divisor consists of a single curve $\Gamma_1$. We have two Wahl singularities $P_0, P_1$ (which may be smooth points). The type of $P_i$ is ${1\over n_i^2}(1,n_i a_i-1)$ and we have \begin{equation}\label{expres} 
\delta_1=n_0 n_1 \, |K_{W} \cdot \Gamma_1| \ \ \ \ \text{and} \ \ \ \ -n_0^2 n_1^2 \, \Gamma_1^2=  \Delta=n_0^2 + n_1^2 \pm \delta_1 n_0 n_1,\end{equation} where $\pm$ is the sign of $K_W \cdot \Gamma_1$. If $\delta_1=0$, then we have the M-resolution of $\frac{\Delta}{\Omega}=\frac{2n^2}{2na-1}$ for some $0<a<n$ coprime \cite{BC}. 
If $W$ is an extremal $M$-resolution with $\delta_1 >0$, then $W$ is an extremal P-resolution introduced and studied in \cite{HTU17}. 
\end{definition}

\begin{lemma}\label{extreNres}
An extremal M-resolution has a unique  N-resolution.
\end{lemma}

\begin{proof}
If $\delta_1=0$, then we have the M-resolution of $\frac{\Delta}{\Omega}=\frac{2n^2}{2na-1}$ for some $0<a<n$ coprime \cite{BC}. Here N-resolution and M-resolution coincide, and there is only one index $i_1$ and $d_{i_1}=2$ (as at the end of Lemma \ref{merken} with $D=2$). If $W^+$ is an extremal $M$-resolution with $\delta_1 >0$, then we have an extremal P-resolution of \cite{HTU17}. Here we have only two indices $i_1,i_2$. We have $d_{i_1}=d_{i_2}=1$, and $$ [b_1,\ldots,b_{i_1}-1, \ldots,b_{i_2}-1 \ldots,b_s]=0.$$ 

We now prove that the N-resolution proposed in Lemma \ref{existencenres} is indeed an N-resolution. By Lemma \ref{merken} (3), $[b_s,\ldots,b_1]-(1)-[{\bar n_1 \choose \bar a_1}]-(1)-[{n_0 \choose a_0}]$ can be blown-down to $[b_s,\ldots,b_{i_2}-1,b_{i_2 -1}, \ldots, b_1]-(1)-[{n_0 \choose a_0}]$ and that can be blown-down to $ [b_1,\ldots,b_{i_1}-1, \ldots,b_{i_2}-1 \ldots,b_s]$, which is zero, and so $\frac{\Delta}{\Omega}=[{\bar n_1 \choose \bar a_1}]-(1)-[{n_0 \choose a_0}]$. Hence we do get a Wahl resolution $W \to \oW$ in this way. 

We now check that $K_{W} \cdot \Gamma <0$, where $\Gamma$ is the central curve, and $\bar \delta_1 = \delta_1$. Let~ $\frac{p_k}{q_k}=[b_1,\ldots,b_{k-1}]$, $p_1=1$, $p_0=q_1=0$, and $q_0=-1$. Then $$\frac{p_{i_1} q_{i_2} - p_{i_2} q_{i_1}}{p_{i_1} q_{i_2-1} -p_{i_2 -1} q_{i_1}}=[b_{i_2-1},\ldots,b_{i_1+1}]=\frac{\delta_1}{\epsilon'_1}$$ by \cite[Lemma 4.2]{HTU17}. But, by definition, we have $p_{i_1}=n_0$, $q_{i_1}=n_0-a_0$, $p_{i_2}=\bar n_1$, and $q_{i_2}=\bar n_1 - \bar a_1$. Therefore $\delta_1 = \bar n_1 a_0 - n_0 \bar a_1$. On the other hand, a toric computation shows that $K_{W} \cdot \Gamma= -1 +\big(1-\frac{\bar n_1 - \bar a_1}{\bar n_1} \big)+ \big(1-\frac{a_0}{n_0} \big)=-\frac{\bar n_1 a_0 - n_0 \bar a_1}{\bar{n}_1 n_0}$, and so  $K_{W} \cdot \Gamma$ is negative and $\bar \delta_1 = \delta_1$. 

Finally, for uniqueness let us consider some Wahl chain $[{\tilde n_1 \choose \tilde a_1}]$ such that $$[b_s,\ldots,b_1]-(1)-[{\tilde n_1 \choose \tilde a_1}]-(1)-[{n_0 \choose a_0}]=0,$$ but then we also have $[{\tilde n_1 \choose \tilde a_1}]-(1)-[{n_0 \choose a_0}]-(1)-[b_s,\ldots,b_1]=0$, and so $[{\tilde n_1 \choose \tilde a_1}]$ is determined, being dual to the contraction of $[{n_0 \choose a_0}]-(1)-[b_s,\ldots,b_1]$.
\end{proof}

\begin{proof} [Proof of Lemma \ref{existencenres}]
Note that the $\bar \delta_i$ are determined by the Wahl resolution, but in the proof we will check that $\bar \delta_i=\delta_{r+1-i}$ as required by the definition of the N-resolution. We will also need to prove property (1) in the definition.

The strategy is to consider the chain
\begin{equation}\label{adfadrh}
[b_s,\ldots,b_1]-(1)-[{\bar n_0 \choose \bar a_0}]-(1)-[{\bar n_1 \choose \bar a_1}]-(1)-\ldots-(1)-[{\bar n_r \choose \bar a_r}],
\end{equation} 
and to prove that it is contractible and contracts to zero. In this way, we would have that 
$[{\bar n_0 \choose \bar a_0}]-(1)-[{\bar n_1 \choose \bar a_1}]-(1)-\ldots-(1)-[{\bar n_r \choose \bar a_r}]=[e_1,\ldots,e_{\ell}]=\frac{\Delta}{\Omega}$. 
Notice that $[{\bar n_r \choose \bar a_r}]$ could correspond to $[2,\ldots,2]$ if $d_1\ne0$. Let us consider  $d_{i_1},\ldots,d_{i_e}$ (the set of nonzero $d_i$)  and write $ [b_s,\ldots,b_1]=[\ldots,b_{i_e},\ldots,b_{i_{e-1}}, \ldots,b_{i_1},\ldots]$ with $i_1<i_2<\ldots<i_e$.
By definition, we have either $\frac{\bar n_0}{\bar n_0 - \bar a_0}=[b_1,\ldots,b_{i_e -1}]$ when $i_e>1$, or $\bar P_0$ is a smooth point. In any case, by Lemma \ref{merken} (3), we have 
$$[\ldots,b_{i_e},\ldots,b_{i_{e-1}}, \ldots,b_{i_1},\ldots]-(1)-[{\bar n_0 \choose \bar a_0}]=[\ldots,b_{i_e}-1,\ldots,b_{i_{e-1}}, \ldots,b_{i_1},\ldots].$$ We continue with the following singularities, applying Lemma \ref{merken} (3) each time since $\frac{\bar n_p}{\bar n_p - \bar a_p}=[b_1,\ldots,b_{i_k -1}]$ for some $k$ depending on $p$. We recall that $i_e>i_{e-1}>\ldots>i_1$, and so this contraction process makes sense. In this way, we arrive to $$[\ldots,b_{i_e}-d_{i_e},\ldots,b_{i_{e-1}}-d_{i_{e-1}}, \ldots,b_{i_1}-d_{i_1},\ldots],$$ and this is the zero continued fraction of the M-resolution. Therefore the original chain is contractible, and it contracts to zero. 

We now show that the $\bar \delta_i$'s are indeed the ones from the algorithm.  
First we know that there are no problems with $\bar \delta_i=0$, they obviously coincide. Let us consider $\bar \delta_p$ corresponding to the break $i_k < i_{k+1}$. In the blowing-down process of (\ref{adfadrh}), consider the step
$$[\ldots,b_{i_{k+1}}-d_{i_{k+1}}+1,\ldots,b_{i_k},\ldots,b_1]-(1)-[{\bar n_{p-1} \choose \bar a_{p-1}}]-(1)-[{\bar n_p \choose \bar a_p}]-(1)-\ldots.$$ 

This is about computing $\bar \delta_p$ for $[{\bar n_{p-1} \choose \bar a_{p-1}}]-(1)-[{\bar n_p \choose \bar a_p}]=\frac{\Delta'}{\Omega'}$ for some $\Delta',\Omega'$. The~strategy is to compute $\bar \delta_p$ in this situation, and show that it coincides with our definition.
For that we compute the dual continued fraction $\frac{\Delta'}{\Delta'-\Omega'}$. But then 
$$[b_{i_{k+1}}-d_{i_{k+1}}+1,\ldots,b_{i_k},\ldots,b_1]-(1)- \frac{\Delta'}{\Omega'}$$ is either $0$ or we can complete it as $$[y_u\ldots,y_1,b_{i_{k+1}}-d_{i_{k+1}}+1,\ldots,b_{i_k},\ldots,b_1]-(1)- \frac{\Delta'}{\Omega'}$$ to make it zero, where $y_j\geq 2$, and $b_{i_{k+1}}-d_{i_{k+1}}+1 \geq 2$. Hence $$\frac{\Delta'}{\Delta'-\Omega'}=[b_1,\ldots,b_{i_k},\ldots,b_{i_{k+1}}-d_{i_{k+1}}+1,y_1,\ldots,y_u].$$ But this is the situation of an extremal N-resolution where we must subtract $-1$ in positions $i_k$ and $i_{k+1}$ to make it zero. Then the computation of $\bar \delta_p$ is identical to what we did in Lemma \ref{extreNres}.

To prove property (1) in Definition \ref{nres}, we use Proposition \ref{partialCQS}. Note that when we eliminate Wahl singularities in the N-resolution, we subtract $1$ from the $b_{i_p}$ from $i_e$ to $i_1$, and that is exactly what we have in Proposition \ref{partialCQS}.
\end{proof}


\begin{corollary}
Every M-resolution has a unique associated N-resolution.
\label{uniquenres}
\end{corollary}

\begin{proof}
We know this is true for $r=1$ by Lemma \ref{extreNres}. For $r\geq 2$ we go by induction on $r$. We have that $\Gamma_1 \cup \ldots \cup \Gamma_{r-1}$ is an M-resolution of $\frac{1}{\Delta_{r-1}}(1,\Omega_{r-1})$, and so we can apply induction for all singularities, deltas, and $\frac{1}{\Delta_{i}}(1,\Omega_{i})$ except for $\bar n_0, \bar a_0$. Let $\frac{\Delta_{r-1}}{\Omega_{r-1}}=[f_1,\ldots,f_t]$. Then we have 
$[b_s,\ldots,b_1]-(1)-[{n'_0 \choose a'_0}]-(1)-[f_1,\ldots,f_t]=0$,  and this implies $[{n'_0 \choose a'_0}]-(1)-[f_1,\ldots,f_t]-(1)-[b_s,\ldots,b_1]=0$, and so $[{n'_0 \choose a'_0}]$ is determined by $\frac{1}{\Delta_{r-1}}(1,\Omega_{r-1})$ and $\frac{1}{\Delta}(1,\Omega)$. 
\end{proof}

\begin{example}
Using the computer program \cite{Z}, we find all M-resolutions and N-resolutions of the c.q.s.~$\frac{1}{85}(1,49)$. We have $\frac{85}{49}=[2, 4, 5, 2, 2]$, and $\frac{85}{36}=[3, 2, 3, 2, 2, 4]$. This c.q.s.~ has a deformation space with $5$ irreducible components. For each of them, we list the corresponding: zero continued fraction, dimension of the component, the vector of the $\delta_i$, the M-resolution, and the N-resolution.

\noindent
$\rule{12.7cm}{1.1pt}$
\noindent

$[1, 2, 2, 2, 2, 1]$, dimension is $10$, $(0, 2, 3, 0, 0)$

$(2)-(4)-(5)-(2)-(2)$ (minimal resolution)

$[{26 \choose 15}]-(1)-[{26 \choose 15}]-(1)-[{26 \choose 15}]-(1)-[{5 \choose 3}]-(1)-(2)$

\noindent
$\rule{12.7cm}{1.1pt}$
\noindent

$[2, 1, 3, 2, 2, 1]$, dimension is $8$, $(1, 7, 0, 0) $ 

$(2)-[{2 \choose 1}]-(5)-(2)-(2)$

$[{26 \choose 15}]-(1)-[{26 \choose 15}]-(1)-[{26 \choose 15}]-(1)-[{3 \choose 2}]-(1)$

\noindent
$\rule{12.7cm}{1.1pt}$
\noindent

$[1, 2, 3, 2, 1, 3]$, dimension is $6$, $(0, 8, 1)$

$(2)-(4)-[{3 \choose 1}]-(2)$

$ [{26 \choose 15}]-(1)-[{19 \choose 11}]-(1)-(2) $

\noindent
$\rule{12.7cm}{1.1pt}$
\noindent

$[2, 2, 3, 1, 2, 4]$, dimension is $2$, $(5)$

$(2)-[{7 \choose 2}]$ (extremal P-resolution)

$[{12 \choose 7}]-(1)$

\noindent
$\rule{12.7cm}{1.1pt}$
\noindent

$[3, 1, 3, 2, 1, 4]$, dimension is $2$, $(5)$

$[{3 \choose 2}]-(1)-[{4 \choose 1}]$ (extremal P-resolution)

$ [{19 \choose 11}]-(1)-[{3 \choose 2}]$

\noindent
$\rule{12.7cm}{1.1pt}$
\noindent

\end{example}



\section{Braid group action on Wahl resolutions}

Given a c.q.s.~$\oW$, we will show how to connect the M-resolution $W^+$ and the N-resolution $W^-$ by a sequence of antiflips, which are generators of the 
braid group $B_{r+1}$ action on the set of Wahl resolutions $W\to\oW$ with $r+1$ Wahl singularities. This action  comes from natural operations on deformations of Wahl resolutions. 


We first describe the action of $B_2\simeq\Z$ on extremal Wahl resolutions $W\to\oW$, where either $K_W\cdot\Gamma_1>0$ (extremal $P$-resolutions), $K_W\cdot\Gamma_1<0$
({\em $K$-negative resolutions}), or $K_W\cdot\Gamma_1=0$ when $\delta_1=0$
({\em $K$-trivial resolutions}).
We will refer to the action of a generator of $B_2$ as the {\em right antiflip}
and to its inverse as the {\em left antiflip}. 

A~$\Q$-Gorenstein smoothing $\cW\to B$ of 
an extremal Wahl resolution $W$
over a smooth curve can be blown-down to a smoothing $\ocW\to B$ of $\oW$. This gives a threefold contraction $\cW\to\ocW$, which is $K_{\cW}$-positive, $K_{\cW}$-negative, or 
$K_{\cW}$-trivial depending on the three cases above.
The antiflip is defined differently in each case.

\bigskip

{\bf Antiflips: $K$-positive case.}
Consider a $\Q$-Gorenstein smoothing $\cW^+\to B$ 
of 
an extremal P-resolution $W^+$ over a smooth curve.
One can ask if the morphism of threefolds $\cW^+\to\ocW$ given by blowing down the deformation admits an antiflip (a relative anticanonical model with terminal singularities).
This  was studied in \cite{HTU17} following an earlier work of Mori, Koll\'ar, and Prokhorov \cite{M88}, \cite{KM92}, \cite{M02}, \cite{MP11}. 
See \cite[S.~2]{U16} for a~summary of results. 
\begin{figure}[htbp]
\centering
\includegraphics[width=11cm]{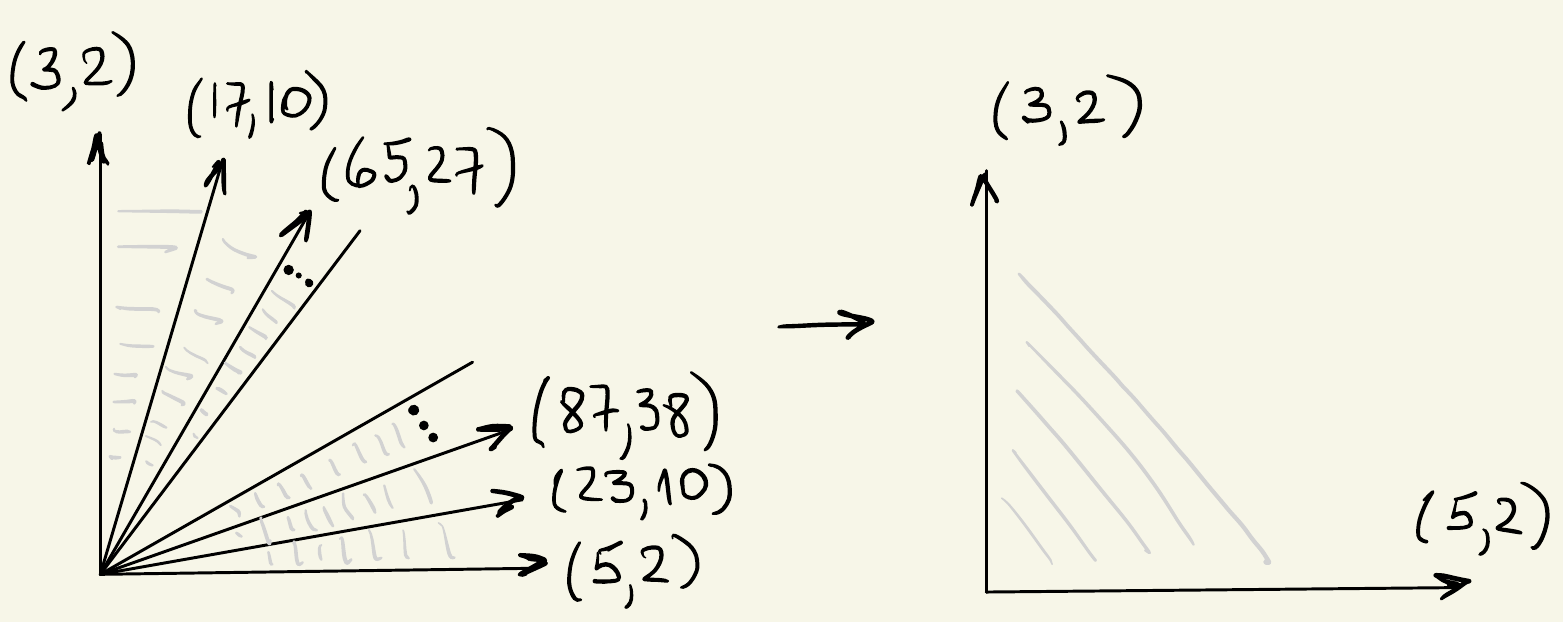}
\caption{Universal family of antiflips (on the left)\\ of the extremal P-resolution (on the right)}
\label{samplefamily}
\end{figure}
A~terminal antiflip exists if and only if the  boundary divisor $\Gamma_0+\Gamma_1+\Gamma_2\subset W^+$ deforms in $\cW^+$ and the axial multiplicities at $P_0$, $P_1$ satisfy $\alpha_0^2-\delta\alpha_0\alpha_1+\alpha_1^2>0$.
Each antiflip is a $\Q$-Gorenstein smoothing of a $K$-negative extremal Wahl resolution $W\to\oW$
and each of them 
appears this way for some $W^+$.
Terminal antiflips admit a universal family (see \cite{HTU17} for explicit equations) illustrated in Figure~\ref{samplefamily} in a concrete example, where  $\delta=4$.

Take an extremal P-resolution $[{n_0\choose a_0}]-(c)-[{n_1\choose a_1}]$  (in our example $[{3\choose 2}]-(1)-[{5\choose 2}]$) and consider its $\Q$-Gorenstein smoothing $\cW^+$ with a vector of axial multiplicities $(\alpha_0,\alpha_1)$  in the first quadrant (see the right side of Figure~\ref{samplefamily}). The corresponding antiflip~$\cW$  exists if $\alpha_0^2-\delta \alpha_0 \alpha_1+\alpha_1^2>0$. When  $(\alpha_0,\alpha_1)$ 
is in the interior of a $2$-dimensional cone $\sigma$ of the fan $\cF$ on the left side of Figure~\ref{samplefamily} then the special fiber $W \subset \cW$ is a $K$-negative extremal Wahl resolution. We ignore $1$-dimensional cones of $\cF$  until Section~\ref{s4}, where they will become crucial in the proof of Prop.~\ref{SRGsrgsRG}. The fan $\cF$ itself depends only on $\delta$ but it is decorated with data of Wahl singularities $(n,a)$ determined by certain recurrence relations. Decorations of the cone $\sigma$ determine Wahl singularities of $W$. If ~$\delta_1>1$ then $\cF$ is infinite and excludes the region $\alpha_0^2-\delta\alpha_0\alpha_1+\alpha_1^2\le 0$. In our example, we get 
$K$-negative extremal Wahl resolutions with 
singularities with data $[{5\choose 2}]-(1)-[{23\choose 10}]$,
$[{23\choose 10}]-(1)-[{87\choose 38}]$,~etc.~(see Figure~\ref{wwGwrgwRGW}). Expressing $(\alpha_0,\alpha_1)$
as a linear combination of generators of the cone $\sigma\subset\cF$ gives axial multiplicities of the anticanonical model (the smoothing $\cW$ of~$W$).

There are  two particular ``initial'' $K$-negative extremal Wahl resolutions $W_0^-$ and~$W_1^-$ that correspond to cones adjacent to the boundary of the first quadrant. 
\begin{figure}[htbp]
\centering
\includegraphics[width=8cm]{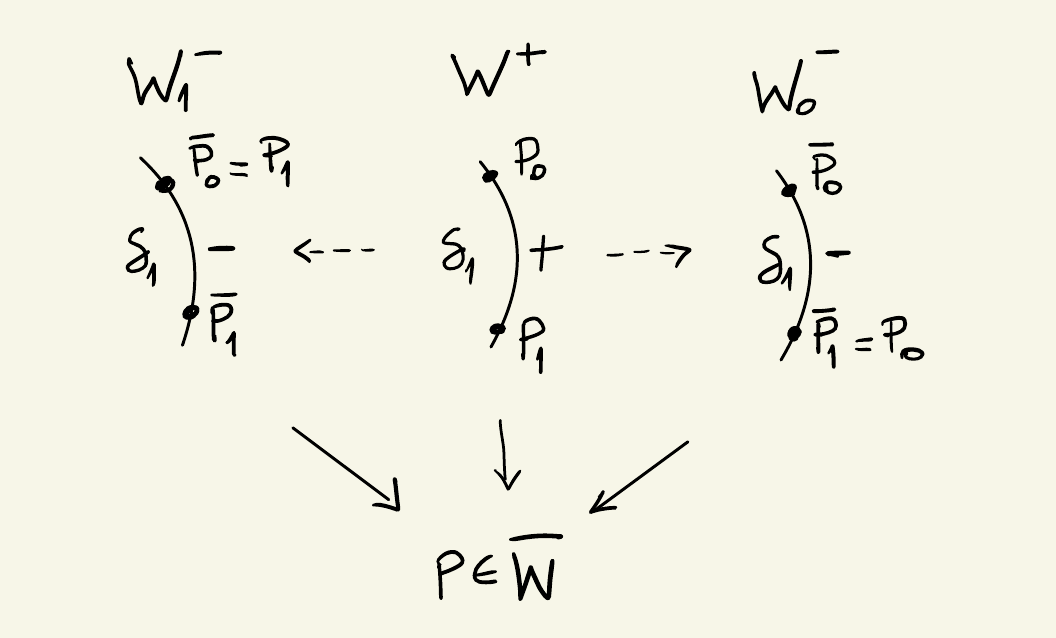}
\caption{Initial negative extremal Wahl resolutions}
\label{f2}
\end{figure}
Each of them preserves one of the Wahl singularities in $W^+$ (including the case of smooth points). To be precise, 
let $-c_1$ be the self-intersection of the proper transform of $\Gamma_1$ is the minimal resolution of $W^+$. 
The singularities of $W_0^-$ are $\bar P_0$ and $\bar P_1=P_0$
with 
\begin{equation}\label{wFwrgshasrh}
\bar n_0 = \delta_1 n_0 + n_1,\ 
\bar a_0 = \delta_1 a_0 + a_1 - (c_1 - 1)n_1\quad\hbox{\rm and}\quad \bar n_1= n_0,\  \bar a_1=a_0.
\end{equation}
The singularities of $W_1^-$ are $\bar P_0=P_1$
with $\bar n_0= n_1$,  $\bar a_0=a_1$ and 
$\bar P_1$ with $\bar n_1 = \delta_1 n_1 + n_0$,  $\bar a_1 = \delta_1 a_1 + a_0 + (c_1 - 1)n_0$.
The proper transform of $\Gamma_{1,i}\subset W_i^-$ for $i=0,1$ in the minimal resolution is a $(-1)$-curve, and $\bar\delta_1=\delta_1$ in both cases.
We refer to $W_0^- \to \overline{W}$ as the {\em right antiflip} (or just the antiflip) of an extremal $P$-resolution $W^+ \to \overline{W}$ and to $W_1^- \to \overline{W}$
as the {\em left antiflip}. 
By~Lemma~\ref{extreNres},
the right antiflip is the N-resolution of the extremal P-resolution.

Given $\Q$-Gorenstein smoothings $\cW_i^-$ of $W_i^-$ ($i=0,1$) over smooth curves $B_i$, the blow-down  deformations
$(\Gamma_{1,i} \subset \cW_i^-) \to (P \in \ocW_i)$  are birational contractions of $K_{\cW_i^-}$-negative curves of flipping type (k2A extremal neighborhoods \cite{HTU17}). The flips $\cW_i^- \dasharrow \cW^+_i$
give $\Q$-Gorenstein smoothings
$\cW_i^+ \to B_i$ of the same central fiber $W^+$. In particular, general fibers of $\cW_i^-$ and $\cW_i^+$ are isomorphic for each $i$.
The~curves $B_0$, $B_1$ are from the same component of $\Def_{P \in \overline{W}}$\footnote{
The threefolds $\cW_i^-$ are anticanonical models of $\cW^+_i$ for $i=0,1$. Since anticanonical models are unique (whenever they exist), we see that curves $B_0$ and $B_1$ are necessarily different.}. 

\smallskip

{\bf Antiflips: $K$-negative case.}
Antiflips 
of $K$-negative extremal resolutions $W$ correspond to counter-clockwise and clockwise rotations in Figure~\ref{samplefamily}
through the sequence of two-dimensional cones of the fan $\cF$.
The {\em right} (resp.~{\em left}) {\em antiflip} of a $K$-negative extremal resolution
$[{n_0\choose a_0}]-(1)-[{n_1\choose a_1}]$ 
different from $W_1^-$ (resp.~$W_0^-$)
is a $K$-negative extremal resolution 
$[{n'_0\choose a'_0}]-(1)-[{n'_1\choose a'_1}]$ 
such that  
$[{n_0\choose a_0}]=[{n'_1\choose a'_1}]$
(resp.~ 
$[{n'_0\choose a'_0}]=[{n_1\choose a_1}]$).
Their $\Q$-Gorenstein smoothings blow-down to 
smoothings of $\oW$ from the same
irreducible component of $\Def_{P \in \overline{W}}$. For example, universal family of  Figure~\ref{samplefamily} gives a sequence of right antiflips in Figure~\ref{wwGwrgwRGW} infinite in both directions.
The right (resp.~left) antiflip of $W_1^-$ (resp.~$W_0^-$) is actually a flip.

\begin{figure}[htbp]
\centering
\includegraphics[width=\textwidth]{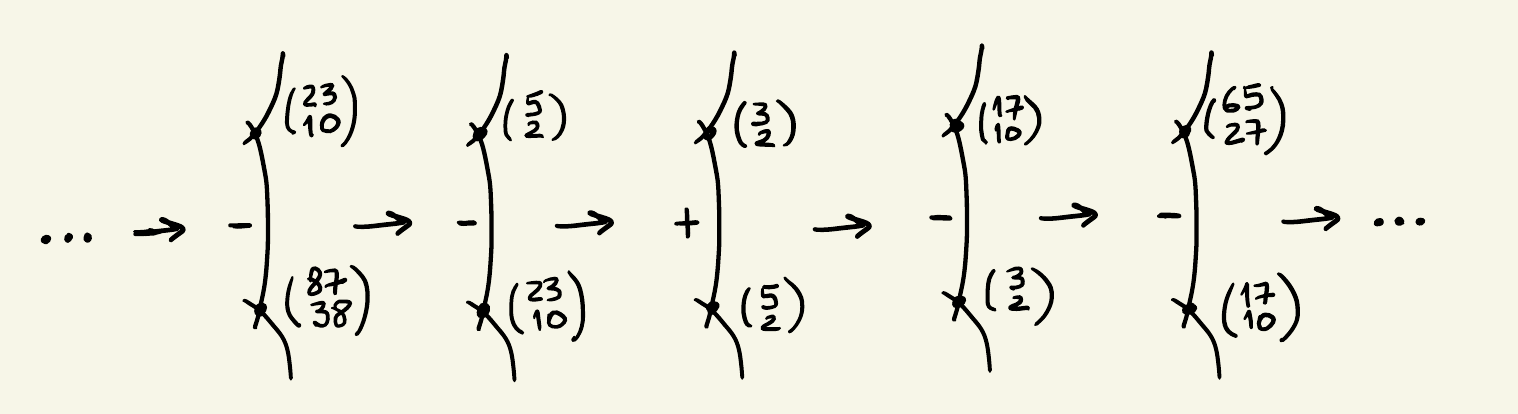}
\caption{Right antiflips of extremal Wahl resolutions.}
\label{wwGwrgwRGW}
\end{figure}

{\bf Antiflips: $K$-trivial  case.} 
Here   $W^+$ is $[{n\choose a}]-(1)-[{n\choose a}]$ (unless $n=1$ in which case it is a $(-2)$-curve with two smooth points) and $\oW={1\over 2n^2}(1,2na-1)$~\cite{BC}.  The blow-down $\ocW$
of a $\Q$-Gorenstein smoothing $\cW^+$ of $W^+$ is a $\Q$-Gorenstein smoothing of $\oW$. The contraction $\cW^+\to\ocW$ is crepant and can be flopped giving a threefold $\cW^-$ (which we call an antiflip of $\cW^+$) with a central fiber 
$W^-\simeq W^+$ (see \cite{BC} or \cite[S.~5]{K21}), which we call an antiflip of $W^+$ in the $K$-trivial case.

Before we consider the action of the braid group $B_{r+1}$ on $r+1$ strands, we address some global questions, namely existence of ``good'' divisors $\Gamma_0$ and $\Gamma_{r+1}$ and vanishing of local-to-global obstructions to deformations of Wahl resolutions.

\begin{lemma}\label{goodA}
Let $P\in \oW$ be a c.q.s.~surface
satisfying Assumption~\ref{assume} (1), (3). We can choose  effective, smooth divisors $\bar A$ and $\tilde{\bar A}$ such that the germ $P\in (\bar A\cup\tilde{\bar A})\subset\oW$ is et\'ale-locally isomorphic to the germ
$0\in(x=0)\cup(y=0)\subset\C^2/\mu_\Delta$.
Proper transforms $\Gamma_0$ of $\bar A$ and $\Gamma_{r+1}$ of $\tilde{\bar A}$ in any Wahl resolution $W$\ of $\oW$
intersect the chain $\Gamma_1\cup\ldots\cup\Gamma_r$
only at the end-points $P_0$ and~$P_r$,
where they give toric boundaries
opposite to $\Gamma_1$ (resp.~$\Gamma_r$).
\end{lemma}

\begin{proof}
We start by choosing $\bar A$ to be a Weil divisor generating the local class group of $P\in\oW$. We can make it effective
by adding a sufficiently ample Cartier divisor $H$. Then $\Delta\cdot \bar A$ is Cartier and effective. We can  add another multiple of $H$ to make $\Delta\cdot \bar A$ base-point-free. By Bertini theorem, we can  find a smooth divisor $D\in|\Delta\cdot\bar A|$ which does not pass through $P$.
Consider a cyclic cover 
$\pi:\,\hat W=\Spec_\oW\bigoplus\limits_{k=0}^{\Delta-1}\cO_\oW(-k\bar A)\to \oW$, where we use the canonical section $s_D\in H^0(\oW,\cO_\oW(D))$ to define the algebra structure using the map
$\cO_\oW(-\Delta\cdot \bar A)\mathop{\to}\limits^{s_D} \cO_\oW$
(see \cite{kolkov} for a  theory of cyclic covers). Then $\hat W$ is smooth and $\oW=\hat W/\mu_n$. More precisely,  $\pi_*\cO_{\hat W}=
\bigoplus\limits_{k=0}^{\Delta-1}\cO_\oW(-k\bar A)$ is an eigenvalue decomposition with respect to a primitive root $\zeta\in\mu_\Delta$,
where $\zeta$ acts on $\cO_\oW(-k\bar A)$ with weight $\zeta^{-k}$.
The cover is branched over $D$, where it is locally given by $(x,y)\mapsto (x,y^\Delta)$, where $y=0$ is a local equation of $D$, 
and over~$P$, where $\hat P=\pi^{-1}(P)$ is  a smooth point
and $\zeta$ acts on the tangent space of $\hat P$ with eigenvalues $(\zeta^a,\zeta^{\tilde a})$,
where $a$ and $\tilde a$ are coprime to $\Delta$.  Then $\mu_n$ acts on $H^0(\hat W,\pi^*H)$ with eigenspaces $H^0(\oW,\cO_\oW(H-k\bar A))$ for $k=0,\ldots,\Delta-1$.
The $\mu_\Delta$-action is free everywhere else.

Required divisors $A$, $\tilde A$ on $\oW$ can be found as equivariant sections  of $H^0(\hat W,\pi^*H)$ with eigenvalues $(\zeta^a,\zeta^{\tilde a})$. We require them to be smooth (in particular they restrict to given $\mu_\Delta$-eigenspaces in the tangent space of $\hat P$) and
be transversal to the branch divisor $\pi^{-1}(D)$. We claim that these divisors exist if $H$ is sufficiently ample. Since these geometric conditions on global sections are open  (in fixed $\mu_\Delta$-eigenspaces of $H^0(\hat W,\pi^*H)$), it suffices to prove that there exist sections as above having required geometric properties in finitely many equivariant affine charts $U_\alpha$ that cover $\hat W$ and equivariantly trivialize $\cO(\pi^*H)$. Since $\pi^*H$ is ample, 
any local section of $\cO_{\hat W}(U_\alpha)$ is a restriction of a global section of $\pi^*(kH)$
for a sufficiently  large $k$. So~it suffices to prove that there exist equivariant sections of $\cO_{\hat W}(U_\alpha)$ which have required geometric properties, which is clear. For example, in the equivariant affine chart containing $\hat P$, we can take any equivariant regular function that restricts to the given eigenspace in the tangent space of $\hat P$, and in particular is smooth at $\hat P$,
and then shrink the affine neighborhood to make it smooth everywhere in it.
\end{proof}

\begin{lemma}\label{NO}
Let $\pi:\,W \to \overline W$ be a c.q.s.~ resolution satisfying Assumption \ref{assume}. 
\begin{enumerate}
    \item 
We can choose divisors $\Gamma_0$ and $\Gamma_{r+1}$ 
as in Lemma~\ref{goodA}  so 
that there are no local-to-global obstructions to deformations of a pair
$(W,\Delta)$, where $\Delta$ is the boundary $\Gamma_{0}+\Gamma_1+\ldots+\Gamma_r+ \Gamma_{r+1}$, i.e.~the morphism 
$\Def_{(W,\Delta)}\to\prod\limits_{P_i\in W}\Def_{P_i\in (W,\Delta)}$ is smooth. 
\item If $W$ is a Wahl resolution then there are no local-to-global obstructions to $\Q$-Gorenstein deformations of $W$ or $(W,\Delta)$, for example there exists a $\Q$-Gorenstein smoothing $Y\rightsquigarrow W$ with a lifting of $\Delta$ 
for any choice of axial multiplicities
$\alpha_0,\ldots,\alpha_{r}$.
\end{enumerate}
\end{lemma}

\begin{proof}
The versal $\Q$-Gorenstein deformation a Wahl singularity ${1\over n^2}(1,na-1)$ 
is $(xy=z^n+t)\subset{1\over n}(1,-1,a,n)$, where $t$ is a deformation parameter and $xy=0$ is the local equation of $\Delta$.
So the morphism $\Def^{\Q G}_{P_i\in (W,\Delta)}\to \Def^{\Q G}_{P_i\in W}$ is obviously smooth and we only need to prove the first statement.

We have a subsheaf $T_W(-\log\sum\limits_{i=1}^{r}{ \Gamma_i})\subset T_W$ of derivations that preserve ideal sheaves of $\Gamma_1,\ldots,\Gamma_r$.
It is well-known that $R\pi_*T_W(-\log\sum\limits_{i=1}^{r}{ \Gamma_i})$ is a subsheaf of $T_\oW$ with the quotient sheaf supported at $P\in\oW$ \footnote{It suffices to prove that $R^k\pi_*T_W(-\log\sum\limits_{i=1}^{r}{ \Gamma_i})=0$ for $k>0$. We work in a toric et\'ale neighborhoods of $P\in\oW$ and its partial resolution $\Gamma_1+\ldots+\Gamma_r\subset W$.
Then $R^k\pi_*T_W(-\log\Delta)=0$ for $k>0$ because $T_W(-\log\Delta)\simeq\cO_W\otimes\C^2$ in this et\'ale neighborhood. 
We have an exact sequence 
$0\to T_W(-\log\Delta)\to
T_W(-\log(\sum\limits_{i=1}^{r}{ \Gamma_i}))\to F\to0$,
where $F$ is supported on $\Gamma_0\cup\tilde \Gamma_{r+1}$.
Then $R^k\pi_*F=0$ for $k>0$ and we have a required vanishing.}. By Assumption \ref{assume}~(4), it follows that 
$H^2(W, T_W(-\log\sum\limits_{i=1}^{r}{ \Gamma_i}))=H^2(\oW,T_\oW)=0$. Since the sheaf  $T_W(-\log\sum\limits_{i=1}^{r}{ \Gamma_i})$ is the sheaf of infinitesimal automorphisms of the pair $(W,\Gamma_1+\ldots+\Gamma_r)$, it follows that there are no local-to-global obstructions to deformations of the pair. 

We need to find $\Gamma_0$ and $\Gamma_{r+1}$ that satisfy  Lemma~\ref{goodA} and
cohomology vanishing $H^2(W, T_W(-\log \Delta))=0$.
Let $q:\,\underline W\to W$ be an orbifold stack associated to $W$. An exact sequence 
$0\!\to\! T_{\underline W}(-\log\sum\limits_{i=0}^{r+1}{\underline \Gamma_i})\!\to\!
T_{\underline W}(-\log\sum\limits_{i=1}^{r}{\underline \Gamma_i})\!\to\! N_{{\underline \Gamma_0}/{\underline W} }\oplus N_{{\underline{\Gamma_{r+1}}}/{\underline W} }\!\to\!0$ on~$\underline W$
pushes forward to $
0\!\to\! T_W(-\log\Delta)\!\to\!
T_W(-\log\sum\limits_{i=1}^r\Gamma_i)\!\to\! F\oplus\tilde F\!\to\!0$,
where $F$ (resp.~$\tilde F$) is a sheaf supported on $\Gamma_0$ (resp.~$\Gamma_{r+1}$). It suffices to prove that 
$H^1({\underline \Gamma_0}, N_{{\underline \Gamma_0}/{\underline W} })=H^1(\underline{\Gamma_{r+1}}, N_{{\underline{\Gamma_{r+1}}}/{\underline W} })=0$.
Pushing forward an exact sequence 
$0\to\cO_{\underline W}\to\cO_{\underline W}({\underline \Gamma_{0}})\to N_{{\underline \Gamma_{0}}/{\underline W} }\to0$,
and using Assumption~\ref{assume} (2),
it suffices to prove that 
$H^1({\underline W},\cO_{\underline W}({\underline \Gamma_{0}}))=0$.
This holds after twisting $\Gamma_{0}$ with a sufficiently ample Cartier divisor $H$ in Lemma~\ref{goodA}. The same proof works for $\underline{\Gamma_{r+1}}$.
\end{proof}


We will define the action of  generators of $B_{r+1}$ on Wahl resolutions with $r+1$ singularities by treating every irreducible curve in its exceptional divisor as an extremal Wahl resolution.
Relations of the braid group are checked in Theorem~\ref{braidrelations}.

\begin{definition}\label{sGshsRHWRH}
Let $W \to \oW$ be a Wahl resolution with exceptional divisor $\Gamma_1\cup\ldots\cup\Gamma_r$ and  
toric boundaries $\Gamma_0$ and $\Gamma_{r+1}$ as in Lemmas~\ref{goodA} and \ref{NO}. The neighborhood of $\Gamma_i\subset W$ contains a subchain $[{n_{i-1}\choose a_{i-1}}]-(c_i)-[{n_{i}\choose a_{i}}]$ of an extremal Wahl resolution.
The~contraction $W\to W_i$ of $\Gamma_i\subset W$  gives a c.q.s.~surface.
The image of $\Gamma_i$ is a c.q.s. $\frac{1}{\Delta_{\Gamma_i}}(1,\Omega_{\Gamma_i})$, which has as toric boundary the image of $\Gamma_{i-1}$ and~ $\Gamma_{i+1}$. By Lemma~\ref{NO}, we can choose two deformations of $W_i$ (the same ones if the extremal resolution is a P-resolution or a  $K$-trivial resolution)
which (1) are equisingular at  singularities of $W_i$ other than $\frac{1}{\Delta_{\Gamma_i}}(1,\Omega_{\Gamma_i})$, (2) lift the boundary of $W_i$, and (3)~ smoothen $\frac{1}{\Delta_{\Gamma_i}}(1,\Omega_{\Gamma_i})$ as in the discussion of antiflips of extremal Wahl resolutions in the beginning of this section. These deformations of $W_i$ are blow-downs of $\Q$-Gorenstein deformations of $W$ and another Wahl resolution $R_i(W)\to \oW$, respectively.
We call $R_i(W)$ the right antiflip of 
$W \to \oW$ at~$\Gamma_i$.
The left antiflip  is defined is a similar way. The singularities of  $R_i(W)$ and $L_i(W)$ 
are the same as for $W$ except at the positions $i-1$ and $i$, where we have the singularities produced by the antiflip of an extremal Wahl resolution
$[{n_{i-1}\choose a_{i-1}}]-(c_i)-[{n_{i}\choose a_{i}}]$.
\end{definition}

\begin{corollary}\label{3foldantiflip}\label{3foldantiantiflip}
Given a sequence of Wahl resolutions  $W_0, W_1, \ldots, W_k\to\oW$ with Wahl chains $\Gamma_0^j,\ldots,\Gamma_{r+1}^j$ for $j=0,\ldots,k$, suppose $W_{i}=R_{l_i}(W_{i-1})$ 
for $i=1,\ldots,k$.

\begin{enumerate}
    \item 
There is a sequence of $\Q$-Gorenstein smoothings  $Y_i\rightsquigarrow W_i$ for $i=0,1,\ldots,k$ over smooth curve germs $B_i$ that  belong to the same component of $\Def_{P \in \oW}$.

\item If  $K_{W_{i-1}}\cdot\Gamma^{i-1}_{l_i}\ge0$
for $i=1,\ldots,k$,
i.e.~on every step we antiflip an extremal P-resolution or a $K$-trivial resolution,
then we can assume that $B_1=\ldots=B_k=B$
is the same curve in $\Def_{P \in \oW}$ and  $(W_{i-1} \subset \W_{i-1}) \to (0 \in B)$ is the flip (or flop) of $(W_{i} \subset \W_{i}) \to (0 \in B)$ for all $i=1,\ldots,k$
with respect to the contraction of $\Gamma^i_{l_i}\subset W_{i}$. In particular,  the smooth fibers $Y_i$ of these families are isomorphic.
\end{enumerate}
\end{corollary}

\begin{proof}
(1) is clear. To prove (2), 
choose a 
$\Q$-Gorenstein smoothing $(W_k \subset \W_k) \to (0 \in B)$ over a smooth curve germ $B$
with all axial multiplicities equal to $1$, which exists by Lemma~\ref{NO}.
Then we  apply a sequence of flips (or flops if $\delta^i_{l_i}=0$) to contractions of 
$\Gamma^i_{l_i}\subset\W_i$ for $i=k,k-1,\ldots,1$.
\end{proof}


\begin{proposition}\label{computo}
Let $W \to \overline{W}$ be a Wahl resolution with a chain of $3$ curves $\Gamma_1, \Gamma_2, \Gamma_3$, and singularities $P_0, P_1, P_2, P_3$ where the type of $P_i$ is ${1\over n_i^2}(1,n_i a_i-1)$. Consider $W':=R_2(W)$ the right antiflip of $W \to \overline{W}$ at $\Gamma_2$. Hence we have a Wahl resolution $W' \to \overline{W}$ with a chain of $3$ curves $\Gamma'_1, \Gamma'_2, \Gamma'_3$, and singularities $P'_0=P_0, P'_1, P'_2=P_1, P'_3=P_3$. Let $\delta'_i=n'_{i-1} n'_i \, |K_{W'} \cdot \Gamma'_i|$. Then we have the following three situations:

\bigskip
\noindent \textbf{(-/-)}: $\Gamma_2 \cdot K_W < 0$ and $\Gamma'_2 \cdot K_{W'} < 0$.

\begin{enumerate}
    \item $n'_1=\delta_2 n_1 - n_2$, $a'_1=\delta_2 a_1 - a_2$, $n'_2=n_1$, $a'_2=a_1$, $\delta'_2=\delta_2$.
    
    \item $n'_0 n'_1 \, \Gamma'_1 \cdot K_{W'}=\frac{ \pm \delta_1 (\delta_2 n_1 - n_2)+ \delta_2 n_0}{n_1}$, where $\pm$ is the sign of $K_W \cdot \Gamma_1$.
    
    \item $ n'_2 n'_3 \, \Gamma'_3 \cdot K_{W'} ={ \pm \delta_3 n_1 - \delta_2 n_3\over n_2}$, where $\pm$ is the sign of $K_W \cdot \Gamma_3$.
    
\end{enumerate}

\bigskip   
\noindent \textbf{(-/+)}: $\Gamma_2 \cdot K_W < 0$ and $\Gamma'_2 \cdot K_{W'} > 0$. Let $-c'_2$ be the self intersection of the proper transform of $\Gamma'_2$ in the minimal resolution of $W'$.

\begin{enumerate}
    \item $n'_1=n_2-\delta_2 n_1$, $a'_1=a_2-\delta_2 a_1 - (c'_2 - 1)n'_1$, $n'_2=n_1$, $a'_2=a_1$, $\delta'_2=\delta_2$, and $\delta'_2=(c'_2-1)n'_1n_1+n_1a'_1-n'_1a_1$.
    
    \item $n'_0 n'_1 \, \Gamma'_1 \cdot K_{W'}=\frac{ \pm \delta_1 (n_2 -\delta_2 n_1) - \delta_2 n_0}{n_1}$, where $\pm$ is the sign of $K_W \cdot \Gamma_1$.
    
    \item $ n'_2 n'_3 \, \Gamma'_3 \cdot K_{W'} ={ \pm \delta_3 n_1 - \delta_2 n_3\over n_2}$, where $\pm$ is the sign of $K_W \cdot \Gamma_3$.
    
\end{enumerate}

\bigskip
\noindent \textbf{(+/-)}: $\Gamma_2 \cdot K_W \geq 0$ and $\Gamma'_2 \cdot K_{W'} \leq 0$. Let $-c_2$ be the self intersection of the proper transform of $\Gamma_2$ in the minimal resolution of $W$.

\begin{enumerate}
    \item $n'_1=\delta_2 n_1 + n_2$, $a'_1=\delta_2 a_1 + a_2 - (c_2 - 1)n_2$, $n'_2=n_1$, $a'_2=a_1$, $\delta'_2=\delta_2$.
    
    \item $n'_0 n'_1 \, \Gamma'_1 \cdot K_{W'}=\frac{ \pm \delta_1 (\delta_2 n_1 + n_2)+ \delta_2 n_0}{n_1}$, where $\pm$ is the sign of $K_W \cdot \Gamma_1$.
    
    \item $ n'_2 n'_3 \, \Gamma'_3 \cdot K_{W'} ={ \pm \delta_3 n_1 + \delta_2 n_3\over n_2}$, where $\pm$ is the sign of $K_W \cdot \Gamma_3$.
    
\end{enumerate}

In particular, we have in all cases that $$K_{W} \cdot \Gamma_1=K_{W'} \cdot (\Gamma'_1 +\Gamma'_2) \ \ \ \text{and} \ \ \ K_{W'} \cdot \Gamma'_2=K_{W} \cdot (\Gamma_1 +\Gamma_2).$$
\end{proposition}
 
\begin{proof}
We only prove the case \textbf{(+/-)}, the others follow similar computations. The new singularities $P_1'$, $P_2'$ in (1) are computed by the formulas \eqref{wFwrgshasrh}. In particular $\delta'_2=\delta_2$. To find $\delta'_1$ and $\delta'_3$, we contract $\Gamma_2 \subset W$ and $\Gamma'_2 \subset W'$, and do intersection theory on the singular surfaces involved. We will only compute $\delta'_1$.

Let $\pi \colon W \to \tilde W$ and $\pi' \colon W' \to \tilde W$ be the contractions of $\Gamma_2 \subset W$ and $\Gamma'_2 \subset W'$ respectively. We have $K_{W} \equiv  \pi^*(K_{\tilde W}) - \epsilon_+ \Gamma_2$, and $K_{W'} \equiv  {\pi'}^*(K_{\tilde W}) + \epsilon_- \Gamma'_2$. Intersection with $\Gamma_2$ and $\Gamma'_2$ gives $\epsilon_+=\frac{\delta_2 n_1 n_2}{\tilde \Delta}$ and $\epsilon_-=\frac{\delta_2 n_1 n'_2}{\tilde \Delta}$, where $\tilde \Delta= n_1^2 + n_2^2 + \delta_2 n_1 n_2$. As $K_{W'} \cdot \Gamma'_1= \frac{\pm \delta'_1}{n_0 n'_1}$, we use both of the previous equations to intersect with $\Gamma'_1$, and we find (2). We note that $K_W \cdot \Gamma_1>0$ implies $K_{W'} \cdot \Gamma'_1 >0$.
\end{proof}

\begin{lemma} \label{braid}
Let $W_0 \to \overline{W}$ be a Wahl resolution. Then we have the braid relation $R_2 R_1 R_2(W_0)=R_1 R_2 R_1(W_0)$.
\end{lemma}

\begin{proof}
We have in principle $7$ distinct surfaces in this proof: $W_0$, $R_2(W_0)=W_1$, $R_1(W_1)=W_2$, $R_2(W_2)=W_3$,
$R_1(W_0)=W'_1$, $R_2(W'_1)=W'_2$, and $R_1(W'_2)=W'_3$. We want to show that $W_3=W'_3$. For each $W_i \to \overline{W}$ we have a chain of $2$ curves $\Gamma_{1,i}, \Gamma_{2,i}$, singularities $P_{0,i}, P_{1,i}, P_{2,i}$, and $\delta_{j,i}$. We have the analogue notation with ' for the Wahl resolutions $W'_i \to \overline{W}$.

By Proposition \ref{computo}, we have that $$K_{W_0} \cdot \Gamma_{1,0}=K_{W_1} \cdot (\Gamma_{1,1} +\Gamma_{2,1}) =K_{W_2}\cdot \Gamma_{2,2}.$$ We also have that $P_{0,0}=P_{1,2}$ and $P_{1,0}=P_{2,2}$. Therefore the antiflipping of $\Gamma_{1,0}$ is equal to the antiflipping of $\Gamma_{2,2}$, and so $P_{1,3}=P'_{1,3}$ and $P_{2,3}=P'_{2,3}$. Morover, by Proposition \ref{computo} again, we have $$K_{W'_1} \cdot \Gamma'_{1,1}=K_{W'_2} \cdot (\Gamma'_{1,2} +\Gamma'_{2,2}) =K_{W_3}\cdot \Gamma'_{2,3},$$ and $\Gamma_{1,1}$ is the antiflip of $\Gamma_{1,0}$. Hence the contraction of $\Gamma'_{2,3}$ and $\Gamma_{2,3}$ define the same c.q.s and have same $\delta$.

Recall that $W_3 \to \overline{W}$ and $W'_3 \to \overline{W}$ contract the Wahl chain to the same c.q.s. On the other hand, as $\Gamma'_{2,3}$ and $\Gamma_{2,3}$ contract to the same c.q.s., we also have that $\Gamma'_{1,3}$ and $\Gamma_{1,3}$ contract to the same c.q.s. They are both extremal Wahl resolutions over the same c.q.s., and so, since c.q.s. at most have two extremal P-resolutions with the same $\delta$ \cite{HTU17}, we have that $\delta_{1,3}=\delta'_{1,3}$. Moreover, since $P_{1,3}=P'_{1,3}$, we have that $P_{0,3}=P'_{0,3}$.  Therefore we obtain $R_2 R_1 R_2(W_0)=R_1 R_2 R_1(W_0)$.
\end{proof}

We will describe the sequence of right antiflips
$W_0,W_1,W_2\ldots$ that terminates with a N-resolution of a given M-resolution  $W_0:=W^+ \to \overline{W}$ with $r+1$ singularities. 
For $r=1$, the N-resolution is equal to the right antiflip of an extremal M-resolution
by Lemma~\ref{extreNres}.

\begin{proposition}\label{r=2}
If $r=2$, then we arrive to the N-resolution after applying three right antiflips as shown in Figure \ref{f3}, where $R_2(W_0)=W_1$, $R_1(W_1)=W_2$, and $R_2(W_2)=W_3$. In this way, by Lemma \ref{braid}, we have $W^-=R_2 R_1 R_2(W_0)=R_1 R_2 R_1(W_0)$.
\end{proposition}

\begin{figure}[htbp]
\centering
\includegraphics[width=8.5cm]{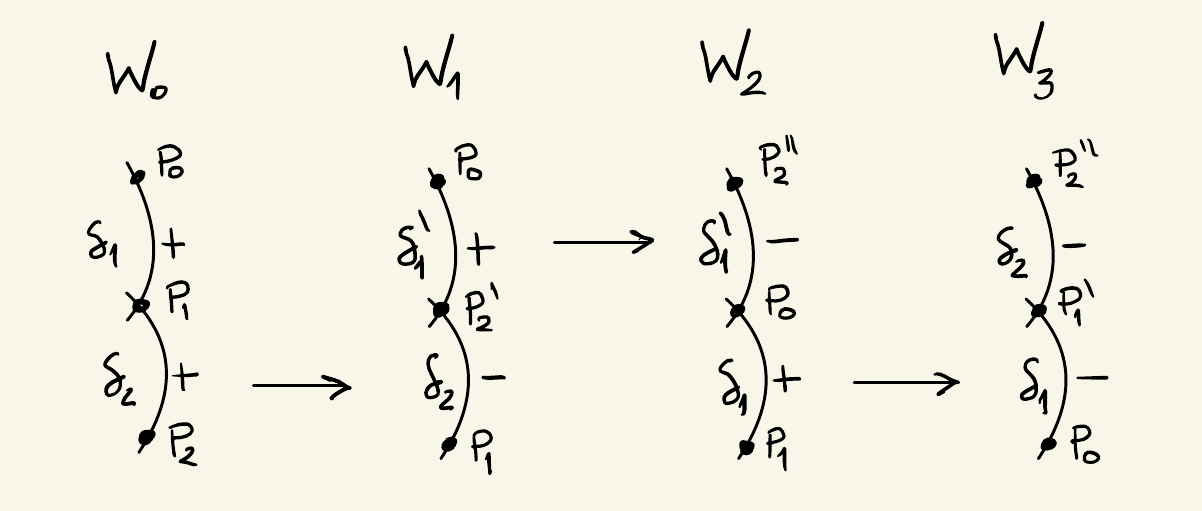}
\centering
\caption{The antiflipping process for $3$ singularities}
\label{f3}
\end{figure}

\begin{proof}
We will do the explicit computation following Prop.~ \ref{computo}. We first antiflip at $\Gamma_{2,0} \subset W_0$ as in Figure~\ref{f3}. By Prop.~ \ref{computo} we have that $\Gamma_{1,1} \cdot K_{W_1} = \frac{\delta_1 \delta_2 n_1 + \delta_1 n_2+ \delta_2 n_0}{n_1 n_0 \bar n_2}\ge0$ using the notation in Figure \ref{f3}. Next we antiflip at $\Gamma_{1,1}\subset W_1$. By using  Prop.~\ref{computo}, one can verify that $\Gamma_{2,2} \cdot K_{W_2}= \frac{\delta_1}{n_0 n_1}\ge0$. 
Finally, we antiflip at $\Gamma_{2,2}\subset W_2$. We compute using Prop.~ \ref{computo} that $\Gamma_{1,3} \cdot K_{W_3}= -\frac{\delta_2}{\bar{\bar{n}}_2 \bar n_1}$. Moreover the c.q.s. contraction of $\Gamma_{1,0}$ in $W_0$ is the c.q.s. of the contraction of $\Gamma_{2,2}$ in $W_2$, which is the c.q.s. of the contraction of $\Gamma_{2,3}$ in $W_3$. Hence $R_2 R_1 R_2(W_0)$ is the N-resolution by Definition~\ref{nres}. 

\end{proof}




\begin{theorem}\label{nresflips}
After applying $r(r+1)/2$ right antiflips of curves contained in the Wahl resolutions starting with $W^+ \to \overline{W}$, we get the corresponding  N-resolution $W^- \to \overline{W}$. On every step, we antiflip either an extremal P-resolution or a curve with $\delta=0$.
\end{theorem}

\begin{proof}
The proof goes by induction on $r$. By Lemma~\ref{extreNres}, Prop.~\ref{r=2}, we have it for $r=1,2$. Let us say we have it for $r$. Consider an M-resolution $W^+ \to \overline{W}$ with the exceptional divisor  $\Gamma_1\cup\ldots\cup\Gamma_{r+1}$. 
We know that there is a unique N-resolution $W^- \to \overline{W}$ of $W^+ \to \overline{W}$ with curves $\bar \Gamma_1,\ldots, \bar \Gamma_{r+1}$ and singularities $\bar n_i, \bar a_i$ for $i=0,\ldots,r$ with $\bar \delta_i$. 

We  note that the N-resolution for the contraction of the chain  $\Gamma_1\cup\ldots\cup\Gamma_{r}\subset W$ to some c.q.s. has the same $\bar n_i, \bar a_i$ for $i=0,\ldots,r-1$, and same $\bar \delta_i$ for $i=1,\ldots,r-1$ as the N-resolution $W^-$. This is just part of the algorithm to find the N-resolution. 

Let us first antiflip the curves $\Gamma_{r+1},\ldots, \Gamma_1$ in that order, starting with $W_0=W^+ \to \overline{W}$ an ending with $W_{r+1} \to \overline{W}$. 
Then by using Proposition \ref{r=2} for each consecutive pair $\Gamma_i, \Gamma_{i+1}$, where we have only the first two flips in Figure \ref{f3} for each pair, we have 

\begin{itemize}
    \item  $n_{i+1,r+1}= n_{i}$, $a_{i+1,r+1}=a_i$ for $i=0,\ldots r$,
    \item $\delta_{i+1,r+1}=\delta_{i}$ for $i=1,\ldots r$,
    \item The Wahl-chain $\Gamma_{2,r+1},\ldots, \Gamma_{r+1,r+1}$ is an M-resolution, and
    \item $\Gamma_{1,r+1} \cdot K_{W_{r+1}} \leq 0$.
\end{itemize}

So we have that $\Gamma_{2,r+1},\ldots, \Gamma_{r+1,r+1}$ is an M-resolution isomorphic to $\Gamma_1,\ldots,\Gamma_r$. We now apply induction on this chain, to obtain via $r(r-1)/2$ antiflips its unique N-resolution. As it was said above, the $\bar n_i, \bar a_i$ coincide with the ones of $W^- \to \overline{W}$, and the only missing part is the first singularity and the intersection with with canonical class of the first curve. In this way, we have a new Wahl-resolution with $r+1$ curves $\tilde{W} \to \overline{W}$, and we want to prove it is indeed the N-resolution. 

At this point, we do not know about $\tilde n_0, \tilde a_0$ and $\tilde \delta_0$. The corresponding curve $\tilde \Gamma_1$ may be positive or negative for canonical class. But we now reverse the continued fractions for $\tilde{W} \to \overline{W}$ and for $W^- \to \overline{W}$, and after contracting we get a
$$[{ \tilde n_0 \choose \tilde n_0 - \tilde a_0}]-(c)-[{\bar n_1 \choose \bar n_1 - \bar a_1}] -(1)-[\text{some c.q.s.}]=0 $$ $$[{ \bar n_0 \choose \bar n_0 - \tilde a_0}]-(1)-[{\bar n_1 \choose \bar n_1 - \bar a_1}] -(1)-[\text{some c.q.s.}]=0 $$ over the same c.q.s. But then $\delta$'s must be equal, as we are in the case of an extremal P-resolution and/or extremal N-resolution over the same c.q.s., and so we can apply \cite{HTU17}. Moreover, they share a Wahl singularity in the same position, then they must be equal (singularity and sign of intersection with canonical class).
\end{proof}

\begin{theorem}\label{braidrelations}
The operations of right antiflips $R_i$ on Wahl resolutions $W\to\oW$
with $r+1$ singularities satisfy braid relations $R_iR_j=R_jR_i$ for $i>j+1$ and 
$R_{i}R_{i+1}R_i=R_{i+1}R_iR_{i+1}$.
In particular, they give the action of the braid group $B_{r+1}$.
\end{theorem}

\begin{proof}
For the relation $R_iR_j=R_jR_i$ we just note that $i-j\geq 2$, and so the birational operations $R_i,R_j$ on $W$ commute. Indeed, the only scenario to consider is when $i=j+2$, in which case we still need to show that the $(j+1)$-st curves in $R_iR_j(W)$
$R_jR_i(W)$ have the same self-intersections in the minimal resolutions. But since all other data is the same, this follows from the fact that both of them resolve $\oW$.

For the other relation, we first note that the curves at $i$ and $i+1$ positions of $R_{i}R_{i+1}R_i(W)$ and $R_{i+1}R_iR_{i+1}(W)$ have the same subWahl chain by Lemma \ref{braid}. 
Hence for the whole Wahl chains in $R_{i}R_{i+1}R_i(W)$ and $R_{i+1}R_iR_{i+1}(W)$ we have the same singularities. We only need to check the effect on the curves in positions $i-1$ and $i+2$, since in their complement we have equal Wahl subchains. Let $[A]$ be the continued fraction for the minimal resolution of the chain from $0$ to $i-2$, let $[B]$ be the one for the chain from $i$ to $i+1$, and let $[C]$ be continued fraction for the chain from $i+3$ to $r$. Then we are in the situation $$[A]-(u)-[B]-(v)-[C]=[A]-(u')-[B]-(v')-[C]$$ as continued fractions for the Wahl chains in $R_{i}R_{i+1}R_i(W)$ and $R_{i+1}R_iR_{i+1}(W)$, since both contract to $\frac{1}{\Delta}(1,\Omega)$. We want $u=u'$ and $v=v'$. Using the representation of continued fractions as multiplications of $2 \times 2$ matrices, we simplify that equation into $(u)-[B]-(v)=(u')-[B]-(v')$. But as continued fractions this means that $u-\frac{1}{[B,v]} = u'-\frac{1}{[B,v']}$, and if $u\geq u'$, then $(u-u')-\frac{1}{[B,v]} = -\frac{1}{[B,v']}$. But $\frac{1}{[B,v']}>0$, and so $(u-u')-\frac{1}{[B,v]} <0$, but then $u=u'$. Hence $v=v'$. Therefore Wahl chains are equal.
\end{proof}

\begin{corollary}
Every Wahl resolution $W\to\oW$ is in the braid group orbit of a  unique $M$-resolution $W^+\to\oW$.
\end{corollary}

\begin{proof}
By definition of a Wahl resolution, it has a $\Q$-Gorenstein smoothing $Y\rightsquigarrow W$ that blows down to a smoothing  $Y\rightsquigarrow\oW$.
Let $\cW$ be its total space.
If $K_W\cdot\Gamma_i<0$
for one of the curves $\Gamma_i\subset W$
then its contraction 
$\cW\to\cW_i$ is a k2A extremal neighborhood of flipping type
and can be flipped to $\cW'\to\cW_i$. By~ the discussion from the beginning of this section, the special fiber $W'\subset\cW'$can be obtained from $W$ via a sequence of antiflips at the same $i$-th curve. 

The indices of the Wahl singularities corresponding to $\Gamma'_i \subset W'$ are always less than or equal to the indices of the Wahl singularities corresponding to $\Gamma_i \subset W$, and one of these new indices is strictly smaller. Indeed, as in Figure \ref{f2}, the first right or left antiflip behaves like that. If $\delta_i=1$, then there are no more k2A antiflips. If $\delta_i>1$, then there are infinitely many antiflips with indices $n_k>n_{k-1}$ which satisfy the recurrence $n_{k-1}+n_{k+1}=\delta_i n_k$, and so indices form an increasing sequence.    

Then we can pick another $K$-negative curve in $W'$ and iterate the process. If we arrive to a surface with no Wahl singularities, then it must be the minimal resolution of $\oW$, as any $(-1)$-curve in the Wahl chain would produce a divisorial contraction blow-down deformation of $\oW$, and this is not allowed by the minimality assumption on Wahl resolutions.  

Therefore, as the set of indices in the Wahl chain strictly decreases at least one of them for each flip, we have that this process eventually gives a $M$-resolution~$W^+$. (Alternatively, any sequence of flips in dimension $3$ terminates.) Since $W^+$ has a $\Q$-Gorenstein smoothing $\cW^+$ in the same irreducible component of $\Def_{P\in\oW}$ as $\cW$, the M-resolution $W^+$ is uniquely determined. 
\end{proof}


\section{Derived category of a c.q.s.~surface following \cite{KKS}} \label{s2}

\begin{notation}\label{asrgasrgarh}
We fix a c.q.s. surface $W$ that satisfies Assumption \ref{assume} (1), (2),~(3). 
We do not need the singularities to be Wahl or the chain $\Gamma_1,\ldots,\Gamma_r$ to be contractible. 
In fact,
a weaker condition
than  Assumption~\ref{assume}  (3),  is sufficient in Sections
\ref{s2} and \ref{AEFefEEfe}
except in Corollaries~\ref{asgasrgsarhasr} and ~\ref{EFegwEWE}:
there exists a Weil divisor $A\subset W$, which is Cartier outside of $P_0$ and generates the local class group $\Cl(P_0 \in W)$. Then $\tilde A=
-K_W-\Gamma_1-\ldots-\Gamma_r-A$
is Cartier outside of $P_r$ and generates $\Cl(P_r \in W)$. 
\end{notation}

Consider the following Weil divisors on $W$:
\begin{equation}\label{EFefeGeEG}
D_0=A, \quad D_1=A+\Gamma_1,\quad\ldots,\quad D_r=A+\Gamma_1+\ldots+\Gamma_r.
\end{equation}
If all points $P_0,\ldots,P_r$ are smooth then this gives a
well-known exceptional collection 
$\cO_W(D_0), \cO_W(D_1), \ldots, \cO_W(D_r)$. 
But if some points are singular then this collection of divisorial sheaves is typically not exceptional. Indeed, if $i>j$ then
\begin{equation}\label{wegwEGweweg}
R\Gamma\cHom(\cO_W(D_i),\cO_W(D_j))=
R\Gamma\cO_W(-\Gamma_{j+1}-\ldots-\Gamma_i))=0
\end{equation}
because of the short exact sequence
$$0\to \cO_W(-\Gamma_{j+1}-\ldots-\Gamma_i)\to\cO_W\to\cO_{\Gamma_{j+1}\cup\ldots\cup\Gamma_i}\to0.$$
However, if $i>k>j$ and $P_k$ is 
not Gorenstein (i.e.~not an $A_s$ singularity)
then $\Ext^1(\cO_W(D_i),\cO_W(D_j))\ne0$.
This follows from \eqref{wegwEGweweg}, the local-to-global spectral sequence for $\Ext$ and  the fact that 
$\Ext^1_{\cO_{W,P_k}}(\cO_{W,P_k}(\Gamma_k+\Gamma_{k+1}),\cO_{W,P_k})\ne0$:

\begin{lemma}\label{wEGweg}
Let $(p\in Z)$ be a germ of a non-Gorenstein c.q.s. with toric boundary divisors $\Gamma, \Gamma'\in\Cl(p\in Z)$. Then $\Ext^1(\cO_{Z,p}(\Gamma+\Gamma'),\cO_{Z,p})\ne0$.
\end{lemma}

\begin{proof}
Let $R=\cO_{Z,p}$.
Note that $\Gamma+\Gamma'\sim-K_Z$ in $\Cl(p\in Z)$.
Since $-K_Z$ is not Cartier, 
$M=\cO_{Z,p}(\Gamma+\Gamma')$ is not a projective $R$-module. Thus we have a non-split
Auslander--Reiten exact sequence \cite{Auslander}
$0\to\tau(M)\to N_M\to M\to0$,
where $\tau(M)=(M\otimes\omega_Z)^{\vee\vee}=R$.
Thus $\Ext^1(M,R)\ne0$.
\end{proof}

Fortunately, orthogonality  holds for the ``reversed''  collection 
\begin{equation}\label{rgargr}
\cO_W(-D_r),\ \ldots, \cO_W(-D_1),\ \cO_W(-D_0),
\end{equation}
which therefore is essentially a unique choice
that works for  singular surfaces.
Indeed, the next result appears almost verbatim in \cite[Th.~2.12]{KKS}.

\begin{definition}
Let $Z$ be a projective variety.
A s.o.d. $D^b(Z)=\langle\cA_0,\ldots,\cA_r,\cB\rangle$ is called a {\em Kawamata decomposition} if every subcategory $\cA_i$ for $i=0,\ldots,r$ is classically generated by a maximal Cohen-Macaulay sheaf  and $\cB\subset D^{\perf}(Z)$.
\end{definition}

\begin{proposition}\label{kjshefjhksEFG}
For $i=0,\ldots,r$, let $A^W_i\subset D^b(W)$
be a triangulated subcategory classically generated by  $\cO_W(-D_i)$. Then we have a Kawamata decomposition
\begin{equation}\label{aeFafaEGaeg}
D^b(W)=\langle\cA^W_r,\ldots,\cA^W_0, \cB^W\rangle,
\end{equation}
so that $\cB^W\subset D^{\perf}(W)$,
as well as an s.o.d.
\begin{equation}D^b(W)=\langle \tilde\cB^W, \cA^W_r,\ldots,\cA^W_0\rangle\label{sdgsgs}
\end{equation}
with the property that $\D_W(B)$
is perfect\footnote{We denote by $\D_{\cX}=R\cHom(\cdot,\omega_\cX^\bullet)$ the  duality functor on a  noetherian scheme $\cX$.} for every $B\in\tilde\cB^W$.
\end{proposition}

\begin{remark}
A~coherent sheaf $\cF$ on a noetherian scheme $\cX$ is called maximal Cohen-Macaulay (mCM) if $\depth_{\cO_{\cX,x}}\cF_x=\dim\cO_{\cX,x}$ for every  $x\in\cX$ \cite{BrunsHerzog}. If~$\cF$ is mCM and $\cO_{\cX,x}$ is regular then $\cF_x$ is 
locally free (since it  has  projective dimension~$0$ by the 
Auslander--Buchsbaum formula). 
If $\cF$ is mCM then $\cF$ is reflexive,
which  implies that  
$\cF$ is torsion-free 
(these notions are equivalent to mCM in dimension $2$ and $1$, respectively).
If $\cF$ is mCM and $i:\,\cD\hookrightarrow \cX$ is an effective Cartier divisor
then $i^*\cF$ is mCM on $\cD$.
If $\cX$ is CM 
with a dualizing sheaf~$\omega$ then $\cF$ is mCM if and only if $\cExt^p(\cF,\omega)=0$ for $p>0$.
And then $\cHom(\cF,\omega)$ is  mCM.
\end{remark}

\begin{proof}[Proof of Prop.~\ref{kjshefjhksEFG}]
First we check orthogonality. Let $i>j$ and $p>0$.
For $k>j$, 
$$\cExt^p(\cO_{W,P_k}(-D_j),\cO_{W,P_k}(-D_i))\simeq
\cExt^p(\cO_{W,P_k},\cO_{W,P_k}(-D_i))=0$$
by obvious reasons and for $k<i$, 
$$\cExt^p(\cO_{W,P_k}(-D_j),\cO_{W,P_k}(-D_i))\simeq
\cExt^p(\cO_{W,P_k}(-D_j),\omega_{W,P_k})=0$$
because $\cO_{W}(-D_j)$ is reflexive 
and therefore mCM.
By \eqref{wegwEGweweg}, this implies  $\cO(-D_i)\subset \langle\cO(-D_j)\rangle^\perp$
and therefore $\cA_W^i\subset (\cA_W^j)^\perp$.
Next we borrow  analysis from \cite{KKS}.

Let $X$ be the resolution of singularities of $W$ with exceptional divisors $E_i^1+\ldots+E_i^{m_i}$ over each $P_i\in W$. 
Let $\Gamma_i^X\subset X$ be the proper transform of $\Gamma_i\subset W$ for $i=1,\ldots,r$
and let $D_0^X$ be the proper transform of $A$.
Define divisors $D_i^X$ for $i=1,\ldots,r$ inductively as follows:
$D_i^X=D_{i-1}^X+E_i^1+\ldots+E_i^{m_i}+\Gamma_i^X$. This gives
a s.o.d.
$D^b(X)=\langle\cB^X, \cA^X_r,\ldots,\cA^X_0\rangle$,
where
$$
\cA^X_i=\langle
\cO(-D^X_i-E_i^1-\ldots-E_i^{m_i}),\ldots,\cO(-D^X_i-E_i^1),\cO(-D^X_i)
\rangle.
$$
The pushforward of this s.o.d.~ to $W$ was studied in \cite{KKS}:

(1) $\cA^X_i$ is classically  generated by $\cO(-D^X_i)$ and sheaves $\cO_{E_i^l}(-1)$, $1\le l\le m_i$. Indeed,  $D^X_i\cdot E_i^1=1$ and $D^X_i\cdot E_i^j=0$,
so  for $k=1,\ldots,m_i$
we have  exact sequences 
$0\to \cO(-D^X_i-E_i^1-\ldots-E_i^k)\to
\cO(-D^X_i-E_i^1-\ldots-E_i^{k-1})\to\cO_{E_i^k}(-1)\to0
$. 

(2) By \cite[Th.~2.12]{KKS}, we have a s.o.d. \begin{equation}\label{wFwefweG}
    D^b(W)=\langle R\pi_*\cB^X, R\pi_*\cA^X_r,\ldots,R\pi_*\cA^X_0\rangle,
    \end{equation}
    where $R\pi_*\cB^X$ is equivalent to $\cB^X$ and $R\pi_*\cA^X_i$ is equivalent to the Verdier quotient
    $\cA^X_i/\langle\cO_{E_i^l}(-1)\rangle_{1\le l\le m_i}$. In particular, it is classically generated by $R\pi_*\cO(-D_i^X)$. 
    
(3) Set $\bar D_i^X=D_i^X+E_i^1+\ldots+E_i^{m_i}$. Then $R\pi_*\cO(-D_i^X)=R\pi_*\cO(-\bar D_i^X)$ and
    $$-\bar D_i^X\cdot E_j^k=\begin{cases}
      0, &       j>i\cr
      -1-(E_i^{m_i})^2, &j=i,\ k=m_i\cr
      -2-(E_j^k)^2, &\hbox{\rm otherwise.}\cr
    \end{cases}
    $$
    By \cite[Cor.~6.2]{KKS},  $R\pi_*\cO(-\bar D_i^X)=\pi_*\cO(-\bar D_i^X)$ is a reflexive sheaf $\cO_W(-D_i)$. This gives an s.o.d.~\eqref{sdgsgs}
    with $\tilde\cB^W=R\pi_*\cB^X$.

(4) Let $B\in \tilde\cB^W$ and write $B=R\pi_*B^X$, $B^X\in\cB^X$.
 Let $i:\,E\hookrightarrow X$ be a component of the exceptional divisor of $\pi$. By coherent duality,
$\D_X(Ri_*\cO_E(-1))=Ri_*\D_{\P^1}(\cO_E(-1))=Ri_*\cO_E(-1)[-1]$.
Since $\RHom(Ri_*\cO_E(-1),B^X)=0$ by (1) above, it follows that
$\RHom(\D_X(B^X),Ri_*\cO_E(-1))=0$. 
By \cite[Lemma 2.5]{KKS}, this implies that $\D_W(B)=R\pi_*\D_X(B^X)$ is perfect.

(5)
To finish the proof, consider divisors
$K_W+D_i=-\tilde A-\Gamma_r-\ldots-\Gamma_{i+1}$
for $i=0,\ldots,r$,
where $\tilde A$ is as in Assumption~\ref{assume}.
Arguing as above, we get an s.o.d.
\begin{equation}\label{sDGSDGSH}
D^b(W)=\langle \cB, \langle\omega_W(D_0)\rangle,\ldots,\langle\omega_W(D_r)\rangle\rangle,
\end{equation}
where $\D_W(B)$ is perfect for every $B\in\cB$. Applying 
the duality anti-equivalence $\D_W$ to the s.o.d. \eqref{sDGSDGSH} gives the Kawamata decomposition \eqref{aeFafaEGaeg}.
\end{proof}

\begin{remark}\label{wfwEGwetgwe}
The proof shows that duality functor $\D_W$ takes the s.o.d.
\eqref{aeFafaEGaeg}
to the s.o.d.~\eqref{sDGSDGSH}, where $\omega_W(D_i)\simeq\cO_W(-\tilde A-\Gamma_r-\ldots-\Gamma_{i+1})$ for every $r=0,\ldots,r$.
\end{remark}

\begin{corollary}\label{asgasrgsarhasr}
Let $\pi:\,W\to\oW$ be a c.q.s.~resolution of a surface with a single c.q.s.~$P$ satisfying Assumption \ref{assume} (1), (2), (3). 
Then s.o.d.'s~\eqref{aeFafaEGaeg}
and \eqref{sdgsgs} for $W$ and $\oW$ 
are compatible via $R\pi_*$, that is
$R\pi_*\langle\cA^W_r,\ldots,\cA^W_0\rangle=\cA^{\oW}$
and $\cB^W$ (resp.~$\tilde\cB^W$) is equivalent to 
$\cB^\oW$ (resp.~$\tilde\cB^\oW$) via $R\pi_*$. Also, $\cB^W=L\pi^*\cB^\oW$.
\end{corollary}

Assumption \ref{assume} (3) implies that proper transforms $\Gamma_0$ (resp.~$\Gamma_{r+1}$) of $\bar A\subset\oW$ 
(resp.~$\tilde{\bar A}\subset\oW$) in $W$ of $\oW$
intersect the chain $\Gamma_1\cup\ldots\cup\Gamma_r$
only at  $P_0$ (resp.~$P_r$),
where they are equivalent to  toric boundaries
opposite to $\Gamma_1$ (resp.~$\Gamma_r$).
So the corollary  follows from the proof 
of Prop.~\ref{kjshefjhksEFG} since the minimal resolution of $W$ is also a resolution of $\oW$
and the s.o.d.'s~\eqref{aeFafaEGaeg}
and \eqref{sdgsgs} defined using \eqref{wFwefweG} are compatible. It generalizes 
\cite[Th.~2.12]{KKS} (which is the special case of Corollary~\ref{asgasrgsarhasr} when $W=X$ is a smooth  resolution of $\oW$). It shows that we can view the category $\langle\cA^W_r,\ldots,\cA^W_0\rangle$
as a partial resolution of singularities of the category 
$\cA^\oW$.

\begin{definition}  
Let
 $\cZ\to B=\Spec\bB$ be a morphism to an affine scheme, 
 let $\cF$ be a coherent sheaf on $\cZ$, and let $\bR\subset\End(\cF)$ be a finite $\bB$-algebra. There is a functor of abelian categories
$\otimes_\bR\cF:\,
\bR\hbox{\rm -Mod}\to\Qcoh(\cZ)$, which  takes a $\bR$-module 
$M=\Coker(\bR^I\mathop{\rightarrow}\limits^\phi\bR^J)$ to $M\otimes_\bR\cF:=\Coker(\cF^I\mathop{\rightarrow}\limits^\phi\cF^J)$. We denote its left derived functor by  $\otimes^L_\bR\cF:\,
D(\bR\hbox{\rm -Mod})\to D(\Qcoh(\cZ))$ and the corresponding homological functors by $\Tor_{\bR}^j(\cdot,\cF):\,\bR\hbox{\rm -Mod}\to\Qcoh(\cZ)$.
We call $\otimes_\bR\cF$ {\em bounded} if, for every f.g.~ $\bR$-module $M$, $\Tor_{\bR}^j(M,\cF)=0$ for $j\gg0$. In this case $\otimes^L_\bR\cF$ induces a  functor 
 $\otimes^L_\bR\cF:\,
D^b(\bR\hbox{\rm -mod})\to D^b(\cZ)$.
We denote its essential image by $\{\cF\}\subset D^b(\cZ)$\footnote{Note that $\{\cF\}\ne\langle\cF\rangle$ in general because $\{\cF\}$ is not necessarily classically generated by~$\cF$.}.
\end{definition}

\begin{definition}
Let $\cD$ be a coherent sheaf on a projective variety $Z$.
An iterated extension of $\cD$ is defined recursively as either $\cD$ or a non-trivial extension of an 
iterated extension by $\cD$.
A universal iterated extension $\cD^p$ for $p\ge0$ 
is a coherent sheaf defined inductively
as $\cD^0=\cD$ and $\cD^p$ for $p>0$ is the universal extension 
\begin{equation}\label{aaegaEGaega}
0\to \cD\otimes_\C\Ext^1(\cD^{p-1},\cD)^\vee\to \cD^p\to\cD^{p-1}\to 0.
\end{equation}
An~iterated extension $\hat\cD$ is called maximal
if $\Ext^1(\hat\cD,\cD)=0$, see  \cite{Kaw}. If the maximal iterated extension  exists then it is unique and equal to $\cD^p$ for some $p\ge0$.
\label{maxitext}
\end{definition}

\begin{definition}\label{kjhfkjgfk}
For every $i=0,\ldots,n$, 
let $F_i$ be  the maximal iterated extension of $\O_W(-D_i)$, where the
Weil divisor $D_i$
on $W$ was defined in \eqref{EFefeGeEG}.
\end{definition}

\begin{lemma}
The sheaf $F_i$ exists and is locally free at every $p\in W$ except at $P_j$ for $j<i$, where $\D_W(F_i)$ is locally free.
The functor $\otimes_{R_i}F_i$ is bounded,
where $R_i=\End(F_i)$ is isomorphic to the 
Kalck-Karmazyn algebra of the germ  $(P_i\in W)$. In notation of Prop.~\ref{kjshefjhksEFG}, 
$$D^b(R_i-\mmod)\simeq
A^W_i=\langle 
\cO_W(-D_i)\rangle=\{F_i\}\subset D^b(W).$$
\end{lemma}

\begin{proof}
This is a summary of results from \cite{KKS,K21}.
\end{proof}

\begin{lemma}\label{asfhsrHs}
For $k>0$, $i\le j$,
\begin{enumerate}
\item $\cExt^k(F_i,F_j)=\cExt^k(F_i,\cO_W(-D_j))=0$;
\item
$\Ext^k(F_i,F_j)=\Ext^k(F_i,\cO_W(-D_j))=0$.
\end{enumerate}
\end{lemma}

\begin{proof}
$F_i$ and $\cO_W(-D_j)$
are locally isomorphic to $\omega_W$ tensored
with a  free sheaf at $P_k$ for $k<i$.
Since $F_i$ is locally free elsewhere,
$\cExt^k(F_i,\cO_W(-D_j))=0$ for $k>0$.
This implies the same vanishing for iterated extensions of $\cO_W(-D_j)$
including $F_j$.

Equation (2) for $i<j$ follows from Prop.~\ref{kjshefjhksEFG} (the s.o.d.). 
Let $\cD=\cO_W(-D_i)$,
$\cD^p=F_i$.
From (1), $\Ext^k(\cD^p,\cD)=H^k(W,\cHom (\cD^p,\cD))$, which is equal to $0$ if $k>2$.
Also, $\Ext^1(\cD^p,\cD)=0$ by the definition of the maximal iterated extension.

The vanishing of $\Ext^2(\cD^p,\cD)=H^2(W,\cHom (\cD^p,\cD))$ is equivalent, by Serre duality and reflexivity of $\cHom (\cD^p,\cD)$, to vanishing of 
$\Hom(\cHom (\cD^p,\cD),\omega_W)$.
We~argue by induction on $p$ using~\eqref{aaegaEGaega}.
The base of induction is vanishing of  
$\Hom(\cHom (\cD,\cD),\omega_W)=H^0(W,\omega_W)$, which holds by Assumption \ref{assume} (2).
Applying $\cHom(\cdot,\cD)$ to \eqref{aaegaEGaega} gives an exact sequence
\begin{equation}\label{qrgwgwwrfwe}
    0\to \cHom (\cD^{j-1},\cD)\to \cHom (\cD^j,\cD)\to \cO_W\otimes \Ext^1(\cD^{j-1},\cD)\to A\to 0,
\end{equation}
where $A$ is a sheaf supported at $P_i$. Since $W$ is a CM surface, $\Ext^1(A,\omega_W)=0$.
Splitting \eqref{qrgwgwwrfwe} into two short exact sequences and using again that $H^0(W,\omega_W)=0$
shows that $\Hom(\cHom (\cD^{j-1},\cD),\omega_W)=0$
implies $\Hom(\cHom (\cD^j,\cD),\omega_W)=0$.

Finally, 
vanishing of $\Ext^k(\cD^p,\cD)$  for $k>0$
implies vanishing of $\Ext^k(\cD^p,\cD^j)$  by induction on $j$, including vanishing of $\Ext^k(F_i,F_i)$ for $k>0$.
\end{proof}

Finally, we address ``locality'' of subcategories $\cA^W_i=\langle\cO_W(-D_i)\rangle\subset D^b(W)$, i.e.~there dependence on the divisor $D_i$.

\begin{lemma}\label{wefeqf}
There is a short exact sequence 
$$0\to \Pic(W)\to\Cl(W)\mathop{\longrightarrow}^\gamma \mathop{\oplus}\limits_{i=0}^r\Cl(P_i\in W)\to0,$$ where $\Cl(P_i\in W)$ is the local class group 
and $\gamma$ is the restriction.
\end{lemma}

\begin{proof}
It suffices to prove surjectivity of $\gamma$. Each $\Cl(P_i\in W)$
is a finite cyclic group generated by any one of the toric boundaries. The divisor  $D_i$ restricts to  a generator in $\Cl(P_i\in W)$ and to $0$ in $\Cl(P_j\in W)$ for $j>i$.
Arguing by induction on~$i$, 
some linear combination $D_i+\sum_{j<i}n_jD_j$ with integer coefficients restricts to a generator in $\Cl(P_i\in W)$ and to $0$ in $\Cl(P_j\in W)$ for $j\ne i$.
\end{proof}

\begin{figure}[htbp]
\center
\includegraphics[width=10cm]{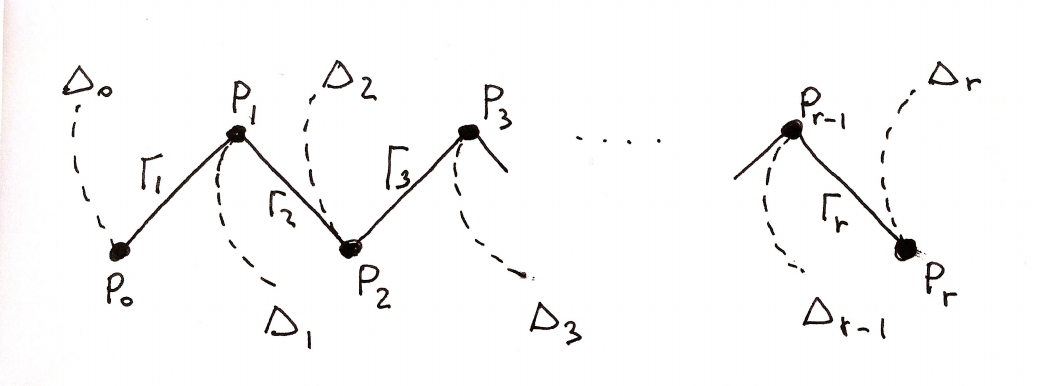}
\caption{Divisors $\Delta_0,\ldots,\Delta_r$ on $W$}
\label{fasfg}
\end{figure}

\begin{definition}\label{wegwRGwrg}
Fix $q=0,\ldots,r$. 
By Lemma~\ref{wefeqf}, there exists
a Weil divisor  $\Delta_q$ which
is equivalent to $D_q$ in $\Cl(P_q\in W)$ and is Cartier at $P_i$ for $i\ne q$, 
see Figure~\ref{fasfg}.
Let $\omega_W^q=\cO_W(-D_q+\Delta_q)$. For  $i=q,\ldots,r$, let $D_{i,q}=\Delta_q+\Gamma_{q+1}+\ldots+\Gamma_i$.
Let $\cA_{i,q}^W\subset D^b(W)$ be a full triangulated subcategory classically generated by 
$\cO_W(-D_{i,q})$. 
\end{definition}

\begin{corollary}\label{WEGsgsRHg}
    There is a s.o.d. $D^b(W)=\langle \bar\cB_q^W, \bar\cA^W_{r,q},\ldots \bar\cA^W_{q,q}\rangle$
such that $\D_W(B)$ is perfect at $P_i$ for $i\ge q$
for every  $B\in\bar\cB_q^W$ and
$\cA_{i,q}^W=\{F_{i,q}\}\simeq D^b(R_{i,q}-\mmod)$,
where $F_{i,q}$ is the maximal iterated extension of $\cO_W(-D_{i,q})$
and $R_{i,q}=\End(F_{i,q})$ is isomorphic to the 
Kalck-Karmazyn algebra of $(P_i\in W)$. 
Furthermore, categories $\langle \cA^W_{r,q},\ldots \cA^W_{q,q}\rangle$ and $\langle\cA^W_r,\ldots,\cA^W_q\rangle$  are equivalent via an adjoint pair of functors  
$\bullet\otimes^L\omega^q_W$ and $R\cHom(\omega^q_W,\bullet)$.
\end{corollary}

\begin{proof}
Existence of the s.o.d. is proved as in Prop.~\ref{kjshefjhksEFG}. We can ignore singularities at $P_0,\ldots,P_{q-1}$, where the sheaves
$\cO_W(-\Delta_q-\Gamma_{q+1}-\ldots-\Gamma_i)$ are Cartier.

Note that $\omega^q_{W,P_i}\simeq\omega_{W,P_i}$ for $i<q$ and is an invertible sheaf elsewhere. In~particular, $R\cHom(\omega^q_W, \omega^q_W)\simeq\cO_W$.
Since $\cO_W(-D_{i,q})$ and $\omega^q_W$ are non-Cartier at disjoint subsets of points, we have $\cO(-D_{i,q})\otimes^L\omega^q_W=\cO(-D_{i,q})\otimes\omega^q_W=\cO(-D_i)$.
Likewise,  we have
$R\cHom(\omega^q_W,\cO(-D_i))=\cHom(\omega^q_W,\cO(-D_i))=\cO(-D_{i,q})$.
Indeed, for $p>0$, $\cExt^p(\omega^q_W,\cO(-D_i))=0$ at $P_k$ for $k\ge q$ because $\omega^q_W$ is invertible there but also for $k<q$ because $\cO_{W,P_k}(-D_i)\simeq\omega_{W,P_k}$ there and $\omega^q_W$ is reflexive.

Let $F_{i,q}$ be the maximal iterated extension of $\cO(-D_{i,q})$.
Adjointness of 
$\bullet\otimes^L\omega^q_W$ and $R\cHom(\omega^q_W,\bullet)$ implies  
$\Ext^1(\cO(-D_i),\cO(-D_i))=\Ext^1(\cO(-D_{i,q}),\cO(-D_{i,q}))$. 
Arguing by induction on iterated extensions, this gives 
$F_{i,q}\otimes^L\omega^q_W=F_i$ and 
$R\cHom(\omega^q_W,F_i)=\cHom(\omega^q_W,F_i)=F_{i,q}$.
In particular, $\cA^W_{i,q}$ is equivalent to $\cA^W_{i}$
and the algebras $R_{i,q}=\End(F_{i,q})$ and $R_{i}=\End(F_i)$ and 
are isomorphic.
\end{proof}

\section{Derived category of a deformation of a  c.q.s.~surface}\label{AEFefEEfe}

We use notation of Section~\ref{s2}: $W$ is a c.q.s. surface satisfying Assumption \ref{assume} (1), (2), (3).
Weaker conditions of Notation~\ref{asrgasrgarh}
will be sufficient except in Corollary~\ref{EFegwEWE}.
For every $i=0,\ldots,r$, 
let $F_i$ be the Kawamata sheaf on $W$,
see Definition~\ref{kjhfkjgfk}.

\begin{lemdef}\label{sFsefwetwet}
Let $\cW$ be a projective deformation
of $W$ over a smooth affine variety~$B=\Spec\B$.
For  $i=0,\ldots,r$, possibly after shrinking~$B$,
there exists a unique  coherent sheaf $\cF_i$ on $\cW$
flat over $B$  such that $\cF_i|_W\simeq F_i$.
We call $\cF_i$ a Kawamata sheaf.
\end{lemdef}

\begin{proof}
Existence and uniqueness of $\cF_i$ follow from Lemma~\ref{asfhsrHs} and  deformation theory of coherent sheaves. 
\end{proof}

\begin{lemma}\label{qrgsrhr}
Kawamata sheaves have the following properties (after shrinking $B$):
\begin{enumerate}
\item 
$\cF_i$ is locally free at every $p \in W$ except at $P_j$ for $j<i$, where $\D_\cW(\cF_i)$ is locally free.
In particular, the restriction of $\cF_i$ to every smooth fiber $Y$ of $\cW\to B$ is locally free.
\item $\cExt^p(F_i,F_j)=0$ for $p>0$, $i\le j$;
\item
$\Ext^p(\cF_i,\cF_i)=0$ for $p>0$ and
$\Ext^p(\cF_i,\cF_j)=0$ for all $p$ and $i<j$.
\item
$\bR_i=\End(\cF_i)$ is a finite projective $\bB$-module and $R_i=\bR_i\otimes_\bB k(0)$.
\item $\otimes_{\bR_i}\cF_i$ is bounded. 
\end{enumerate}
\end{lemma}

\begin{proof}
Locally free sheaves deform to locally free sheaves.
Since $R\cHom(\omega_W,\omega_W)=\cO_W$,
the sheaf $\omega_W$ deforms locally uniquely. Since $\omega_\cW$ is flat over $B$, the deformation of $\omega_W$ must be locally isomorphic to $\omega_{\cW}$ by adjunction. It~follows that $\cF_i$
is locally isomorphic to $\omega_\cW$ tensored with a locally free sheaf at $P_j$
for $j<i$. Equivalently, 
we have (1), which immediately implies (2).

By (2),  $\Ext^p(\cF_i,\cF_j)=H^p(\cW,\cHom(\cF_i,\cF_j))$ for 
$i<j$ and all $p$. Thus (3)  follows from Lemma~\ref{asfhsrHs},  cohomology and base change
\cite[Tag 0AA7]{Stacks} and Nakayama's lemma. This also proves (4).


It remains to prove (5). We drop the index $i$.
By \cite{ELS}, the category $D^b(\bR\hbox{\rm -mod})$ has a classical generator, without loss of generality 
a $\bR$-module  $\hat M$. It suffices to prove that $\hat M\otimes_\bR^L\cF$ has bounded cohomology.
A free $\bR$-resolution $P^\bullet\to\hat M$
can also be viewed as a projective $\bB$-resolution.
Since $\cO_{B,0}$ is a regular local ring,  
$\Tor_{\bB}^j(\hat M, k(0))=0$ for $j> \dim B$.
Thus $P^\bullet\otimes_\bB k(0)=P^\bullet\otimes_\bR R$ has bounded cohomology, i.e.~gives an object in $D^b(R-\mmod)$. Since $\otimes_{R} F$ is bounded, it follows that 
$(P^\bullet\otimes_\bB k(0))\otimes^L_RF$ has bounded cohomology, and so
$\Tor_{\bR}^j(\hat M,F)=0$ for $j\ge j_0$.
Since $\cF$ is flat over $B$, this implies that 
$\Tor_{\bR}^j(\hat M,\cF)\otimes_\bB k(0)=0$ for $j\ge j_0$ 
and thus 
$\Tor_{\bR}^j(\hat M,\cF)=0$ for $j\ge j_0$ by Nakayama's lemma (after shrinking $B$).
\end{proof}

We also need an analogue of Kawamata sheaves
for chains of curves in $W$ that start  not at $P_0$ but at some  point $P_q$, $q>0$.
As in Definition~\ref{wegwRGwrg},
we take a Weil divisor  $\Delta_q$ which
is equivalent to $D_q$ in $\Cl(P_q\in W)$ and is Cartier at $P_i$ for $i\ne q$
(see Figure~\ref{fasfg}). Let 
$F_{i,q}$ be the maximal iterated extension of $\cO_W(-\Delta_q-\Gamma_{q+1}-\ldots-\Gamma_i)$
for $i=q,\ldots,r$ (see Corollary~\ref{WEGsgsRHg} for its properties).

\begin{lemma}\label{qrgwGwhwH} 
Fix $q=0,\ldots,r$. 
\begin{enumerate}
    \item 
There exist unique mCM sheaves $\omega^q_{\cW}$ and $\cF_{i,q}$ (for $i\ge q$) on $\cW$ flat over $B$ that restrict to sheaves $\omega^q_{W}$ and $F_{i,q}$ on $W$. 
\item The sheaves $\cF_{i,q}$ satisfy (1)--(5) of Lemma~\ref{qrgsrhr} (ignoring singularities at $P_0,\ldots,P_{q-1}$). 
\item 
We have 
$M^\bullet\otimes^L_{\bR_i}\cF_i\simeq
(M^\bullet\otimes^L_{\bR_i}\cF_{i,q})\otimes^L\omega^q_{\cW}$
for every $M^\bullet\in D^b(\bR_i-\mmod)$.
\item
$\RHom(\cF_i, M^\bullet\otimes^L_{\bR_i} \cF_i)\simeq M^\bullet$
for every $M^\bullet\in D^b(\bR_i-\mmod)$.
\end{enumerate}
\end{lemma}

\begin{proof}
(1) and (2) are proved in the same way as  
Lemmas~\ref{sFsefwetwet} and \ref{qrgsrhr}, 
respectively.

It is clear that 
$\cF_{i}\simeq\cF_{i,q}\otimes^L\omega^q_{\cW}\simeq\cF_{i,q}\otimes\omega^q_{\cW}$
and $\bR_i=\End(\cF_{i})\simeq\End(\cF_{i,q})$. Thus (3) follows from  associativity of the derived tensor product.

By Lemma~\ref{qrgsrhr}, $\RHom(\cF_i,\cF_i)=\bR_i$ and $\RcHom(\cF_i,\cF_i)=\cHom(\cF_i,\cF_i)$. 
 Let $P^\bullet$ be a bounded above complex of free $\bR_0$-modules quasi-isomorphic to $M^\bullet$. Suppose first that $i=0$.
 Then $\cF_0$ is locally free and 
$$\RcHom(\cF_0, M^\bullet\otimes^L_{\bR_0} \cF_0)\simeq
\cF_0^\vee\otimes(P^\bullet\otimes_{\bR_0} \cF_0)\simeq
P^\bullet\otimes_{\bR_0}\cHom(\cF_0,  \cF_0).
$$
Applying $R\Gamma$ to both sides gives a desired isomorphism.
The general case $i=q$ follows by the same calculation applied to a locally free sheaf $\cF_{q,q}$ using (3).
\end{proof}

\begin{definition}
Let $\cA_i^\cW=\{\cF_i\}\subset D^b(\cW)$ for $i=0,\ldots,r$.
\end{definition}

\begin{theorem}\label{asfbsfbsf}
Let $\cW$ be a projective deformation
of $W$ over a smooth affine variety~$B$.
After shrinking $B$,  $\cA_i^\cW\simeq D^b(\bR_i-\mmod)$  and we have $B$-linear s.o.d.'s~
\begin{equation}\label{qsfesfsrgasrg}
D^b(\cW)=\langle \cA_r^\cW,\ldots,\cA_0^\cW, \cB^\cW\rangle=
\langle \tilde\cB^\cW,\cA_r^\cW,\ldots,\cA_0^\cW\rangle
\end{equation}
that restrict to~\eqref{aeFafaEGaeg} and \eqref{sdgsgs}
on $W$.
Furthermore,
$\cB^\cW$ and $\D_\cW(\tilde\cB^\cW)\subset D^{\perf}(\cW)$.
\end{theorem}

\begin{proof}
We  will prove by induction on $k$ the following statements:
(1)
$D^b(\cW)$ has a $B$-linear s.o.d. $\langle \tilde\cB_k^\cW,\cA_k^\cW,\ldots,\cA_0^\cW\rangle
$; 
(2)
$\cA_k^\cW\simeq D^b(\bR_k-\mmod)$;
(3) $\cA_k^\cW$ restricts to $\cA_k^W$ on $W$; 
(4)  
$\D_\cW(B)$ is perfect at $P_0,\ldots,P_k$ for every $B\in\tilde\cB_k^\cW$. 
Suppose this is known for  $k<i$, we will prove it for $k=i$ (for $k=-1$ there is nothing to prove). 

First we claim that $\cA_i^\cW\subset \tilde\cB_{i-1}^\cW$. Fix $q<i$. Since objects in $\cA_q^\cW$ are represented by bounded above complexes with components isomorphic to $\cF_q^{\oplus n}$ for some $n$, 
by the spectral sequence for $\Ext$ \cite[Tag 07A9]{Stacks} 
it suffices to prove that $\Hom(\cF_q,T)=0$ for every $T\in \cA_i^\cW$.
For $q=0$, we argue as follows. $T$ can be represented by a bounded above complex with bounded cohomology and with components isomorphic to direct sums of~$\cF_i$. By Lemma~\ref{qrgsrhr},
$\RHom(\cF_q,\cF_i)=0$. So we can use naive truncations to reduce to the case of $T\simeq G[k]$, where $G$ is a sheaf and $k\gg0$.
Since $\cF_0$ is locally free, $\cExt^l(\cF_q,G)=0$ for $l>0$. 
Thus $\Hom(\cF_q,T)=\Ext^k(\cF_q,G)=H^k(\cW,\cHom(\cF_q,G))=0$ since $k\gg0$.
For general $q$, we use Lemma~\ref{qrgwGwhwH} and write
$\cF_q\simeq \cF_{q,q}\otimes^L\omega^q_{\cW}$ and $T\simeq T_q\otimes^L\omega^q_{\cW}$ with  $T_q\in \{\cF_{i,q}\}$.
Thus $\RHom(\cF_q,T)=\RHom(\cF_{q,q},T_q)=0$ as above because
$\cF_{q,q}$ is locally free.

The functor  $\tilde\Phi=\otimes^L_{\bR_i}\cF_i:\,
D(\bR_i-\Mod)\to D(\Qcoh(\cW))$ commutes with arbitrary direct sums, and therefore has a right adjoint functor
 $\tilde\Psi:\,D(\Qcoh(\cW))\to D(\bR_i-\Mod)$
by \cite{N96}.
By adjunction,
we have
$$\tilde\Psi(T)\simeq \RHom(\bR_i, \tilde\Psi(T))
\simeq \RHom(\tilde\Phi(\bR_i), T)\simeq\RHom(\cF_i, T).$$
We claim that $\tilde\Psi$ induces a functor
$\cB_{i-1}^\cW\to D^b(\bR_i-\mmod)$. Indeed, 
let $T\in\cB_{i-1}^\cW$. It~suffices to prove boundedness of $\RcHom(\cF_i, T)$. After shrinking $B$, it suffices to prove boundedness at every point $p\in W$. At $p\ne P_0,\ldots,P_{i-1}$,
$\cF_i$ is locally free and boundedness is clear.
At one the remaining points $P_j$, $\cF_i$ is locally a deformation of $\omega_W^{\oplus s}$ for some $s$,
and therefore is locally isomorphic to $\omega_{\cW}^{\oplus s}$. On the other hand, $\D_{\cW}(T)$ is perfect at $P_j$, and so $T$ is locally isomorphic to $\omega_{\cW}\otimes^LS$, where $S$ is a perfect complex. Thus $\RcHom(\cF_i, T)$ is bounded at $P_j$ as well.

By Lemma~\ref{qrgwGwhwH} (4), we have
$$\Psi(\Phi(M^\bullet))\simeq \RHom(\bR_i, \Psi(\Phi(M^\bullet)))\simeq\RHom(\cF_i, M^\bullet\otimes^L_{\bR_i} \cF_i)\simeq M^\bullet.$$ Thus the adjunction morphism $M^\bullet\to\Psi(\Phi(M^\bullet))$ 
is an isomorphism. Therefore, $\Phi$ is fully
faithful,  $\cA_i^{\cW}$
is a right admissible subcategory of $\cB^\cW_{i-1}$ equivalent to the category $D^b(\bR_i-\mmod)$, and we have a s.o.d.
$\langle \tilde\cB_i^\cW,\cA_i^\cW,\ldots,\cA_0^\cW\rangle
$. 

We need to prove that the s.o.d.~$\langle \tilde\cB_i^\cW,\cA_i^\cW,\ldots,\cA_0^\cW\rangle$
is $B$-linear and restricts to the s.o.d.~$\langle \tilde\cB_i^W,\cA_i^W,\ldots,\cA_0^W\rangle$ on $W$.
It is enough to show that
$\cA_i^\cW$ is $B$-linear and restricts to $\cA_i^W$.
Equivalently, 
it suffices to show that $D^b(\bR_i-\mmod)$ is $\B$-linear, which is clear, 
and also that it restricts to $D^b(R_i-\mmod)$ at $0\in B$. Indeed, 
since $B$ is smooth, a $\B$-module $k(0)$ is resolved by a finite Koszul complex (after shrinking~$B$). Thus the functor $\otimes_{\B}k(0)$ sends
$D^b(\bR_i-\mmod)$ to $D^b(R_i-\mmod)$. Moreover, the same Koszul complex shows that every
$M^\bullet\in D^b(\bR_i-\mmod)$ is a direct summand of 
$\tilde M^\bullet\otimes^Lk(0)$, where $\tilde M^\bullet$
is the restriction of scalars to $\bR_i$.

Next we claim that 
$\D_\cW(B)$ is perfect at $P_i$ for every $B\in\tilde\cB_i^\cW$. 
It suffices to prove that $\D_\cW(B)\otimes^Lk(P_i))$ has bounded cohomology.
Since $i_W:\,W\hookrightarrow\cW$ is a regular embedding, 
$Li_W^*\D_\cW(B)\in D^b(W)$, so it suffices to prove
that $Li_W^*\D_\cW(B)$ is perfect.
Since $i_W:\,W\hookrightarrow\cW$ is a regular embedding,
$Li_W^*\D_\cW(B)=\D_W(Li_W^*(B))$ (up to a shift),
so it suffices to prove that the latter is perfect.
Since $Li_W^*(B)\in \tilde\cB_i^W$, this follows from 
Prop.~\ref{kjshefjhksEFG}. Finally, the fact that 
$\cB^\cW\subset D^{perf}(\cW)$ is proved in the same way using 
Prop.~\ref{kjshefjhksEFG} and Remark~\ref{wfwEGwetgwe}.
\end{proof}


The next corollary provides categorification for ``blowing down deformations''.

\begin{corollary}\label{EFegwEWE}
Let $W$ be a c.q.s.~resolution of a surface $\oW$ with a single c.q.s.~$P$ satisfying Assumption \ref{assume} (1), (2), (3).
Let $\cW$ be a projective deformation
of $W$ over a smooth affine variety~$B$. After shrinking $B$,
there is a morphism $\pi:\,\cW\to\ocW$ to a projective deformation of $\oW$.
The s.o.d.'s~\eqref{qsfesfsrgasrg}
for $\cW$ and $\ocW$ 
are compatible, namely
$R\pi_*\langle\cA^{\cW}_r,\ldots,\cA^{\cW}_0\rangle=\cA^{\ocW}$
and $\cB^\cW$ (resp.~$\tilde\cB^\cW$) is equivalent to 
$\cB^{\ocW}$ (resp.~$\tilde\cB^{\ocW}$) via $(R\pi_*, L\pi^*)$.
\end{corollary}

\begin{proof}
For blowing-down deformations, see \cite{Wahl}.
Let $j:\,W\hookrightarrow\cW$ be the (regular) embedding.
Choose classical generators $A_i\in\cA_i^\cW$, $\bar A\in\cA^{\ocW}$, $B\in\cB^\cW$, $\bar B\in\cB^{\ocW}$. 

By Assumption \ref{assume} (3),
proper transforms $\Gamma_0$ (resp.~$\Gamma_{r+1}$) of $\bar A\subset\oW$ 
(resp.~$\tilde{\bar A}\subset\oW$) in $W$ of $\oW$
intersect the chain $\Gamma_1\cup\ldots\cup\Gamma_r$
only at  $P_0$ (resp.~$P_r$),
where they are equivalent to  toric boundaries
opposite to $\Gamma_1$ (resp.~$\Gamma_r$).
So we can apply Corollary~\ref{asgasrgsarhasr}.
For every $i=0,\ldots,r$,  
$\RHom(Lj^*L\pi^*\bar B, Lj^*A_i)=0$ by Corollary~\ref{asgasrgsarhasr} and $\RHom(L\pi^*\bar B, A_i)$ has bounded cohomology since  $L\pi^*\bar B$ is perfect. Thus  $\RHom(L\pi^*\bar B, A_i)=0$ by cohomology and base change and Nakayama's lemma (after shrinking $B$).
It follows that $L\pi^*(\cB^{\ocW})\subset \cB^{\cW}$ and, by adjunction,
$R\pi_*\cA^{\cW}_i\subset\cA^{\ocW}$. Next, 
$\RHom(Lj^*R\pi_* B, Lj^*\bar A)=
\RHom(R\pi_*Lj^* B, Lj^*\bar A)=0$ by Corollary~\ref{asgasrgsarhasr}.
Furthermore, $\RHom(R\pi_* B, \bar A)=\RHom(B, \pi^!\bar A)$ is bounded above since $B$ is perfect.
Thus, $\RHom(R\pi_* B, \bar A)=0$ by cohomology and base change and Nakayama's lemma (after shrinking $B$).
It follows that $R\pi_*(\cB^{\cW}) \subset \cB^{\ocW}$.
Since $\pi$ has fibers of dimension at most $1$, $R\pi_*(D^b(\cW))=D^b(\ocW)$ by \cite[Cor.~2.5]{KuzSex}.
Therefore, 
$R\pi_*\langle\cA^{\cW}_r,\ldots,\cA^{\cW}_0\rangle=\cA^{\ocW}$
and $R\pi_*(\cB^{\cW})=\cB^{\ocW}$. Since $R\pi_*{\cO_{\cW}}=\cO_{\ocW}$,
$R\pi_*L\pi^*$ gives the identity functor on $\cB^{\ocW}$ by projection formula. It remains to show that $L\pi^*:\,\cB^{\ocW}\to\cB^{\cW}$
is essentially surjective. 
Indeed, 
the unit of adjunction $L\pi^*R\pi_*B\to B$ is an isomorphism
for any $B\in \cB^\cW$ by Nakayama's lemma since its restriction to $W$
is an isomorphism by Corollary~\ref{asgasrgsarhasr}.
\end{proof}

\begin{remark}\label{wEFegEWG}
We focus on situations when $\pi$ induces an isomorphism on generic fibers $Y\simeq \oY$ of deformations, e.g.~ when $\oY$ is a smoothing of $\oW$.  Corollary \ref{EFegwEWE} shows that then $\cA^{\oW}$ deforms to a category
$\cA^{\oY}$, which has an  s.o.d.~ $\langle\cA^{Y}_r,\ldots,\cA^{Y}_0\rangle$
related to the fact that $\cA^Y=\cA^\oY$
is also a deformation of $\langle\cA^{\cW}_r,\ldots,\cA^{\cW}_0\rangle$.
\end{remark}

The following theorem completes the proof of Theorem~\ref{mainKawa}.

\begin{theorem}\label{sgsGsrgsRHsrh}
In the notation of Theorem~\ref{asfbsfbsf},
suppose $W$ has Wahl singularities. Let $B^0\subset B$ be the locus of smooth fibers~ $Y$ of the deformation and suppose $B^0\ne\emptyset$. Let the s.o.d.~$D^b(Y)=\langle \cA^Y_r,\ldots,\cA^Y_0, \cB^Y\rangle$ be the base change of 
$D^b(\cW)=\langle  \cA^{\cW}_r,\ldots,\cA^{\cW}_0,\cB^{\cW} \rangle$.
After shrinking $B$, for $i=0,\ldots,r$, the restriction of the Kawamata bundle
$\cF_i|_Y$ is isomorphic to $E_i^{\oplus n_i}$
and $\cA^{Y}_i=\langle E_i\rangle$, where $E_i$ 
is a Hacking exceptional vector bundle with
$$\rk E_i=n_i,\quad
c_1(E_i)=-n_i(A+\Gamma_1+\ldots+\Gamma_i) \in H_2(Y).$$
\end{theorem}

\begin{remark}
By Riemann--Roch, we also have
$c_2(E_i)=\frac{n_i-1}{2n_i}(c_1(E_i)^2 + n_i +1)$.
\end{remark}

A direct corollary is the Hacking exceptional collection (H.e.c.) announced in the introduction, which for dual  bundles
was stated in \cite{H16} (without proof).

\begin{corollary}\label{RHrh}
Hacking vector bundles $E_r,\ldots,E_0$ on $Y$  form an exceptional collection.
\end{corollary}

\begin{proof}[Proof of Theorem~\ref{sgsGsrgsRHsrh}]
This is essentially \cite[Th.~ 4.3]{K21} applied to each Kawamata bundle $\cF_i$. More precisely, 
choose $y\in B^0$ and let $(0\in\Delta)\to(0\in  B)$ be a curve passing through~$y$. For inductive reasons, we allow $\Delta$ to be singular at $0$. By \cite[Th.~ 4.3]{K21}, after possibly shrinking $\Delta$, there exists a finite morphism  $u:\,(0\in\Delta')\to(0\in\Delta)$ such that $u^{-1}(0)=0$, $\Delta'$ is smooth, and $\cW|_{\Delta'\setminus\{0\}}$ carries a relatively exceptional Hacking vector bundle $\cE_i$ 
such that  $\cF_i\simeq E_i^{\oplus n_i}$ on $\cW|_{\Delta'\setminus\{0\}}$.
We claim that the bundle $\cE_i$ and this isomorphism descend to $\cW|_{\Delta\setminus\{0\}}$.
Indeed, otherwise $Y$ carries non-isomorphic exceptional Hacking bundles $E_i$ and $E'_i$ such that
$\cF_i\simeq E_i^{\oplus n_i}\simeq {E'_i}^{\oplus n_i}$. So $E_i'\in\langle E_i\rangle$, which is impossible, because $\langle E_i\rangle$ contains only one exceptional object (up to a shift). Since $E_i$ is exceptional, it deforms uniquely to relatively exceptional vector bundle $\cE_i$ on  $\cW|_U$ for some open subset $U\subset B^0$. It follows that 
$F_i|_Y$ deforms to $\cE_i^{\oplus{n_i}}$ over $U$.
But $\Ext^p(F_i|_Y,F_i|_Y)=0$ for $p>0$, so this deformation must be equal to $\cF_i$ over $U$.
If $U=B^0$ (after shrinking $B$) then we are done.
If not, choose the next curve $\Delta$ intersecting $B^0$ and contained in $B\setminus U$ and argue by Noetherian induction.
\end{proof}

In the next lemma we consider a ``hybrid'' sitiation when the 
$\Q$-Gorenstein deformation smoothens some points (just one to simplify notation), is locally isotrivial around others, and the chain of rational curves connecting them also deforms.

\begin{lemma}\label{argadrhadh}
Suppose  $W$ has a Wahl singularity at $P_i$. Let $\cW\to B$ be a $\Q$-Gorenstein deformation
over a smooth curve which is locally isotrivial at $P_j$ for $j\ne i$, smoothens $P_i$, and such that the chain 
$A+\Gamma_1+\ldots+\Gamma_r$ on~$W$ deforms to the chain 
$A^Y+\Gamma_1^Y+\ldots+\Gamma^Y_{i-1}+\Gamma^Y_{i+1}+\ldots+\Gamma^Y_r$ on 
a general fiber~$Y$. More precisely, $\Gamma_{i-1}+\Gamma_{i}\subset W$ deforms to 
$\Gamma^Y_{i+1}\subset Y$ connecting $P_{i-1}$ and $P_{i+1}$.
Take the s.o.d. 
$D^b(\cW)=\langle \cA^{\cW}_r,\ldots,\cA^{\cW}_0, \cB^{\cW} \rangle$ 
as in Theorem~\ref{asfbsfbsf} and its restriction to~ $Y$,
$D^b(Y)=\langle\cA_r^Y,\ldots,\cA^Y_0, \cB^Y \rangle$.
Then $\cA^Y_j=\{ F^Y_j \}$ for $j\neq i$ and $\cA^Y_i=\langle E_i\rangle$,
where $E_i$ is the Hacking vector bundle associated with a singularity $P_i\in W$ as in Th.~\ref{sgsGsrgsRHsrh} and 
$F^Y_j$ for $j\ne i$ are Kawamata sheaves on $Y$ associated with 
singularities  $P_j\in Y$ as in Prop.~\ref{kjshefjhksEFG}.
\end{lemma}

\begin{proof}
$\cA^Y_i=\langle E_i\rangle$ by the argument of
Th.~\ref{sgsGsrgsRHsrh}. We just have to analyze $\cA^Y_j$ for $j\ne i$,
i.e.~prove that the Kawamata sheaf $F_j$ of $W$ deforms to the Kawamata sheaf $F_j^Y$, i.e.~$\cF_j|_Y=F_j^Y$. There is a divisorial sheaf $\cD_j$ on $\cW$
flat over $B$ and such that $\cD_j|_W=\O_W(-A-\Gamma_1-\ldots-\Gamma_j)$
and $\cD_j|_{Y}:=\O_Y(-A^Y-\Gamma^Y_1-\ldots-\Gamma^Y_j)$.

The sheaf $F_j$ (resp., $F_j^{Y}$) is a maximal iterated extension of  $\cD_j|_W$ (resp., $\cD_j|_{Y}$). 
Furthermore,
$\cExt^p(\cF_j|_Y,\cD_j|_Y)=0$ for $p>0$ and 
$\RHom(\cF_j|_Y,\cD_j|_Y)=\C$ 
by Lemma~\ref{asfhsrHs},  cohomology and base change, and Nakayama's lemma. 
It follows that $\RHom(\cF_j|_{Y},F_j^{Y})=\C^{\Delta_j}$.
Both $\cF_j|_{Y}$, and $F_j^{Y}$ are vector bundles on $Y$
of rank $\Delta_j$ except at points $P_k$ for $k<j$, where 
they are locally isomorphic to $\omega_Y^{\oplus \Delta_j}$.
Since $F_j^{Y}$ has a filtration with quotients isomorphic to $\cD_j|_Y$,
and $\Hom(\cF_j|_Y,\cD_j|_Y)=\C$, it follows that there
exists a generically surjective morphism $\psi \colon \cF_j|_{Y} \to F_j^{Y}$. As both are torsion-free sheaves of the same rank, we have $\ker(\psi)=0$. Thus we have a short exact sequence $0\to\cF_j|_{Y} \to F_j^{Y}\to G\to0$. But
$\cF_j|_{Y}$ and $F_j^{Y}$ have the same Chern character 
$\Delta_j\ch(\cD_j|_Y)$. Thus $\ch(G)=0$, and therefore $G=0$.
\end{proof}

\section{Comparing braid group actions: mutations and antiflips} \label{s4}

According to Theorem~\ref{braidrelations}, the braid group $B_{r+1}$ on $r+1$ strands acts on the set of Wahl resolutions $W\to\oW$ of a fixed c.q.s. surface $\oW$ with $r+1$ singularities by antiflips 
of curves $\Gamma_i$ contained in the exceptional divisor. 
These Wahl resolutions 
admit $\Q$-Gorenstein smoothings 
to the surfaces within the same irreducible component of the versal deformation space of $\oW$. 
According to Corollary~\ref{RHrh},
each Wahl resolution $W$ gives rise to an exceptional collection of Hacking vector bundles on its $\Q$-Gorenstein smoothing.  Since Hacking bundles deform uniquely, in fact $W$ gives a H.e.c.
on all sufficienly general smoothings of $\oW$
within a given irreducible component of its versal deformation space.
In this section we show that antiflips of Wahl resolutions correspond to mutations of these H.e.c. While the braid group action on exceptional collections is  well-known, the special feature of our situation is that mutations of exceptional vector bundles are also exceptional vector bundles (up to a shift)
and not more complicated exceptional objects. This has strong consequences for their $\Hom$ spaces and clarifies the structure of the deformation of the Kalck--Karmazyn algebra that corresponds to the germ $P\in\oW$.

We recall that from an exceptional collection $\langle F,E \rangle \subset D^b(Y)$ we can obtain two other exceptional collections $\langle E,R_E(F) \rangle$ (right mutation of $F$ over $E$) and $\langle L_E(F),E \rangle$ (left mutation of $F$ over $E$), so that $\langle F,E \rangle=\langle E,R_E(F) \rangle=\langle L_E(F),E \rangle$. 
The objects are defined by distinguished evaluation triangles 
$$
E \otimes \RHom(E,F) \to F \to L_E(F)\to
    \quad\hbox{\rm and}\quad R_E(F) \to F \to E \otimes \RHom(F,E)^\vee\to.
$$
For a longer exceptional collection $\langle E_r,\ldots E_0 \rangle$, the action of left and right mutations induces an action of the braid group $B_{r+1}$ of $r+1$ strands on $\langle E_r,\ldots E_0 \rangle$. We~will also use mutations of more general s.o.d.'s. Matching braid group actions on Wahl resolutions and exceptional collections on their smootings relies on a geometric construction 
(Proposition~\ref{SRGsrgsRG}), which uses deformations from the universal family of antiflips that correspond to $1$-dimensional cones of the fan $\cF$ in Figure~\ref{samplefamily}.


\begin{proposition}\label{SRGsrgsRG}
Let $W$ be a Wahl resolution of a c.q.s.~ surface $\oW$ satisfying Assumption~\ref{assume}.
Let $W'=R_i(W)$ be the right antiflip of $W$ at the curve $\Gamma_i$ for some $i=1,\ldots,r$. 
The Wahl singularities of $W'$ are  $P'_j=P_j$ for $j \neq i-1,i$, and $P'_{i}=P_{i-1}$. 

We~consider two geometrically different  situations:
either (a) $K_W\cdot\Gamma_i\ge0$ or\break (b) $K_W\cdot\Gamma_i<0$.
Suppose also that  $K_{W'}\cdot\Gamma_i'<0$ in case (b)\footnote{If $K_{W'}\cdot\Gamma_i'>0$
then the left antiflip $W=L_i(W')$ will be in case (a) and the same results hold.}

Then there exist $\Q$-Gorenstein  families
$\cW,\cW'$
over a smooth curve $B$
with the following properties.
In case~(a) $\cW$ (resp.~$\cW'$) has special fiber $W$ (resp.~$W'$) over $0\in B$, the families are 
isomorphic over $B\setminus\{0\}$.
In case (b) we have  $B=\P^1$, $\cW=\cW'$, and the family has fiber $W$ (resp.~$W'$) over $0\in \P^1$ (resp.~$\infty\in\P^1$).
Let $Z\simeq Z'$ be  isomorphic (in both cases) general fibers of these families. These deformations have the following properties:
 
\begin{enumerate}
    \item $Z \rightsquigarrow W$ is equisingular at $P_j$ for $j\ne i$,
    $P_i$ is smoothened.
    \item $Z' \rightsquigarrow W'$
    is equisingular at $P_j'$ for $j\ne i-1$, 
      $P'_{i-1}$  is smoothened.
\item
Boundary divisors $\Gamma_0+\ldots+\Gamma_{r+1}\subset W$ and
$\Gamma'_0+\ldots+\Gamma'_{r+1}\subset W'$ lift to $Z\simeq Z'$.
\item The surface $Z\simeq Z'$ admits a $\Q$-Gorenstein smoothing to a surface $Y$ in the fixed irreducible component of the versal deformation space of $\oW$.
\end{enumerate}
\end{proposition}

\begin{remark} By analogy with case (b), 
in case (a) one can glue families $\cW$ and $\cW'$ to one family over a non-separated curve with a double origin $\{0,0'\}$.
\end{remark}

\begin{proof}
We start with case (a).
By Lemma~\ref{NO}, there exists a deformation $Z'\rightsquigarrow W'$
over a smooth curve
with properties (2) and (3) and total space $\cW'\to B$.
(This deformation corresponds to one of the two boundary rays of the first quadrant on the left side of Figure~\ref{samplefamily}.)
By \cite[Prop.~3.16]{HTU17}, there exists a contraction $\cW' \to \cW'_i$ of the curve $\Gamma'_i$ only.
This contraction is either $K$-negative or $K$-trivial, let 
$\cW^+ \to \cW'_i$ be its flip (or flop)
with special fiber $W^+$.
We claim that $W^+=W$ and the flip (or flop) gives a required deformation $Z \rightsquigarrow W$.
Indeed, $W^+$ has the same singularities as $W$ and the same flipped curve $\Gamma_i$ by \cite[Prop.~3.16]{HTU17}.
The chains up to $P'_{i-2}$ and from $P'_{i+1}$ are not affected by the flip. 
The curve from $P'_{i-2}$ to $P_i'$ in $Z'$ degenerates in $W^+$ to an irreducible curve $\Gamma_{i-1}$ from $P_{i-2}$ to $P_{i-1}$ while the curve from $P_i'$ to $P'_{i+1}$ breaks in $W^+$ into a union of a curve $\Gamma_{i+1}$ from $P_i$ to $P_{i+1}$ and a flipped curve $\Gamma_i$ from $P_{i-1}$ to ~$P_i$. 
Since $\cW'_i$ can be further contracted to a deformation $\ocW$ of $\oW$, $W^+$~is a Wahl resolution of $\oW$. 
By Prop.~\ref{computo}, the curves $\Gamma_{i-1}$
and $\Gamma_{i+1}$ have the same self-intersections in $W$ and $W^+$ and so the same $\delta$'s.

Geometry in case (b) is  different. Take a contraction  $(\Gamma_i\subset W)\to (Q_i\in W_i)$.
By~Lemma~\ref{NO}, there exists a deformation $Z\rightsquigarrow W$ with properties (1) and (3).
(This deformation corresponds to one of the non-boundary rays of the fan $\cF$ on the left side of Figure~\ref{samplefamily}.) By~\cite[Prop.~2.4]{HTU17}, $Z$ is  the special fiber of a k1A neighborhood. This means the following: there is a contraction $Z\to W_i$, which is a (non-c.q.s.!) resolution of $Q_i\in W_i$.
Its exceptional divisor $E\simeq\P^1$ passes through the Wahl singularity $P_{i-1}\in W$. The proper transform of $E$ intersects exactly one of the irreducible curves $D_s$ in the chain
$D_1,\ldots,D_p$ of the minimal resolution
$P_{i-1}\in W$, the intersection of $E$ and $D_s$ is 
transversal in one point.
Contracting $E$ gives a chain $D_1,\ldots,\tilde D_s,\ldots,D_p$ with 
$\tilde D_s^2=D_s^2+1$.
This chain is a (non-minimal) resolution of the singularity $Q_i\in W_i$. Blowing up a varying point of $\tilde D_s$ and contracting the chain 
$D_1,\ldots,\tilde D_s,\ldots,D_p$ back to 
$P_{i-1}\in W$ 
gives an obvious equisingular family with fibers $Z$ over ${\Bbb  G}_m$. Furthermore, by  \cite[Prop.~2.4]{HTU17}, the same family arises from deformations
$Z'\rightsquigarrow W'$ with properties (2) and (3). This gives the family $\cW$ over $\P^1$. By the proof of  Lemma~\ref{NO} and  semicontinuity, we have $H^2(Z,T_Z(-\log\Delta))=0$,
which again by Lemma~\ref{NO} shows that there are no local-to-global obstructions to its deformations. So we can obtain a $\Q$-Gorenstein smoothing $Y\rightsquigarrow Z$. Since $Z$ admits a contraction to $W_i$, it can be further contracted to $\oW$. By blowing down deformations, the smoothing $Y\rightsquigarrow Z$ blows down to a 
smoothing $Y\rightsquigarrow \oW$ proving (4).
\end{proof}


\begin{lemma}\label{hola}
In notation of Prop.~\ref{SRGsrgsRG},
consider  s.o.d.'s 
$D^b(\cW)=\langle \cA^{\cW}_r,\ldots,\cA^{\cW}_0, \cB^{\cW} \rangle$ and $D^b(\cW')=\langle \cA^{\cW'}_r,\ldots,\cA^{\cW'}_0, \cB^{\cW'}\rangle$ of Theorem~\ref{asfbsfbsf} and
their restrictions to $Z\simeq Z'$:
\begin{equation}\label{EGegEWGAR}
    D^b(Z)=\langle\cA_r,\ldots,\cA_0, \cB \rangle =
\langle \cA'_r,\ldots,\cA'_0,\cB' \rangle.
\end{equation}
Then $\cB=\cB'$,
$\cA_j=\{ F_j \}$ for $j\neq i$, and
$\cA'_j=\{ F_j \}$ for $j\neq i-1$.
Furthermore,
$$\langle\cA_{i},\cA_{i-1}\rangle = 
\langle E_{i}, \{ F_{i-1}\}\rangle = 
\langle \{ F_{i-1}\}, E'_{i-1} \rangle = 
\langle\cA_i',\cA'_{i-1}\rangle.$$
Here $E_i$ and $E_{i-1}'$ are Hacking vector   bundles associated with singularities $P_i\in W$ and $P_{i-1}'\in W'$ of central fibers of deformations as in Th.~\ref{sgsGsrgsRHsrh} and 
$F_j$ for $j\ne i$ are Kawamata sheaves  associated with 
singularities  $P_j$ on the general fiber $Z$ as in Prop.~\ref{kjshefjhksEFG}.
\end{lemma}

\begin{proof}
In case (a) of Prop.~\ref{SRGsrgsRG}, the deformations $\cW$ and $\cW'$ blow-down to the same deformation $\ocW$ of~$\oW$, so we have $\cB=\cB'$ by Corollary~\ref{EFegwEWE}.
In case (b) the deformation $\cW\to\P^1$ blows down to the constant
deformation $\oW\times\P^1$ of $\oW$. For fibers over $b\ne0,\infty$, the contraction 
$\pi:\,\cW_b\to\oW$ is not a c.q.s. resolution. However, 
$\cW_0=W\to\oW$ and $\cW_\infty=W'\to\oW$ are c.q.s. resolutions, so Corollary~\ref{EFegwEWE} still applies and gives $\cB=\cB'$ since it is a pullback of the  category $\cB^{\oW}$ via $L\pi^*$.

So the only thing to check is that subcategories $\cA_j$ and $\cA_j'$ that correspond to singular points of $W$ and $W'$ where the deformations are equisingular are subcategories of $D^b(Z)$  associated with Kawamata sheaves of these 
singularities. This follows from Lemma~\ref{argadrhadh}.
\end{proof}

\begin{theorem}\label{main3}
In assumptions of Prop.~\ref{SRGsrgsRG}, let $Y$
be a general smooth surface within an irreducible 
component of the versal deformation space of $\oY$
that contains $\Q$-Gorenstein smoothings of $W$ and $W'=R_i(W)$. Consider the corresponding H.e.c.
\begin{equation}\label{AEFefgesge}
D^b(Y)= \langle E_r,\ldots, E_0, \cB^Y \rangle = \langle E'_r,\ldots, E'_0, \cB'^Y \rangle.
\end{equation} 
Then $\cB'^Y=\cB^Y$, $E'_j=E_j$ for $j \neq i,i-1$, $E'_i=E_{i-1}$, and $E'_{i-1}=R_{E_{i-1}}(E_{i})[k]$,
where $k=0$ in case (a) of Prop.~\ref{SRGsrgsRG} and $k=1$ in case (b).
Moreover, in case (a)
$$\Hom(E_{i}, E_{i-1})=\Ext^2(E_i, E_{i-1}) 
=\Ext^1(E'_{i}, E'_{i-1})=\Ext^2(E'_{i}, E'_{i-1})=
0,$$ 
$$\Ext^1(E_{i}, E_{i-1})\simeq \Hom(E'_{i}, E'_{i-1})^\vee\simeq\C^{\delta_i}.$$
In case (b),
$\Ext^k(E_i, E_{i-1})=\Ext^k(E'_{i}, E'_{i-1})=
0$ for $k>0$, and  $$\Hom(E_{i}, E_{i-1})\simeq \Hom(E'_{i}, E'_{i-1})^\vee\simeq\C^{\delta_i}.$$
\end{theorem}

\begin{proof}
By Theorem~\ref{sgsGsrgsRHsrh}, the subcategories $\langle E_i\rangle$
and $\langle E'_i\rangle$ are also generated by the  Kawamata bundles $F_i$ and $F_i'$.
By Lemma~\ref{hola}, we can consider a $\Q$-Gorenstein smoothing $Y \rightsquigarrow Z$ with the total space $\cY$ inside the $\Q$-Gorenstein deformation spaces of $W$ and $W'$ and Kawamata sheaves on $Z$ are deformations of Kawamata sheaves on $W$ and $W'$. Therefore, s.o.d.'s \eqref{AEFefgesge}
are specializations of the s.o.d.'s 
$$
    D^b(\cY)=\langle\cA^{\cY}_r,\ldots,\cA^{\cY}_0, \cB^\cY \rangle =
\langle \cA'^{\cY}_r,\ldots,\cA'^{\cY}_0,\cB'^{\cY} \rangle
$$
obtained by Theorem~\ref{asfbsfbsf}
as deformation of the s.o.d.'s \eqref{EGegEWGAR}.
Thus $\cB'^Y=\cB^Y$ by Corollary~\ref{EFegwEWE}, $E'_j=E_j$ for $j \neq i,i-1$, $E'_i=E_{i-1}$, and 
$R_{E_{i-1}}(E_{i}) \in \langle E'_{i-1} \rangle$, and so $E'_{i-1}=R_{E_{i-1}}(E_{i})[k]$ for some $k$. 

Next we consider case (a), where we will show now that $k=0$. 
Let $C^{\bullet}=R_{E_{i-1}}(E_{i})$. By definition of mutation we have 
an exact sequence
$$0 \to E_{i-1} \otimes \Ext^2(E_i,E_{i-1})^{\vee} \to C^{-1} \to 0 \to E_{i-1} \otimes \Ext^1(E_i,E_{i-1})^{\vee} \to $$ $$ C^0 \to  E_i \to E_{i-1} \otimes \Hom(E_i,E_{i-1})^{\vee} \to C^1 \to 0.$$

So either $k=0$ or $k=\pm1$. 
Since $\rk(E'_{i-1})= n_{i-1}'=\delta_i n_{i-1} + n_i=\delta_i \rk(E_{i-1}) + \rk(E_i)$ by Prop.~\ref{computo} and $\rk(R_{E_{i-1}}(E_i))= \delta_i \rk(E_{i-1}) + \rk(E_i)$ by definition of mutation, we have $k=0$, and so $E'_{i-1}=R_{E_{i-1}}(E_{i})$. 

Since $C^{-1}=C^1=0$, we have $\Ext^2(E_i,E_{i-1})=0$.
Recall that $E_{i-1}\simeq E_i'$. 
Applying 
$\RHom(E_{i-1},\bullet)$ to
the distinguished triangle $E'_{i-1} \to E_i \to E_{i-1} \otimes \RHom(E_i,E_{i-1})^{\vee}$ gives  $ \RHom(E_i,E_{i-1})^{\vee}\simeq\RHom(E_{i}',E'_{i-1})[1]$. Therefore 
$$\Hom(E'_{i},E'_{i-1})=\Ext^1(E_i,E_{i-1})^{\vee} \ \ \ \text{and} \ \ \ \Ext^1(E'_{i},E'_{i-1})=\Hom(E_i,E_{i-1})^{\vee}.$$
This also shows that $\Ext^2(E'_{i},E'_{i-1})=0$.
We claim that $\Ext^1(E'_{i},E'_{i-1})=0$ as well, which will show that $\Hom(E_i,E_{i-1})=0$ and therefore that $\Ext^1(E_{i}, E_{i-1})\simeq\C^{\delta_i}$ by Lemma~\ref{compu}.
In order to do this we will analyze the next
antiflip $W''=R_i(W')$, which is an instance of case (b). This calculation will also settle the case~(b).

As before, we have $E_i''=E_{i-1}'$ and $E_{i-1}''=R_{E_{i-1}'}(E_i)[k]$ for some $k$.
Let $C^\bullet=R_{E_{i-1}'}(E_i')$. By definition of 
mutation we have an exact sequence
$$0 \to   C^{-1} \to 0 \to E'_{i-1} \otimes \Ext^1(E'_i,E'_{i-1})^{\vee} \to $$ $$ C^0 \to  E'_i \to E'_{i-1} \otimes \Hom(E'_i,E'_{i-1})^{\vee} \to C^1 \to 0.$$
So $C^{-1}=0$ and therefore $k=0$ or $k=1$. We claim that $k=1$.
Indeed, $\rk(E_{i-1}'')=n_{i-1}''=\delta n'_{i-1}-n_i'=\delta\rk(E'_{i-1})-\rk(E_i')$. On the other hand,
$\rk(R_{E_{i-1}'}(E_i'))=-\chi(E_i',E_{i-1}')\rk(E'_{i-1})+\rk(E_i')=\delta\rk(E'_{i-1})+\rk(E_i')$.
So $k=1$. This shows that $\Ext^1(E_i',E_{i-1}')=0$,
which proves all claims in case (a).

Finally, to finish case (b), we apply 
$\RHom(E'_{i-1},\bullet)$ to
the distinguished triangle $E''_{i-1}[-1] \to E'_i \to E'_{i-1} \otimes \RHom(E'_i,E'_{i-1})^{\vee}$, which  gives  $ \RHom(E'_i,E'_{i-1})^{\vee}\simeq\RHom(E_{i}'',E''_{i-1})$. So $\Ext^k(E''_{i}, E''_{i-1})=
0$ for $k>0$, and we also have an isomorphism $\Hom(E''_{i}, E''_{i-1})\simeq \Hom(E'_{i}, E'_{i-1})^\vee\simeq\C^{\delta_i}$.
\end{proof}

\begin{lemma}\label{compu}
Let $Y$ be a $\Q$-Gorenstein smoothing of a Wahl resolution $W$. Let $E_r,\ldots,E_0$ be the corresponding H.e.c.~on $Y$.
Then $\chi(E_{i},E_{i-1})=-n_{i} n_{i-1} \, \Gamma_{i} \cdot K_W$\footnote{The Euler pairing for $\alpha, \beta \in D^b(Z)$ is the integer
$\chi(\alpha,\beta)= \sum_j (-1)^j \dim_{\C} \Hom(\alpha,\beta[j])$.}.
In particular, if $K_W\cdot\Gamma_i>0$ then $\chi(E_{i},E_{i-1})=-\delta_i$ and 
if $K_W\cdot\Gamma_i<0$ then $\chi(E_{i},E_{i-1})=\delta_i$.
\end{lemma}

\begin{proof}
$\chi(E_i,E_{i-1})=\sum (-1)^j\ext^j(E_i,E_{i-1})$, and so by Riemann-Roch
$$\chi(E_i,E_{i-1})=n_i \ch_2(E_{i-1}) - c_1(E_i)\cdot c_1(E_{i-1}) + n_{i-1} \ch_2(E_i)\qquad{}$$
$${}\qquad+{1\over 2}\left(n_{i-1} c_1(E_i) - n_i c_1(E_{i-1})\right)\cdot K_{Y} + n_i n_{i-1},$$ where $\ch_2=\frac{1}{2}(c_1^2-2c_2)$. By Theorem \ref{sgsGsrgsRHsrh}, we know that $$ c_1(E_i)=-n_i(A+\Gamma_1+\ldots+\Gamma_i) \in H_2(Y),$$ and $c_2(E_i)=\frac{n_i-1}{2n_i}(c_1(E_i)^2 + n_i +1)$. Hence we do the computation in $W$. We have $\ch_2(E_i)=\frac{1}{2}\Big(n_i(A+\Gamma_1+\ldots+\Gamma_i)^2-\frac{n_i^2-1}{n_i}\Big)$, and so $$n_i \ch_2(E_{i-1}) - c_1(E_i)\cdot c_1(E_{i-1}) + n_{i-1} \ch_2(E_i)=\frac{1}{2} n_i n_{i-1} \Gamma_i^2 -n_i n_{i-1} + \frac{n_{i-1}}{2n_i} + \frac{n_i}{2n_{i-1}}.$$ But on $W$ we have $\Gamma_i^2=-\Gamma_i \cdot K_W - \frac{1}{n_i} - \frac{1}{n_{i-1}}$. On the other hand $n_{i-1} c_1(E_i) - n_i c_1(E_{i-1})=-n_{i-1} n_i \Gamma_i$, and so we plug in the formula for $\chi(E_i,E_{i-1})$ to obtain the formula.
\end{proof}

Finally,  we can finish the proof of Theorem~\ref{weFwetwET}
with Lemmas~\ref{ssgsGsrhSR}-\ref{qergawgwrg}.

\begin{lemma}\label{ssgsGsrhSR}
Let $W$ be a Wahl resolution of $\overline W$
satisfying Assumption~\ref{assume} and let $Y$
be a sufficiently general surface from the corresponding versal deformation space of $\oW$.
Let~$E_r,\ldots, E_0$ be the corresponding Hacking exceptional collection on $Y$. Then, for every $i>j$, either $\Ext^k(E_i, E_j)=0$ for $k\ne1$ or $\Ext^k(E_i, E_j)=0$ for $k\ne0$.

In particular, suppose $E_r,\ldots, E_0$ 
(resp.~$\bar E_r,\ldots, \bar E_0$)
is the H.e.c.~ 
that corresponds to the M-resolution $W^+$
(resp.~the N-resolution $W^-$). Then
\begin{enumerate}
\item 
$\bar E_r,\ldots, \bar E_0$ is a strong exceptional collection, i.e.~$\Ext^k(\bar E_i, \bar E_j)=0$ for $k>0$, $i>j$.
\item
$\Ext^k(E_i, E_j)=0$ for $k\ne1$, $i>j$. 
\item
For $i=1,\ldots, r-1$, we have
$\Hom(\bar E_{r-i+1}, \bar E_{r-i})\simeq
\Ext^1(E_i, E_{i-1})^\vee\simeq
\C^{\delta_i}$.
\item $\bar E_r,\ldots, \bar E_0$ is a mutation of $E_r,\ldots, E_0$ 
(no homological shifts).
\end{enumerate}
\end{lemma}

\begin{proof}
Let $W$ be an arbitrary Wahl resolution. 
We claim that we can bring any two singularities $P_i$, $P_j$, $i>j$, in $W$ together (without changing them) via a sequence of right antiflips. If $i=j+1$, then $P_i$, $P_j$ are already together via $\Gamma_{i}$. Otherwise we have a chain $\Gamma_{j+1}, \ldots, \Gamma_{i}$ connecting them, which we can antiflip from $k=j+1$ to $k=i-1$, after that the new singularities in positions $i-1$ and $i$ are $P_j$ and $P_i$.
By Th.~\ref{main3}, the corresponding mutations do not change Hacking bundles $E_i$ and $E_j$
and we have $\Ext^k(E_i, E_j)=0$ for $k\ne1$ 
if $K\cdot\Gamma_i\ge0$ and 
or $\Ext^k(E_i, E_j)=0$ for $k\ne0$
if $K\cdot\Gamma_i\le0$, where $\Gamma_i$ is the curve connecting the points after the antiflips.

Now suppose $E_r,\ldots, E_0$ 
(resp.~$\bar E_r,\ldots, \bar E_0$)
is the H.e.c.~ 
that corresponds to the M-resolution $W^+$
(resp.~the N-resolution $W^-$)
Part (3) follows from Th.~\ref{main3} and definition of the N-resolution. For part (2), we argue as follows.
By Prop.~\ref{computo}, for every  antiflip of $\Gamma_k$, $k=j+1,\ldots, k=i-1$ described above, 
the curves in positions $>k$ retain 
non-negative intersection with the canonical class,
including the last curve $\Gamma_i$ that will connect $P_i$ to~ $P_j$.
Thus (2) follows from Th.~\ref{main3}.

By Th.~\ref{main3}, the sequence of antiflips that connects $W^+$ to $W^-$
corresponds to the sequence of mutations that
takes an exceptional collection $E_r,\ldots, E_0$ literally 
to the exceptional collection~$\bar E_r,\ldots, \bar E_0$, i.e.~without homological shifts. This shows (4). 
Finally, we prove (1). 
We follow $E_i$ and $E_j$ through a sequence of mutations. To simplify notation, we will denote their images after mutations by the same letters. We can decompose this element of the braid group as follows: 
arrange bundles into three blocks: $A=\langle E_r,\ldots,E_{i+1}\rangle$, $B=\langle E_i,\ldots,E_{j}\rangle$, $C=\langle E_{j-1},\ldots,E_{0}\rangle$.
Rearrange the bundles in $A$ and $C$ in the opposite order (moving $E_{i+1}$, resp.~$E_0$ all the way to the left without changing it as in the definition of the N-resolution.)
This does not change $E_i$ or $E_j$, so the $\RHom$ between them stays the same. Next, rearrange bundles in  $B$ in the opposite order by moving $E_j$ all the way to the left.
By the analysis in the beginning of the proof of the lemma, this will change non-zero components of $\RHom$ from $\Ext^1(E_i,E_j)$ to $\Hom(E_j,E_i)$.
Finally, we  mutate s.o.d.'s:
$\langle A,B,C\rangle\to
\langle A,C,B'\rangle\to
\langle C,A',B'\rangle\to
\langle C,B',A''\rangle$.
This gives an equivalence $B\to B'$ which does not change $\RHom$'s between its objects,
for example between $E_i$ and $E_j$.
\end{proof}

\begin{lemma}\label{askbfhsbKFHjsghj}
Let $\pi:\,W\to\oW$ be a Wahl resolution 
satisfying Assumption~\ref{assume} and let $Y$
be a sufficiently general surface from the corresponding versal deformation space of $\oW$.
Let~$E_r,\ldots, E_0$ be the corresponding Hacking exceptional collection on $Y$. 
\begin{enumerate}
    \item 
The Kawamata vector bundle $\bar F$ on $\overline W$ deforms uniquely to a vector bundle $F$ on~$Y$.
\item For every $i=0,\ldots,r$,  $\Hom(F,E_i)=\C^{\rk E_i}$ and 
$\Ext^k(F,E_i)=0$ for $k>0$.
\item 
$\{ F \} = \langle F \rangle = \langle E_r,\ldots, E_0 \rangle\subset D^b(Y)$.
\end{enumerate}
\end{lemma}

\begin{proof}
Let $F=\bar\cF|_Y$ be the unique deformation of $\bar F$ given by Lemma~\ref{sFsefwetwet}, which is locally free by Lemma~\ref{qrgsrhr}, which proves (1).
By Corollary~\ref{EFegwEWE}, Remark~\ref{wEFegEWG}
the admissible subcategory $\{F\}\subset D^b(Y)$ is equal to the subcategory $\langle E_r,\ldots, E_0 \rangle$. By Theorem~\ref{sgsGsrgsRHsrh}, the Kawamata bundle $\cF_i|_Y$ on $Y$ associated with the singularity $P_i\in W$ is isomorphic to $E_i^{\rk E_i}$. So (2) will follow if we can show that $\Hom(F,\cF_i|_Y)=\C^{\rk \cF_i|_Y}$ and 
$\Ext^k(F,\cF_i|_Y)=0$ for $k>0$. Since $\cF_i|_Y$ is a deformation of the Kawamata sheaf $F_i$ on $W$, by semi-continuity it suffices to show that
$\RHom_W(\pi^*\bar F, F_i)=\C^{\rk F_i}$.
Since $F_i$ is the maximal iterated extension of $\cO_W(-D_i)$, it suffices to show that 
$\RHom_W(\pi^*\bar F,\cO_W(-D_i))=\C$.
This will follow at once from Lemma~\ref{asfhsrHs} and adjunction if we can show that
$R\pi_*\cO_W(-D_i)=\cO_\oW(-\bar A)$. In the exact sequence $0\to \cO_W(-D_i)\to\cO_W\to\cO_{D_i}\to 0$, the derived pushforward by $\pi$ of 
$\cO_{D_i}=\cO_{\Gamma_0\cup\ldots\cup\Gamma_i}$ is equal to $\cO_{\pi(\Gamma_0)}=\cO_{\bar A}$ and the derived pushforward of $\cO_W$ is $\cO_\oW$. So~ $R\pi_*\cO_W(-D_i)=\cO_\oW(-\bar A)$. In (3), we just need to prove that $\{ F \} = \langle F \rangle$, which will follow once we show that $F \simeq \bigoplus\limits_{i=0}^r \bar E_i^{n_{r-i}}$ in the next Lemma~\ref{qergawgwrg} because then
$\langle F \rangle$ contains $\langle \bar E_r,\ldots, \bar E_0 \rangle$, which is equal to $\{ F \}$.
\end{proof}

\begin{lemma}\label{qergawgwrg}
Let $\oW$ be a c.q.s. surface 
satisfying Assumption~\ref{assume} and let $Y$
be its sufficiently general smoothing from a fixed component 
of the versal deformation space.
Let $E_r,\ldots, E_0$ 
(resp.~$\bar E_r,\ldots, \bar E_0$)
be the H.e.c.~on $Y$ 
that corresponds to the M-resolution $W^+$ (resp.~the N-resolution $W^-$).
Then
\begin{enumerate}
\item These collections are dual: $\RHom(\bar E_j,E_{r-i})=\C$ if $i=j$ and $0$ otherwise.
\item 
The Kawamata bundle $F$ is isomorphic to $\bigoplus\limits_{i=0}^r \bar E_i^{n_{r-i}}$, 
where $n_{j}=\rk E_{j}$. The Kalck--Karmazyn algebra $R=\End(\bar F)$ deforms to the algebra $\End(F)$ Morita-equivalent to the endomorphism algebra
$\hat R=\End(\bar E_r\oplus\ldots\oplus\bar E_0)$ of a strong exceptional collection.
\item $\hat R$ is a path algebra of a quiver with $(r+1)$ vertices labeled by $\bar E_r,\ldots,\bar E_0$ and with arrows connecting $\bar E_i$ to $\bar E_j$ for $i>j$ so that the total number of paths connecting $\bar E_i$ to $\bar E_j$ is equal to $\hom(\bar E_i,\bar E_j)$.
In particular, the category $\langle F\rangle$ does not depend on~ $Y$.
\end{enumerate}
\end{lemma}

\begin{proof}
We will prove by a down-ward induction on $j$ that if $X\in\langle\bar E_r,\ldots,\bar E_j\rangle$ and $\RHom(X,\bar E_i)=\C^{\bar n_i}$,
$\RHom(X,E_{r-i})=\C^{n_{r-i}}$
for $i\ge j$ then 
$X\simeq \bar E_j^{n_{r-j}}\oplus\ldots\oplus \bar E_r^{n_0}$.
If $j=r$ then this is clear because $\bar E_r\simeq E_0$ and  $\bar n_r=n_0$.
When $j=0$, this will give part (2) by Lemma~\ref{askbfhsbKFHjsghj}  (2). To prove the step of induction, let the claim be true for $j+1$. Take $X\in\langle\bar E_r,\ldots,\bar E_j\rangle$. By definition of s.o.d., we have a distinguished triangle \begin{equation}\label{asrgasrgarh}
    \bar E_j\otimes C^\bullet\to X\to X'\to,
\end{equation}
where $C^\bullet$ is a complex of vector spaces and $X'\in\cA$ , where $\cA$ is an admissible subcategory $\langle\bar E_r,\ldots,\bar E_{j+1}\rangle$, which is equal to $\langle E_{r-j-1},\ldots,E_0\rangle$ by Th.~\ref{main3}. 
Also, $\bar E_j$ is the right mutation of $E_{r-j}$ over $\cA$.
So we have a mutation triangle
\begin{equation}\label{SGsrghsrharh}
\bar E_j\to E_{r-j}\to T\to 0
\end{equation}
with $T\in\cA$. This triangle proves part (1).
Since $\RHom(\bar E_j,\bar E_i)=0$ for $i>j$,
we have $\RHom(X',\bar E_i)=\RHom(X,\bar E_i)=\C^{\bar n_i}$ by \eqref{asrgasrgarh}.
Since $\RHom(\bar E_j,E_{r-i})=0$ for $i>j$
by \eqref{SGsrghsrharh},
we have $\RHom(X',E_{r-i})=\RHom(X, E_{r-i})=\C^{n_{r-i}}$ by \eqref{asrgasrgarh}.
This implies that $X'\simeq  \bar E_{j+1}^{n_{r-j-1}}\oplus\ldots\oplus \bar E_r^{n_0}$ by inductive assumption. 
The triangle \eqref{asrgasrgarh} and part (1) imply
$C^\bullet=\RHom(X,E_{r-j})=\C^{n_{r-j}}$. Since $\Ext^1(\bar E_i,\bar E_j)=0$ for $i>j$, this implies that $X\simeq \bar E_j^{n_{r-j}}\oplus\ldots\oplus \bar E_r^{n_0}$. This proves (2).

Under the equivalence $\langle F\rangle \to D^b(\hat R-\mmod)$, $X\mapsto\RHom(\hat R, X)$, the vector bundles $E_r,\ldots,E_0$ go to simple modules of $\hat R$ by part (1). Since $\Ext^2(E_i,E_j)=0$ for all $i,j$ by Lemma~\ref{ssgsGsrhSR} (2), this implies that $\hat R$ is hereditary, and therefore $\hat R$ is isomorphic to a path algebra of a quiver by a well-known theorem of Gabriel.
\end{proof}

\begin{lemma}\label{sdtheth}
The formula \eqref{afbzdfbdfna} holds.
\end{lemma}

\begin{proof}
One can compute this dimension by $\Q$-Gorenstein smoothing all singularities between $\bar P_i$ and $\bar P_j$ to obtain a curve $\Gamma$ through singularities $\bar P_i$ and $\bar P_j$ which are now consecutive. Now use Theorem~\ref{main3}.
\end{proof}

\begin{remark}
In Definition \ref{sRGsrgrH} we required Wahl resolutions $W\to\oW$ to have a $\Q$-Gorenstein smoothing $Y\rightsquigarrow W$ that blows down to a smoothing $Y\rightsquigarrow \oW$, because otherwise there seems to be no natural way to define the braid group action.
\begin{figure}[htbp]
\centering
\includegraphics[width=2.5cm]{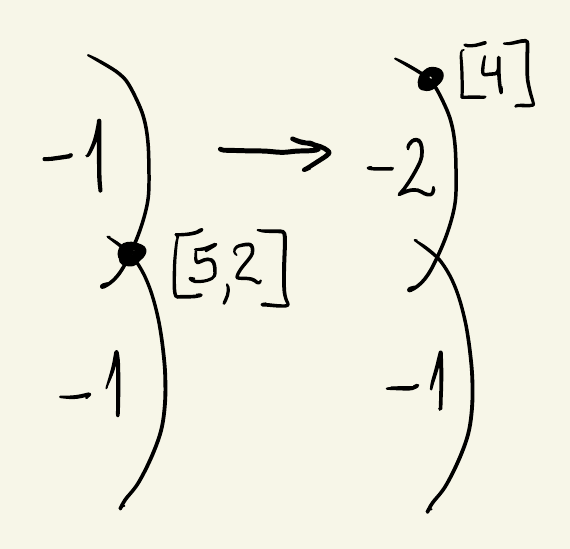}
\centering
\caption{Divisorial contraction}
\label{sfvasvasb}
\end{figure}
For example, consider the Wahl~resolution $W$ of $\oW={1\over 3}(1,1)$ singularity on the left side of Figure~\ref{sfvasvasb}. Its $\Q$-Gorenstein smoothing $Y\rightsquigarrow W$ blows down 
to a smoothing $\oY\rightsquigarrow \oW$ but $\oY\ne Y$, in fact $Y$ is a blow-up of $\oY$ at a smooth point. To see this, we can flip $W$ at the top curve producing a resolution $W'$ on the right side of Figure~\ref{sfvasvasb}, which 
contains a $(-1)$-curve $E$ which deforms to a nearby fiber producing a divisorial contraction of the threefold. Accordingly, an antiflip of $W'$ at $E$ is not defined. On the level of derived categories, $W$ gives a H.e.c.~$E_2,E_1,E_0$ on $Y$ where $E_0$ and $E_2$
are line bundles and $E_1$ has rank $3$. This is mutated to a H.e.c.~$E_2',E_1',E_0'$ on $W'$,
where $E'_0$ has rank $2$ and $E'_2=E'_1(-E)$ are line bundles. One can of course mutate this exceptional collection further but
this gives $E_1',E'_1|_E,E_0'$, where the sheaf in the middle is a torsion sheaf supported on $E$.
By contrast, mutating collections of vector bundles associated with Wahl resolutions always gives collections of vector bundles associated with other Wahl resolutions (up to a homological shift which is easy to compute, see Theorem~\ref{main3}). 
\end{remark}

We end this section by analyzing which quivers show up in Theorem \ref{main} when the M-resolution $W^+ \to \oW$ has two exceptional curves $\Gamma_1, \Gamma_2$ with associated $\delta_1=:a$ and $\delta_2=:b$. We have Wahl singularities $P_i$ of type $\frac{1}{n_i^2}(1,n_i a_i-1)$ for $i=0,1,2$. For the N-resolution we use the bar notation: Wahl singularities $\bar P_i$ of type $\frac{1}{\bar n_i^2}(1,\bar n_i \bar a_i-1)$ for $i=0,1,2$, etc. The associated H.e.c. is $\langle \bar E_2, \bar E_1, \bar E_0 \rangle$, and we have $\hom(\bar E_2,\bar E_0)= ab + c$, where $c:= \frac{a n_2+ b n_0}{n_1}$.
The algebra $\hat R$ is the path algebra of the quiver $Q_{a,b,c}$ with vertices $\bar E_2, \bar E_1, \bar E_0$ and $a$ arrows between $\bar E_2, \bar E_1$, $b$ arrows between $\bar E_1, \bar E_0$, and $c$ arrows between $\bar E_2, \bar E_0$, as in Figure \ref{triangle}.  

\begin{figure}[htbp]
\includegraphics[width=2.5in]{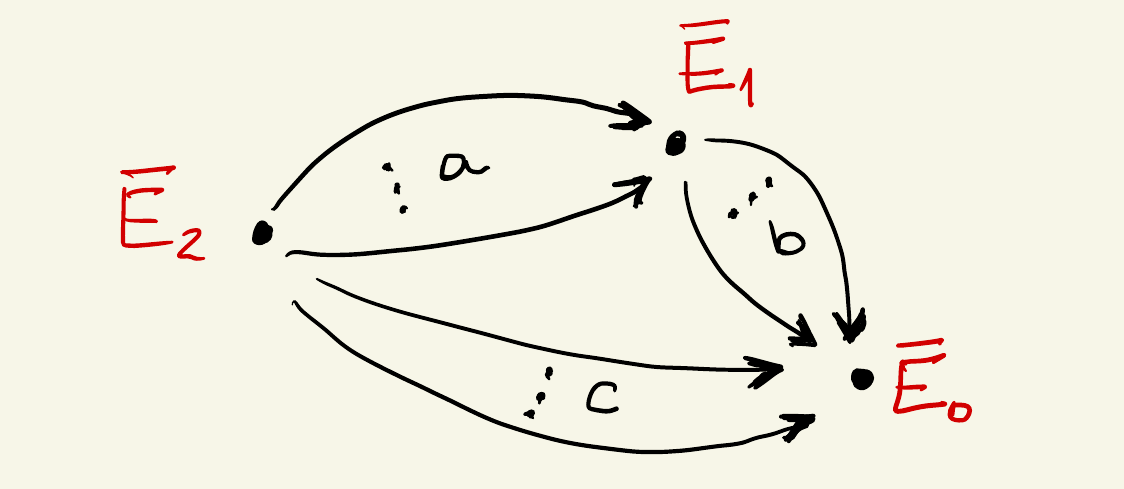}
\caption{The quiver $Q_{a,b,c}$ for an N-resolution with two curves}\label{triangle}
\end{figure}

\begin{proposition}\label{Q4answered}
The representation algebra of the quiver $Q_{a,b,c}$ can be realized by the categorified Milnor fiber, i.e.~ the algebra $\hat R$
of Theorem \ref{main}, if and only if there exists an extremal P-resolution with Wahl singularities of indices $a$ and $b$ and $\delta=c$ (see Definition~\ref{srbsbsr}).
See~Lemma~\ref{pun} for the list of possible $c$ for fixed values of $a,b$. 
\end{proposition}

A direct corollary is that if we realize $Q_{a,b,c}$ by our construction, then all permutations of $a,b,c$ are realizable by our construction. This is because we can reverse the orientation of the c.q.s. to obtain $Q_{b,a,c}$, and we can use the ``circular" zero continued fraction of the extremal P-resolution of $Q_{a,b,c}$ to realize zero continued fractions for $Q_{c,a,b}$ and $Q_{b,c,a}$ (see \cite[Section 4]{HTU17}). 

By Example \ref{rationalsmoothing}, 
$D^b(\hat R)$ is an admissible subcategory of the derived category of a smooth projective rational surface.
This partially answers question (Q5) in \cite{BR}.

\begin{proof}
If $a=0$ and $b>0$, then $b=c$; similarly for $b=0$ and $a>0$ we have $c=a$. 
If $a=b=0$, then there are no arrows.

Let us assume that $a,b>0$. Let $\frac{\Delta}{\Delta-\Omega}=[b_1,\ldots,b_s]$. As in Section \ref{s1}, we have that $d_{i_1}=d_{i_2}=d_{i_3}=1$, and the zero continued fraction corresponding to the N-resolution is $[b_s,\ldots,b_{i_3}-1,\ldots,b_{i_2}-1,\ldots,b_{i_1}-1, \ldots,b_1]=0$. We have that $$\frac{\bar n_{2-k}}{\bar n_{2-k} -\bar a_{2-k}}=[b_1\ldots,b_{i_{k+1}-1}],$$ for $k=0,1,2$. (If $i_1=1$, then we set $\bar n_2=\bar a_2=1$, and so $\bar P_2$ is a smooth point in that case.) We also have $$\frac{a}{\epsilon_a}=[b_{i_2+1},\ldots,b_{i_3-1}] \ \ \ \text{and} \ \ \ \ \frac{b}{\epsilon_b}=[b_{i_2-1},\ldots,b_{i_1+1}]$$ for some $\epsilon_a,\epsilon_b$ coprime to $a,b$ respectively. (If $i_3=i_2+1$ ($i_2=i_1+1$), then $a=1$ ($b=1$).) 

Let us consider the Hirzebruch-Jung continued fraction that results from contracting (possibly) $1$ and all new $1$'s from $[b_{i_3-1},\ldots,b_{i_1}-1, \ldots,b_1]$. This corresponds to a rational number $\frac{\alpha}{\beta}>1$. Then, if $\frac{\alpha}{\alpha-\beta}=[x_1,\ldots,x_t]$, we have $$[x_t,\ldots,x_1,1,b_{i_3-1},\ldots,b_{i_1}-1, \ldots,b_1]=0,$$ and so the $\Q$-Gorenstein smoothing of $\bar P_1$ which keeps the singularities at $\bar P_0$ and $\bar P_2$ and a curve between them corresponds to that extremal N-resolution. In particular, we have $$\frac{ab+c}{\epsilon_{ab+c}}=[b_{i_3-1},\ldots,b_{i_1+1}] $$ for some $\epsilon_{ab+c}$.

\begin{lemma} \label{pun}
Let $\lambda:=b_{i_2} \geq 2$.
\begin{itemize}
    \item[(0)] If $a=b=1$, then $c=\lambda-1$,
   
    \item[(1)] If $a=1$ and $b>1$, then $c=\lambda b-b-\epsilon_b$; If $b=1$ and $a>1$, then $c=\lambda a-a-\epsilon_a$,
   
    \item[(2)] If $a,b>1$, then $c=(\lambda-1)ab-\epsilon_a b - \epsilon_b a$. In particular gcd$(a,b)$ always divides $c$.
\end{itemize}
\end{lemma}

\begin{proof} Part (0) is trivial because $a=b=1$ implies $i_3=i_2+1=i_1+2$, and so $\frac{ab+c}{\epsilon_c}=[b_{i_2}]$. For (1) and (2) we use the identity (for (1) we eliminate the ``$a$ or $b$ matrix") $$\begin{pmatrix} a & -\epsilon_a \\ \epsilon'_a & \frac{1-\epsilon_a \epsilon'_a}{a} \end{pmatrix} \cdot \begin{pmatrix} \lambda & -1 \\ 1 & 0 \end{pmatrix} \cdot \begin{pmatrix} b & -\epsilon'_b \\ \epsilon_b & \frac{1-\epsilon_b \epsilon'_b}{b} \end{pmatrix} = \begin{pmatrix} ab+c & -\epsilon'_{ab+c} \\ \epsilon_{ab+c} & \frac{1-\epsilon_{ab+c} \epsilon'_{ab+c}}{ab+c} \end{pmatrix},$$ where $0<\epsilon'_x<x$ is the inverse modulo $x$ of $\epsilon_x$. 
\end{proof}

\begin{proposition}
Each of the cases $(t)$ in Lemma \ref{pun} corresponds to an extremal P-resolution with $t$ singularities and $\delta=c$, represented by the continued fraction (here $*$ is a smooth point): 
\begin{itemize}
    \item[(0)] $*-(\lambda+1)-*$,
    
    \item[(1)] $[{a \choose  a-\epsilon_a}]-(\lambda)-*$ or $*-(\lambda)-[{b \choose  \epsilon_b}] $,
    
    \item[(2)] $[{a \choose  a-\epsilon_a}]-(\lambda-1)-[{b \choose  \epsilon_b}]$.
    
\end{itemize}
On the other hand, each extremal P-resolution in those three cases (with $\delta=c$) produces infinitely many M-resolutions with two exceptional curves, so that $\delta_1=a$, $\delta_2=b$ and $n_1c=an_2+bn_0$. 
\end{proposition}

\begin{proof} 
First, it is easy to check that each (t) case in  Lemma \ref{pun} corresponds to the described extremal P-resolution, since $\frac{1}{a}(1,\epsilon'_a)-(\lambda-1)-\frac{1}{b} (1,\epsilon_b)$ is contractible. The opposite is trickier for the case $\lambda=2$, so we first show it for $\lambda \geq 3$. In this case we have that  $\frac{1}{a}(1,\epsilon'_a)-(\lambda-1)-\frac{1}{b} (1,\epsilon_b)$ contracts to a c.q.s. Let us choose arbitrarily a Hirzebruch-Jung continued fractions $[x_1,\ldots,x_u]$ where $x_i \geq 2$ for all $i$. Then $$\frac{\alpha}{\beta}= \frac{1}{a}(1,\epsilon'_a)-(\lambda-1)-\frac{1}{b} (1,\epsilon_b)-[x_1,\ldots,x_u]$$ is the continued fraction of some c.q.s., and so if $\frac{\alpha}{\alpha-\beta}=[y_1,\ldots,y_v]$ we have $$[y_v,\ldots,y_1]-(1)-\frac{1}{a}(1,\epsilon'_a)-(\lambda-1)-\frac{1}{b} (1,\epsilon_b)-[x_1,\ldots,x_u]=0.$$ Therefore we can consider $$[y_v,\ldots,y_1]-(b_{i_3}=2)-\frac{1}{a}(1,\epsilon'_a)-(b_{i_2}=\lambda)-\frac{1}{b} (1,\epsilon_b)-[b_{i_1}=x_1+1,\ldots,x_u]$$ as the continued fraction that defines an M-resolution with $d_{i_3}=d_{i_2}=d_{i_1}=1$.

The key for a similar construction in the case $\lambda=2$ is to prove that $\frac{1}{a}(1,\epsilon'_a)-(1)-\frac{1}{b} (1,\epsilon_b)$ contracts (to a c.q.s.~ or  a smooth point), which is proved in the next lemma.
\end{proof}

\begin{lemma}
Given an extremal P-resolution $[{a \choose  a-\epsilon_a}]-(1)-[{b \choose  \epsilon_b}]$, then $\frac{1}{a}(1,\epsilon'_a)-(1)-\frac{1}{b} (1,\epsilon_b)$ contracts.
\end{lemma}

\begin{proof}
Let us write $\frac{a}{\epsilon_a}=[x_1,\ldots,x_p]$, $\frac{a}{a-\epsilon_a}=[y_1,\ldots,y_q]$, $\frac{b}{\epsilon_b}=[z_1,\ldots,z_u]$, and $\frac{b}{b-\epsilon_b}=[w_1,\ldots,w_v]$. Then by Lemma \ref{merken} we have
$$[{a \choose  a-\epsilon_a}]-(1)-[{b \choose  \epsilon_b}] = [y_1,\ldots,y_q+x_p, \ldots,x_1]-(1)-[z_1,\ldots,z_u+w_v, \ldots,w_1].$$ Hence, if $[x_p, \ldots,x_1]-(1)-[z_1,\ldots,z_u]$ does not contract, then $[z_1,\ldots,z_u]=[y_1,\ldots,y_q,t_1,\ldots,t_l]$, where $[t_1,\ldots,t_l]$ is the Hirzebruch-Jung continued fraction of some c.q.s. This is because of the algorithm that constructs Wahl chains, and we are assuming that $x_p$ is the first curve that becomes $0$ (so that it does not contract). But now we can compute that the intersection of the canonical class with the image of the $(-1)$-curve in the original extremal P-resolution is $-1+ \frac{a-\epsilon_a}{a} + \frac{b-\epsilon_b}{b} >0$. On the other hand, the existence of $[t_1,\ldots,t_l]$ gives $\frac{a}{a-\epsilon_a}>\frac{b}{\epsilon_b}$, which is a contradiction. Therefore we do have contraction.
\end{proof}
This finishes the proof of Proposition~\ref{Q4answered}
\end{proof}

\begin{example}
Let us quickly classify triangles, i.e. $a=b=c=1$. In this case $\lambda =2$, and we have that the zero continued fraction is $$[b_s,\ldots,b_{i_3}-1,1,b_{i_1}-1, \ldots,b_1]=0,$$ and so $\frac{\alpha}{\beta}=[b_{i_2}-1,\ldots,b_s]$ and $\frac{\alpha}{\alpha-\beta}=[b_{i_1}-1, \ldots,b_1]$. The smallest case is $\frac{12}{12-7}=[3,2,3]$ with M-resolution is $(2)-[4]-(2)$.
\end{example}

\begin{example}
Not any $a,b,c$ with gcd$(a,b)$ dividing $c$ is possible. For example, if $b=1$ and $a=2$, then $c=2\lambda-3$ and so $c$ must be odd.  
\end{example}

\section{N-resolution of the minimal resolution} \label{s5}

Let $0 < \Omega < \Delta$ be coprime integers, and let $P \in \overline{W}$ be a c.q.s. of type ${1\over \Delta}(1,\Omega)$. 
In~this section we specialize our results to the case when the M-resolution $W^+$ is the minimal resolution of $\oW$, i.e.~all points $P_i$'s are smooth points. The corresponding smoothing $Y\rightsquigarrow\oW$ is by definition a smoothing from the Artin component. 

\begin{example}
Suppose $W^+$ is an extremal  minimal resolution, i.e. $W^+\to\oW$ is a contraction of a single smooth rational curve $\Gamma$ of self-intersection $-\Delta\le -2$. A~concrete example of a surface $\oW$ satisfying Assumption~\ref{assume} 
is the projective cone $\P(1,1,\Delta)$ over a rational normal curve of degree $\Delta$. 
We  take a ruling of the cone as a divisor $\bar A$. The surface $W^+$ is  the Hirzebruch surface with a negative curve~$\Gamma$ and a ruling ~$A$.
Its general smoothing $Y$ is either $\P^1\times\P^1$ or $\Bl_p\P^2$ depending on parity of $\Delta$.
The derived category $D^b(W)$ contains an admissible subcategory $\langle \cA^{W^+}_1, \cA^{W^+}_0\rangle$, which in this case is just generated by an exceptional collection $\langle\cO_{W^+}(-A-\Gamma), \cO_{W^+}(-A)\rangle$, i.e.~Kawamata sheaves on $W^+$ are line bundles. This subcategory deforms to a subcategory $\langle E_1,E_0\rangle\subset D^b(Y)$, i.e.~in this case Hacking bundles are line bundles equal to the corresponding Kawamata bundles. While $\Ext^k_{W^+}(\cO(-A-\Gamma), \cO(-A))\ne0$ for $k=0,1$, we have
$\Hom_Y(E_1,E_0)=0$ because a general smoothing $Y$ does not contain a lift of a negative curve $\Gamma\subset W$. Thus  
$\Ext^k(E_1,E_0)=0$ for $k\ne1$ 
and $\Ext^1(E_1,E_0)=\C^\delta$
as predicted by Theorem~\ref{weFwetwET}. Here $\delta=\Delta-2$. 
The N-resolution $W^-$ is equal to $W^+$ if $\delta=0$, otherwise it is
$$[{\delta+1\choose 1}]-(1)-[*]=[\delta+3,2,\ldots,2]-(1)-[*],$$ where $[*]$ is a smooth point. The corresponding exceptional collection $\langle\bar E_1,\bar E_0\rangle$ on $Y$ is strong and contains a line bundle $\bar E_1=E_0$ and a vector bundle $\bar E_0$ of rank $\delta+1=\Delta-1$, which is the universal extension of $E_0$ by $E_1$. The Kawamata bundle $\bar F$ on $\oW$ is the maximal iterated extension of $\cO_{\oW}(-\bar A)$ by itself.
It~has rank $\Delta$.
The Kalck--Karmazyn algebra $\End(\bar F)$ is commutative and isomorphic to $\C[z_1,\ldots,z_{\delta+1}]/(z_1,\ldots,z_{\delta+1})^2$ \cite{KK}. By Theorem~\ref{weFwetwET},
$\bar F$ deforms to the Kawamata bundle 
$F=\bar E_0\oplus\bar E_1$ on $Y$.
$\End(\bar F)$ deforms to the algebra 
$\End(\bar E_0\oplus\bar E_1)$ of representations of the Kronecker quiver with $\delta$ arrows. The case $\Delta=4$ is especially interesting because this is the only case when $\oW=\P(1,1,4)$ has another component in the versal deformation space: a $\Q$-Gorenstein smoothing to $Y=\P^2$. In this case $W^+=W^-=\oW$ and the Hacking bundle $E$ has rank $2$ and is isomorphic to $\Omega^1_{\P^2}(1)$. The Kawamata bundle $\bar F$ of rank $4$ deforms to $E^{\oplus2}$ and the Kalck--Karmazyn algebra $\C[z_1,z_2,z_3]/(z_1,z_2,z_{3})^2$ to $\End(E^{\oplus2})=\Mat_2(\C).$
\end{example}

\begin{example}\label{generalminimal}
In general, we write the Hirzebruch-Jung continued fraction as $$\frac{\Delta}{\Omega} = [\underbrace{2,\ldots,2}_{y_1}, x_1, \underbrace{2,\ldots,2}_{y_2},x_2,\ldots,\underbrace{2,\ldots,2}_{y_{e-1}}, x_{e-1}, \underbrace{2,\ldots,2}_{y_e}],$$ where $y_i\geq 0$ and $x_i\geq 3$ for all $i$. This describes the minimal resolution $W^+$. 
We now compute the N-resolution $W^-$ of $W^+$ explicitly. 
The dual fraction is $$\frac{\Delta}{\Delta-\Omega} = [y_1+2,\underbrace{2,\ldots,2}_{x_1-3}, y_2+3, \underbrace{2,\ldots,2}_{x_2-3},y_3+3,\ldots,y_{e-1}+3,\underbrace{2,\ldots,2}_{x_{e-1}-3},y_e+2].$$ 
Using the notation of Section \ref{s1}, we have that the $d_i \neq 0$ are exactly in the positions of the $y_i$. In particular, if $d_{i_1},\ldots,d_{i_e}$ are the $d_i\neq 0$ with $i_1<\ldots<i_e$, then $i_1=1$, $i_e$ is the index of the last position, and $d_{i_k}=y_k+1$ for all $k$. Note that the non zero $\bar \delta$ are computed via $[\underbrace{2,\ldots,2}_{x_i-3}]$, and so they are equal to $x_i-2$. We have that the data $\bar n_{i_k}, \bar a_{i_k}$ for the distinct Wahl singularities in the N-resolution is $$\frac{\bar n_{i_k}}{\bar n_{i_k}- \bar a_{i_k}}=[y_1+2,\underbrace{2,\ldots,2}_{x_{1}-3},y_{2}+3,\ldots,y_{k-2}+3,\underbrace{2,\ldots,2}_{x_{k-2}-3},y_{k-1}+3,\underbrace{2,\ldots,2}_{x_{k-1}-3}]$$ for $k>1$, and smooth point for $k=1$. It follows that a general smoothing $Y$ of $\oW$
carries a strong exceptional collection of Hacking bundles
$$\langle
\bar E^1_1,\ldots,\bar E^{y_1+1}_1, \bar E_2^1,\ldots,\bar E_2^{y_2+1}, \ldots
\bar E_e^1,\ldots,\bar E_e^{y_e+1}\rangle. 
$$
Here $E_1^j$'s are line bundles and $\rk E_k^j=\bar n_{i_k}$ for $k>1$.
Furthermore, bundles $E_k^j$ with the same $k$ are orthogonal. 
Note that we have $$ \Delta=\bar n_e (y_e+1)+ \bar n_{e-1} (y_{e-1}+1)+ \ldots + \bar n_2 (y_2+1) + (y_1+1)$$
as predicted by Theorem~\ref{weFwetwET}.
Indeed, let us consider the matrix  $$M=\begin{bmatrix} y_e+2 & -1 &  &  & & & \\ -1 & 2 & -1 &  &  & & \\  & & \ddots  & & & & \\ &  -1 & 2 & -1 & & & \\  &  & - 1 & y_{e-1} +3  & -1 & &  \\    &  &    &  & \ddots & & \\ &  &  &  & -1 & 2 & -1 \\  &  &  &  &  & -1 & y_1+2 \end{bmatrix},$$ whose diagonal has the sequence $$ \{y_e+2, \underbrace{2,\ldots,2}_{x_{e-1}-3},y_{e-1}+3, \ldots, y_3+3,\underbrace{2,\ldots,2}_{x_2-3},y_2+3,\underbrace{2,\ldots,2}_{x_1-3}, y_1+2\}.$$ Its determinant is equal to $\Delta$. On the other hand, we can use the linearity of the determinant on its first row $(y_e+2,-1,0,\ldots,0)= (1,-1,0,\ldots,0) + (y_e+1,0,0,\ldots,0)$, via the sum $M=M_1+M_2$ where $M_1$ corresponds to the continued fraction $$[1,\underbrace{2,\ldots,2}_{x_{e-1}-3},y_{e-1}+3, \ldots, y_3+3,\underbrace{2,\ldots,2}_{x_2-3},y_2+3,\underbrace{2,\ldots,2}_{x_1-3}, y_1+2],$$ and det$(M_2)=(y_e+1) \bar n_e$. But then det$(M_1)$ is the numerator of the continued fraction $$[y_{e-1}+2, \ldots, y_3+3,\underbrace{2,\ldots,2}_{x_2-3},y_2+3,\underbrace{2,\ldots,2}_{x_1-3}, y_1+2],$$ by contracting the $1$ and the consecutive $2$'s in the diagonal of $M_1$. Now we use induction on $e$ to write the claimed formula.   
\end{example}

\begin{example}\label{rationalsmoothing}
Let us realize Example~\ref{generalminimal} in a projective surface, where we can apply Theorem \ref{main}. 
In fact the following construction from \cite[Section 3]{PPSU} works for any M-resolution $W^+$ of $P \in \overline{W}$ 
and gives 
a normal rational projective surface $W^+$ that satisfies Assumption \ref{assume}. 
Its $\Q$-Gorenstein smoothing is the compactified Milnor fiber of the corresponding smoothing of $P \in \overline{W}$. 
Let $\F_1$ be the blow-up of $\P^2$, with $(-1)$-curve $S_0$. We have the fibration $\F_1 \to \P^1$ where $S_0$ is a section. Let $F$ be a fiber. Choose another section $S_{\infty}$ disjoint from $S_0$. The configuration $S_0,F,S_{\infty}$ gives us a chain of rational curves with self-intersections $\{-1,0,+1\}$. 

Let us come back to the notation  $\frac{\Delta}{\Omega} = [e_1, \ldots ,e_{\ell}]$ and $ \frac{\Delta}{\Delta-\Omega} = [b_1,\ldots,b_s]$. Then, by doing blow-ups over $S_0 \cap F$, the chain $S_0,F,S_{\infty}$ can be transformed into a new chain of rational curves with self-intersections $$\{+1,1-b_s,-b_{s-1},\ldots,-b_1,-1, -e_1,-e_2,\ldots,-e_{\ell}\}$$ where the first curve on the left corresponds to the proper transform of $S_{\infty}$, and the last curve on the right to the proper transform of $S_{0}$. Let $\pi \colon X \to \F_1$ be the corresponding composition of blow-ups. Let $G$ be the reduced total pull-back by $\pi$ of $S_0+F+S_{\infty}$. Then  $H^2(X,T_X(-\log G))=0$ by \cite[Lemma 3.3]{PPSU}. The contraction of $[e_1,\ldots,e_s]$ defines our surface $\overline{W}$ where $P$ is the singularity. By blowing-up adequately $X$ over the chain $[e_1,\ldots,e_s]$, we obtain a surface $\tilde X$ and a contraction $\tilde X \to W^+$, giving the M-resolution $W^+ \to \overline{W}$ over $P$. Then $W^+$ satisfies Assumption \ref{assume} (we can use as $A$ the ``middle" $(-1)$-curve in the original configuration in $X$). A $\Q$-Gorenstein smoothing $Y$ of $W^+$ is the compactified Milnor fiber.  We now can apply Theorem \ref{main} to $Y$. 
\end{example}

\begin{example}\label{Q5answered}
In the notation of Example~\ref{generalminimal},
if $y_1=\ldots=y_e=0$ then the algebra $\hat R$ of Theorem~\ref{main} is the representation algebra of a quiver that contains a Dynkin subquiver of type $A_{e-1}$. It follows by Example \ref{rationalsmoothing} that 
the derived category of a smooth projective rational surface
can contain as an admissible subcategory the derived category of representations of a quiver with a path of an arbitrary length.
This answers the question (Q4) from \cite{BR}.
\end{example}

\section{S.o.d.'s for maximally degenerated Dolgachev surfaces} \label{s6}

In examples of Section~\ref{s5},  the surface $W$ and its $\Q$-Gorenstein smoothing $Y$ were rational. In this section, $W$ will be rational but $Y$ will be a Dolgachev surface.


We start with a topological fact which will allow us to guarantee that Assumption \ref{assume} part (3) is satisfied in our examples.

\begin{lemma}\label{topo}
Let $Z$ be a surface with only c.q.s.~ $\{Q_0,\ldots, Q_s\}$ of type $\frac{1}{m_i}(1,q_i)$, and with $H^1(Z,\O_Z)=H^2(Z,\O_Z)=0$. Let $Z^o:=Z \setminus \{Q_0,\ldots, Q_s\}$. If $H_1(Z^o,\Z)=0$, then there is a short exact sequence $$0\to \Pic(Z)\to\Cl(Z)\to \mathop{\oplus}\limits_{i=0}^s\Cl(Q_i\in Z)\to0,$$ where $\Cl(Q_i\in Z)\simeq \Z/m_i\Z$ is the local class group of $Q_i\in Z$.
\end{lemma}

\begin{proof}
We have the long exact sequence of the pair $(Z^o,\cup L_i)$ for integral homology groups, where $L_i$ is the link of $Q_i \in Z$: 
$$
\begin{CD}
\mathop{\oplus}\limits_{i=0}^s H_2(L_i,\Z)&\longrightarrow& &H_2(Z^o,\Z) &\longrightarrow& H_2(Z,\Z) &\longrightarrow& 
\mathop{\oplus}\limits_{i=0}^s H_1(L_i,\Z) &\longrightarrow& H_1(Z^o,\Z) &\longrightarrow& H_1(Z,\Z) 
\end{CD}
$$

Since $H^i(Z,\cO_Z)=0$ for $i=1,2$, we have
$H_2(Z^0,\Z)=\Pic(Z)$, $H_2(Z,\Z)=\Cl(Z)$, and
$H_1(L_i,\Z)$ is the local class group of $Q_i\in Z$,
see \cite[Prop.4.2 and 4.11]{Ko05}. The claim follows since $H_1(L_i,\Z)=\Z/m_i \Z$, $H_2(L_i,\Z)=0$, and $H_1(Z^o,\Z)=0$.
\end{proof}

\begin{corollary}\label{srgsHsr}
If 
$\oW^o$ is simply-connected,
then 
$\oW$ satisfies Assumption \ref{assume} (3).  
\end{corollary}

Let $p,q\geq 2$ be coprime integers, where $q$ is not divisible by $3$. 
We first construct a resolution $X$ of  $\oW$ by blowing up a
rational elliptic fibration $S$ with a section $\sigma_0$ and a $I_{9}$ fiber (Kodaira notation). There is a unique such elliptic fibration \cite{P90}. 
We  further blow-up $X$ to get a  resolution $\tilde X$ of an $M$-resolution $W^+$ of $\oW$. Finally, 
$D_{p,q}$ is a $\Q$-Gorenstein smoothing of $\W^+$.
The diagram  summarizes the construction: 
$$ \xymatrix{ \mathbb P^2 & S \ar[d] \ar[l] & X  \ar[l] & \tilde X  \ar[l] \ar[r] & W^+ \ar[d] &  D_{p,q} \ar@{~>}[l]  \\ & \mathbb P^1 &  &  & \oW & }$$

Here $S \to \P^2$ is the resolution of the base points of the cubic pencil which defines the elliptic fibration $S \to \P^1$.
Let us choose one of the $I_1$ fibers of $S$, and let $\pi \colon X \to S$ be the blow-up of $S$ as indicated in Figure \ref{Xpq}, where $A_1^2=B_1^2=-1$, $A_i^2=B_i^2=-2$ for $i>1$. 
The curves $G_0,\ldots,G_9$ are proper transforms of irreducible components of $I_9$ and $I_1$ fibers.
The curves $\sigma_0,\sigma_1,\sigma_2$ in Figure \ref{Xpq} are the $3$ sections of $S \to \P^1$. 
\begin{figure}[htbp]
\includegraphics[width=3.5in]{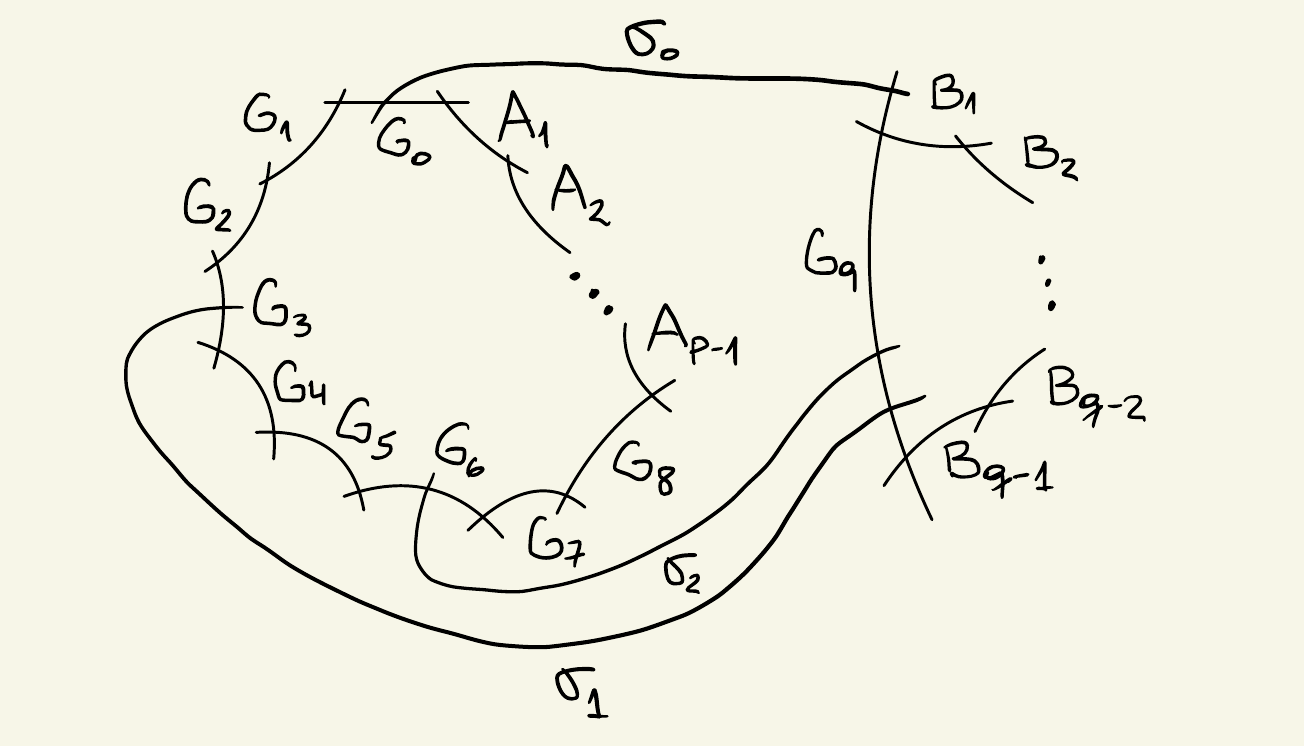}
\caption{$X$ as a blow-up of $S$ ($I_9+3I_1$)
}\label{Xpq}
\end{figure}
Note that $\rho(X)=10+p-1+q-1=p+q+8$.
The surface $\oW$ is obtained by contracting the chain
$B_2,\ldots,B_{q-1},G_q,\sigma_0,G_0,G_1,\ldots,G_8,A_{p-1},\ldots,A_2\subset X$, which corresponds to the continued fraction $$[\underbrace{2,\ldots,2}_{q-2},q+2,1,p+1,\underbrace{2,\ldots,2}_{7},3,\underbrace{2,\ldots,2}_{p-2}].$$ 
In~particular, $\oW$ has Picard number  $1$.

\begin{proposition}
The surface 
$\oW^o$ is simply-connected. 
\label{sc}
\end{proposition}

\begin{proof}

We are going to use Mumford's computation \cite[p.99]{Mum} on the resolution $X \to \oW$, which is not minimal. To show that $\pi_1(\oW^o)=\{1\}$,  we use Van-Kampen's Theorem  for a neighborhood of $P$ and its complement. As $\pi_1(\oW)=\{1\}$, it suffices to show that a generator loop around the exceptional divisor is homotopically trivial in the complement of the exceptional divisor. We note that loops $\alpha$ and $\beta$ around the ending $(-2)$-curves $A_2$ and $B_2$ respectively are (each) generators of the fundamental group of the link of $P \in \oW$. Let $\gamma$ be a loop around~$\sigma_0$. By Mumford's computation, we have that $\alpha^{9p^2}$ if conjugate to $\gamma$, as well as $\beta^{q^2}$. We claim that $\gamma$ is homotopically trivial in the complement of the exceptional divisor.
Given the claim, $\alpha^{9p^2}$ and $\beta^{q^2}$ are trivial. But because gcd$(9p^2,q^2)=1$, we have that both loops $\alpha$ and $\beta$ are actually trivial as had to be demonstrated. 

To prove the claim we use that  $\sigma_0$ intersects another $I_1$ fiber transversally at one point.
\begin{figure}[htbp]
\includegraphics[width=3.5in]{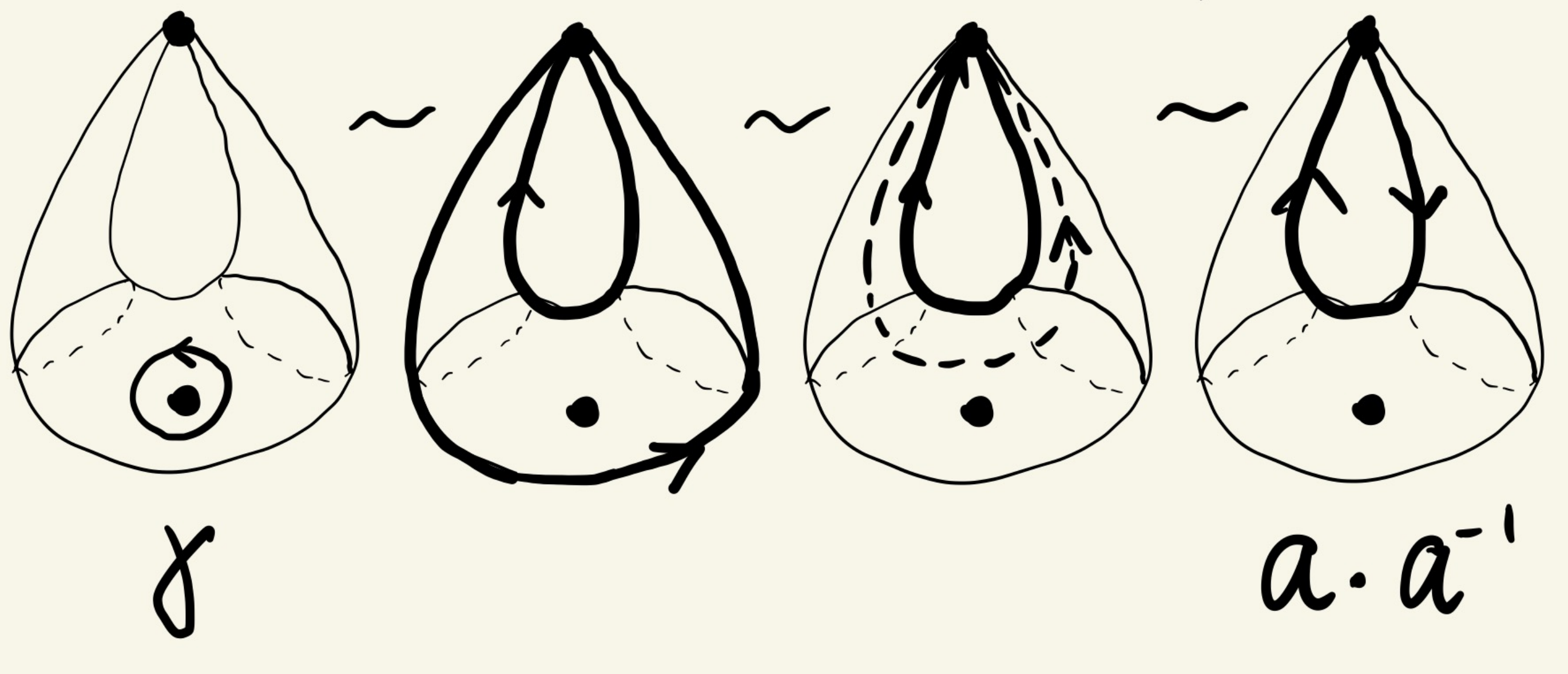}
\caption{Homotopy from $\gamma$ to identity
}\label{homotopy}
\end{figure}
It follows that $\gamma$ is a loop around a puncture in the nodal cubic punctured in one point. As Figure~\ref{homotopy} shows, this loop is homotopically trivial.
\end{proof}

We now consider a composition of blow-ups $\tilde{X} \to X$ to obtain $9$ Wahl chains of type $[p+2,2,\ldots,2]$ over the $I_9$ fiber together with the Wahl chain $[q+2,2,\ldots,2]$ from the chosen $I_1$ fiber. The Wahl chains $[p+2,2,\ldots,2]$ have the proper transform of $G_i$ as $-(p+2)$-curve for $i=0,\ldots,8$. The contraction of these Wahl chains defines the surface $W^+$ with nine $\frac{1}{p^2}(1,p-1)$ and one $\frac{1}{q^2}(1,q-1)$ singularities. The surface $W^+$ is an $M$-resolution of $\oW$ with curves $\Gamma_i$, $i=1,\ldots 9$ defined as follows: $\Gamma_1$ is the image of $\sigma_0$, the rest are the images of the $(-1)$-curves connecting the Wahl chains $[p+2,2,\ldots,2]$.
The construction of 
$D_{p,q}$  via $\Q$-Gorenstein smoothings of surfaces $W^+$ was considered in \cite[Cor.~4.3]{U16}. 
In particular,  $H^2(\overline{W},T_{\overline{W}})=0$. 

Now we can apply Theorem \ref{main} to prove the following.

\begin{theorem}
A Dolgachev surface $D_{p,q}$
has a H.e.c. $E_{9},\ldots,E_0$ associated with $W^+$ and a strong H.e.c. $\bar E_{9},\ldots,\bar E_0$ associated with the N-resolution $W^-$,
where

\begin{enumerate}
    \item $\delta_1=pq-p-q$ and $\delta_i=0$ for $i=2,\ldots,9$, in particular 
    $\End(\bar E_{9}\oplus\ldots\oplus\bar E_0)$ is the endomorphism algebra of the quiver with vertices $\bar P_0,\ldots,\bar P_9$ and with $pq-p-q$ arrows connecting each $\bar P_i$ to $\bar P_9$ for $i=0,\ldots,8$. 

    \item $\bar n_i=\rank \bar E_i$ and Wahl singularities in the N-resolution can be computed using continued fractions $\frac{\bar n_k^2}{\bar n_k \bar a_k-1}= [\underbrace{2,\ldots,2}_{q-2},q+1,p+1,\underbrace{2,\ldots,2}_{p-4},3,\underbrace{2,\ldots,2}_{q-2},q]$\break 
    for $k=0,\ldots8$ and $\frac{\bar n_9^2}{\bar n_9 \bar a_9-1}= [\underbrace{2,\ldots,2}_{q-2},q+2]$.
    
    \item The orthogonal complement of 
    $\langle\bar E_{9},\ldots,\bar E_0\rangle\subset D^b(Y)$ has Mukai lattice $\Z^2$ with Euler pairing given by the Gram matrix \begin{equation}\label{wrgsgSGwhg}\left(\begin{matrix}
-1 & 3(pq-p-q) \cr
0  & -1\cr
\end{matrix}\right).\end{equation}
    This lattice has a full numerical exceptional collection if and only if $p=3$, $q=2$.
    \end{enumerate}
\label{I93I1}
\end{theorem}

\begin{example}
For $p=3$, $q=2$, 
we obtain a M-resolution $$[4]-(1)-[5,2]-(1)-[5,2]-(1)-\ldots-(1)-[5,2],$$ and a N-resolution $$[3,5,2]-(1)-[3,5,2]-(1)-\ldots-(1)-[3,5,2]-(1)-[4].$$
\end{example}

\begin{proof}

We first compute the numerical data.
Let us assume $p,q \geq 3$. Then the c.q.s. $P \in \overline{W}$ is  $\frac{\Delta}{\Omega}=[\underbrace{2,\ldots,2}_{q-2},q+1,p,\underbrace{2,\ldots,2}_{7},3,\underbrace{2,\ldots,2}_{p-2}]$, and so the dual fraction is $\frac{\Delta}{\Delta-\Omega}=[q,\underbrace{\overline{2},2,\ldots,2}_{q-2},3,\underbrace{2,\ldots,2}_{p-3},\overline{10},p]$, where $\overline{2}$ is the position of $d_{i_1}=1$, and $\overline{10}$ is the position of $d_{i_2}=9$. In this way $\frac{\bar n_k}{\bar n_k - \bar a_k}=[q,\underbrace{2,\ldots,2}_{q-2},3,\underbrace{2,\ldots,2}_{p-3}]$ for all $k=0,\ldots8$, and $\frac{\bar n_9}{\bar n_9 - \bar a_9}=[q]$. Therefore, the Wahl singularities in the N-resolution have continued fractions as in part (2). The only $\delta\ne0$  is $\delta_1=qp-p-q$, since it is the numerator of $[\underbrace{2,\ldots,2}_{q-3},3,\underbrace{2,\ldots,2}_{p-3}]$. It remains to compute the Mukai lattice. 

Recall that if $Z$
is a rational projective normal surface
then we have an isomorphism of abelian groups
$(r,c_1,\chi)\colon \,G_0(Z)\to\Z\oplus \Cl Z \oplus \Z$,
where $G_0(Z)$ is the Grothendieck group of $D^b(Z)$, see \cite[Lemma 4.2, Remark 4.3]{KKS}.


\begin{lemma} \label{K0X}
Let $X$ be a resolution of singularities (not necessarily minimal) with the exceptional divisor $C_1, C_2,\ldots, C_s$
of a c.q.s.~surface $\overline{W}$ satisfying Assumption~\ref{assume}. The Mukai lattice of $\cB^W$ is isomorphic to a sublattice in $K_0(X)$ formed by elements $\beta$ that satisfy the  equations
$\chi(\beta)=0$, 
$r(\beta)=-C_1\cdot c_1(\beta)$, and $ C_i\cdot c_1(\beta)   =0$ for $i>1$.
\end{lemma}

\begin{proof}
We recall that the Euler pairing $\chi(\alpha,\beta)=\sum (-1)^i\ext^i(\alpha,\beta)$  has a form
$$\chi(\alpha,\beta)=r(\alpha) \ch_2(\beta) - c_1(\alpha)\cdot c_1(\beta) + r(\beta) \ch_2(\alpha)\qquad{}$$
$${}\qquad+{1\over 2}\left(r(\beta)c_1(\alpha) - r(\alpha)c_1(\beta)\right)\cdot K_X + r(\alpha)r(\beta)\chi(\cO_X).$$
In our case $\chi(\cO_X)=1$.
The subcategory $\cB^W$ is the orthogonal complement to the exceptional collection of line bundles $\O_X(-C_s-\ldots-C_1-C_0),\O_X(-C_{s-1}-\ldots-C_1-C_0),\ldots,\O_X(-C_1-C_0), \O_X(-C_0)$, where $C_0$ is the strict transform of $\bar A$ in $\oW$ (Assumption \ref{assume} (3)). As tensoring by $\O_X(C_0)$ is an autoequivalence of $D^b(X)$, for this computation we consider the orthogonal complement of $L_s:=\O_X(-C_s-\ldots-C_1),L_{s-1}:=\O_X(-C_{s-1}-\ldots-C_1),\ldots,L_1:=\O_X(-C_1), L_0:=\O_X$, which have classes
 $r(\alpha_i)=1$, $c_1(\alpha_i)=L_i$, $\ch_2(\alpha_i)={L_i^2\over 2}$.
It follows that
$$ \ch_2(\beta) - L_i\cdot c_1(\beta) + r(\beta){L_i^2\over 2}+{1\over 2}(r(\beta)L_i - c_1(\beta))\cdot K_X + r(\beta)=0$$
We start with $L_0=0$, which gives
$ \ch_2(\beta)  -{1\over 2}c_1(\beta)\cdot K_X + r(\beta)=0$.
The remaining equations then become
$ - 2L_i\cdot c_1(\beta) + r(\beta)(K+L_i)\cdot L_i  =0$.
Since $p_a(L_i)=0$ for every $i>0$, the equation in fact is simply
$ L_i\cdot c_1(\beta) +r(\beta)  =0$.
Finally, notice that $L_i-L_{i-1}=C_i$, which gives equations in the statement.
\end{proof}

For any rational elliptic fibration $S$ with a section $\sigma_0$,
$
\Pic(S)=-\bE_8\oplus\langle 1\rangle\oplus\langle -1\rangle,
$,
where $\langle -1\rangle$ is generated by $\sigma_0$ and $\langle 1\rangle$
is generated by $\sigma_0+F$, where $F\sim -K_S$ is a general fiber. Furthermore, $-\bE_8$ 
contains a sublattice $T$ generated by components of reducible fibers that do not intersect $\sigma_0$.

In our case we consider the basis $G_1, G_2, G_3, G_4, G_5, G_6, G_7, v:=\sigma_1-\sigma_2$ of $-\bE_8$. Let $\beta \in \cB$. By Lemma \ref{K0X}, we have that $(r(\beta),c_1(\beta),\chi(\beta))=(-B_2 \cdot B,B,0),$ where $B$ satisfies $C_i\cdot B=0$ for $i>1$ and $C_1:=B_2$. Let us write $$B= z_1 G_1+ z_2 G_2 + z_3 G_3 + z_4 G_4 + z_5 G_5 + z_6 G_6 + z_7 G_7 + \ \ \ \ \ \ $$ $$ \ \ \ \ \ \ \ x v + y (\sigma_0+F) + z \sigma_0 + \sum_{i=1}^{q-1} x_i \beta_i + \sum_{i=1}^{p-1} y_i \alpha_i,$$
where
$\alpha_i=A_1+\ldots+A_i$ and $\beta_j=B_1+\ldots+B_j$ for $i=1,\ldots,p-1$ and $j=1,\ldots, q-1$. 
By intersecting $B$ 
with $$C_2:=B_3, C_3:=B_4,\ldots, B_{q-1}, G_9, \sigma_0, G_0, G_1, G_2, \ldots, G_8,A_{p-1},\ldots,C_{p+q+7}:=A_2,$$
we obtain the linear system $$
0= x_2-x_3, \ \ \ 
0= x_3-x_4, 
\ldots, 
0=  x_{q-2}-x_{q-1}, $$ $$
0= y+z+x_1+x_2+\ldots+x_{q-2}+2 x_{q-1}, \ \ \ 
0=  -z $$ $$ 
0=  z_1+y+z+y_1+y_2+\ldots+y_{p-1}, \ \ \  
0=  -2 z_1+z_2, \ \ \  
0=  z_1 -2 z_2 + z_3, $$ $$
0=  z_2 - 2 z_3 + z_4 + x, \ \ \  
0=  z_3 - 2 z_4 + z_5, \ \ \ 
0=  z_4 - 2 z_5 + z_6, $$ $$ 
0=  z_5 - 2 z_6 + z_7 -x, \ \ \  
0=  z_6 - 2 z_7, \ \ \  
0=  z_7+ y_{p-1}, \ \ \  
0=  -y_{p-1}+y_{p-2}, 
\ldots, 
0=  -y_{2}+y_{1},$$ which has solutions  $y_1=\ldots=y_{p-1}=-3 z_5$, $x_2=\ldots=x_{q-1}$, $x=-8z_5$, $y=3p z_5$, $z=0$, $x_1= -3pz_5-(q-1)x_2$, $z_1=-3z_5$, $ z_2=-6z_5$, $z_3= -9z_5$, $z_4=-4 z_5$, $z_6=6 z_5$, and $z_7=3 z_5$. This solution can be expressed via the $\Z$-basis $$v_1:= -3 G_1 -6 G_2 - 9 G_3 -4 G_4 + G_5 + 6 G_6 + 3 G_7 - 8 v + 3p (\sigma_0+F) - 3p B_1 - 3 \sum_{i=1}^{p-1} \alpha_i,$$ and $v_2:= -(q-1)B_1+\sum_{i=2}^{q-1} \beta_i$. Using Lemma \ref{K0X}, these vectors can be considered as generators $\tilde v_1, \tilde v_2$ of $\cB$ by setting $\chi(\tilde v_i)=0$, rank$(\tilde v_1)=3p$, and rank$(\tilde v_2)=q$. We have the intersection numbers $v_1 \cdot K_X=3(p-1)$, $v_1\cdot v_1=-9p+1$, $v_2 \cdot K_X=1$, $v_2 \cdot v_2=-q^2+q+1$, and $v_1 \cdot v_2= -3p(q-1)$. Therefore $\chi(\tilde v_1, \tilde v_1)=-1$,  $\chi(\tilde v_2, \tilde v_2)=-1$, $\chi(\tilde v_1, \tilde v_2)=3(pq-p-q)$, and $\chi(\tilde v_2, \tilde v_1)=0$, and so the Gram matrix is \eqref{wrgsgSGwhg}.
\end{proof}



\begin{thebibliography}{99}


\bib{Auslander}{article}{
  title={Rational singularities and almost split sequences},
  author={Auslander, M.},
  journal={Trans. of the AMS},
  year={1986},
  volume={293},
  pages={511--531}
}

\bib{BC}{article}{
 author = {Behnke, Kurt},
 author = {Christophersen, Jan Arthur},
 journal = {American Journal of Mathematics},
 number = {4},
 pages = {881--903},
 publisher = {Johns Hopkins University Press},
 title = {M-Resolutions and Deformations of Quotient Singularities},
 volume = {116},
 year = {1994}
}

\bib{BrunsHerzog}{book}
{place={Cambridge}, edition={2}, series={Cambridge Studies in Advanced Mathematics}, title={Cohen-Macaulay Rings}, DOI={10.1017/CBO9780511608681}, publisher={Cambridge University Press}, author={Bruns, Winfried},
author={Herzog, H. Jürgen}, year={1998}, collection={Cambridge Studies in Advanced Mathematics}}


\bib{BR}{article}{
   author={Belmans, Pieter},
   author={Raedschelders, Theo},
   title={Embeddings of algebras in derived categories of surfaces},
   journal={Proc. Amer. Math. Soc.},
   volume={145},
   date={2017},
   number={7},
   pages={2757--2770}
}




\bib{C20}{article}{
author = {Cho, Yonghwa},
title = {Orthogonal exceptional collections from Q-Gorenstein degeneration of surfaces},
journal = {Communications in Algebra},
volume = {50},
number = {12},
pages = {5410-5429},
year  = {2022},
publisher = {Taylor & Francis},
}


\bib{CL18}{article}{
   author={Cho, Yonghwa},
   author={Lee, Yongnam},
   title={Exceptional collections on Dolgachev surfaces associated with
   degenerations},
   journal={Adv. Math.},
   volume={324},
   date={2018},
   pages={394--436}
}

\bibitem[Ch]{Ch} J. A. Christophersen, {\em On the components and discriminant of the versal base space of
cyclic quotient singularities}, In: Singularity theory and its applications, Part I (Coventry,
1988/1989), volume 1462 of Lecture Notes in Math., pages 81--92. Springer, Berlin, 1991.
	
\bib{ELS}{article}{
  title={Smoothness of Derived Categories of Algebras},
  author={Elagin, Alexey},
  author={Lunts, Valery A.},
  author={Schnurer, Olaf M.},
  journal={Moscow Mathematical Journal},
  volume={20},
  issue={2}
  pages={277-–309},
  year={2020}
}

\bib{Gab74}{article}{
author={Gabriel, Peter},
title={Finite representation type is open},
journal={Lecture Notes in Math.},
volume={488},
year={1974},
pages={132--155}
}

\bib{H04}{article}{
author = {Hacking, Paul},
journal = {Duke Math. J.},
month = {08},
number = {2},
pages = {213--257},
title = {Compact moduli of plane curves},
volume = {124},
year = {2004}
}


\bib{H13}{article}{
   author={Hacking, Paul},
   title={Exceptional bundles associated to degenerations of surfaces},
   journal={Duke Math. J.},
   volume={162},
   date={2013},
   number={6},
   pages={1171--1202}
}

\bib{H16}{article}{
   author={Hacking, Paul},
   title={Compact moduli spaces of surfaces and exceptional vector bundles},
   conference={
      title={Compactifying moduli spaces},
   },
   book={
      series={Adv. Courses Math. CRM Barcelona},
      publisher={Birkh\"{a}user/Springer, Basel},
   },
   date={2016},
   pages={41--67}
}

\bib{HP}{article}{
   author={Hacking, Paul},
   author={Prokhorov, Yuri},
   title={Smoothable del Pezzo surfaces with quotient singularities},
   journal={Compos. Math.},
   volume={146},
   date={2010},
   number={1},
   pages={169--192}
}



\bib{HTU17}{article}{
   author={Hacking, Paul},
   author={Tevelev, Jenia},
   author={Urz\'{u}a, Giancarlo},
   title={Flipping surfaces},
   journal={J. Alg. Geom.},
   volume={26},
   date={2017},
   number={2},
   pages={279--345}
}






\bib{KK17}{misc}{
  author = {Kalck, Martin},
  author={Karmazyn, Joseph},

  
  title = {Noncommutative Knörrer type equivalences via noncommutative resolutions of singularities},
  
  publisher = {arXiv:1707.02836},
  

  copyright = {arXiv.org perpetual, non-exclusive license}
}



\bibitem[KKS]{KKS}
    J. Karmazyn, A. Kuznetsov, E. Shinder,
    \emph{Derived categories of singular surfaces},
    J. Eur. Math. Soc. {\bf 24}, (2022) 461--526.
    
\bib{K21}{misc}{
  author = {Kawamata, Yujiro},
  
  keywords = {Algebraic Geometry (math.AG), FOS: Mathematics, FOS: Mathematics, 14J60, 14B07, 14E30},
  
  title = {Semi-orthogonal decomposition and smoothing},
  
  publisher = {arXiv:2112.14452},
}


\bib{Kaw}{article}{
  title={On multi-pointed non-commutative deformations and Calabi--Yau threefolds}, 
  volume={154}, 
  number={9}, journal={Compositio Mathematica}, 
  publisher={London Mathematical Society}, 
  author={Kawamata, Yujiro}, year={2018}, pages={1815-1842}

}

\bib{K92}{article}{
   author={Kawamata, Yujiro},
   title={Moderate degenerations of algebraic surfaces},
   conference={
      title={Complex algebraic varieties},
      address={Bayreuth},
      date={1990},
   },
   book={
      series={Lecture Notes in Math.},
      volume={1507},
      publisher={Springer, Berlin},
   },
   date={1992},
   pages={113--132}
}

\bib{Ko05}{article}{
   author={Koll\'{a}r, J\'{a}nos},
   title={Einstein metrics on five-dimensional Seifert bundles},
   journal={J. Geom. Anal.},
   volume={15},
   date={2005},
   number={3},
   pages={445--476}
}

\bib{kolkov}{book}{
  title={Singularities of the Minimal Model Program},
  author={Koll{\'a}r, J.},
  author={Kovacs, S},
  isbn={9781107035348},
  lccn={2012043204},
  series={Cambridge Tracts in Mathematics},
  url={https://books.google.com/books?id=QwgQZHvESFQC},
  year={2013},
  publisher={Cambridge University Press}
}


\bib{KM92}{article}{
   author={Koll\'{a}r, J\'{a}nos},
   author={Mori, Shigefumi},
   title={Classification of three-dimensional flips},
   journal={J. Amer. Math. Soc.},
   volume={5},
   date={1992},
   pages={533--703}
}
	
\bib{KSB}{article}{
   author={Koll\'{a}r, J.},
   author={Shepherd-Barron, N. I.},
   title={Threefolds and deformations of surface singularities},
   journal={Invent. Math.},
   volume={91},
   date={1988},
   number={2},
   pages={299--338}
}


\bib{KuzSex}{article}{
   author={Kuznetsov, Alexander},
   title={Derived categories of families of sextic del Pezzo surfaces},
   journal={Int. Math. Res. Not. IMRN},
   date={2021},
   number={12},
   pages={9262--9339}}



\bib{LP}{article}{
   author={Lee, Yongnam},
   author={Park, Jongil},
   title={A simply connected surface of general type with $p_g=0$ and
   $K^2=2$},
   journal={Invent. Math.},
   volume={170},
   date={2007},
   number={3},
   pages={483--505}
}

\bib{M88}{article}{
   author={Mori, Shigefumi},
   title={Flip theorem and the existence of minimal models for $3$-folds},
   journal={J. Amer. Math. Soc.},
   volume={1},
   date={1988},
   number={1},
   pages={117--253}
}

\bib{M02}{article}{
   author={Mori, Shigefumi},
   title={On semistable extremal neighborhoods},
   conference={
      title={Higher dimensional birational geometry},
      address={Kyoto},
      date={1997},
   },
   book={
      series={Adv. Stud. Pure Math.},
      volume={35},
      publisher={Math. Soc. Japan, Tokyo},
   },
   date={2002},
   pages={157--184}
}

\bib{MP11}{article}{
   author={Mori, Shigefumi},
   author={Prokhorov, Yuri},
   title={Threefold extremal contractions of type (IA)},
   journal={Kyoto J. Math.},
   volume={51},
   date={2011},
   number={2},
   pages={393--438}
}

\bib{Mum}{article}{
   author={Mumford, David},
   title={The topology of normal singularities of an algebraic surface and a
   criterion for simplicity},
   journal={Inst. Hautes \'{E}tudes Sci. Publ. Math.},
   number={9},
   date={1961},
   pages={5--22}
}

\bib{N96}{article}{
   author={Neeman, Amnon},
   title={The Grothendieck duality theorem via Bousfield's techniques and
   Brown representability},
   journal={J. Amer. Math. Soc.},
   volume={9},
   date={1996},
   number={1},
   pages={205--236}
}

\bib{Or16}{article}{
title = {Smooth and proper noncommutative schemes and gluing of DG categories},
journal = {Advances in Mathematics},
volume = {302},
pages = {59-105},
year = {2016},
issn = {0001-8708},
url = {https://www.sciencedirect.com/science/article/pii/S0001870816300457},
author = {Orlov, D.}}

\bib{PPSU}{article}{
   author={Park, Heesang},
   author={Park, Jongil},
   author={Shin, Dongsoo},
   author={Urz\'{u}a, Giancarlo},
   title={Milnor fibers and symplectic fillings of quotient surface
   singularities},
   journal={Adv. Math.},
   volume={329},
   date={2018},
   pages={1156--1230}
}
\bib{PSU}{article}{
   author={Park, Heesang},
   author={Shin, Dongsoo},
   author={Urz\'{u}a, Giancarlo},
   title={Simple embeddings of rational homology balls and antiflips},
   journal={Algebr. Geom. Topol.},
   volume={21},
   date={2021},
   number={4},
   pages={1857--1880}
}

\bib{P90}{article}{
   author={Persson, Ulf},
   title={Configurations of Kodaira fibers on rational elliptic surfaces},
   journal={Math. Z.},
   volume={205},
   date={1990},
   pages={1--47}
}

\bib{Pi74}{article}{
     author = {Pinkham, Henry C.},
     title = {Deformations of algebraic varieties with $G_m$ action},
     journal = {Ast\'erisque},
     publisher = {Soci\'et\'e math\'ematique de France},
     number = {20},
     year = {1974}
}


\bib{Ri}{article}{
   author={Riemenschneider, Oswald},
   title={Deformationen von Quotientensingularit\"{a}ten (nach zyklischen
   Gruppen)},
   journal={Math. Ann.},
   volume={209},
   date={1974},
   pages={211--248}
}


\bib{Stevens}{article}{
   author={Stevens, Jan},
   title={On the versal deformation of cyclic quotient singularities},
   conference={
      title={Singularity theory and its applications, Part I},
   },
   book={
      series={Lecture Notes in Math.},
      volume={1462},
      publisher={Springer, Berlin},
   },
   date={1991},
   pages={302--319}
}

\bib{Stacks}{misc}{	
  author       = {The {Stacks project authors}},
  title        = {The Stacks project},
  howpublished = {\url{https://stacks.math.columbia.edu}},
  year         = {2022}}


	
 \bib{U16}{article}{
   author={Urz\'{u}a, Giancarlo},
   title={Identifying neighbors of stable surfaces},
   journal={Ann. Sc. Norm. Super. Pisa Cl. Sci. (5)},
   volume={16},
   date={2016},
   number={4},
   pages={1093--1122}
}


\bib{Vial}{article}{
   author={Vial, Charles},
   title={Exceptional collections and the N\'{e}ron-Severi lattice for
   surfaces},
   journal={Adv. Math.},
   volume={305},
   date={2017},
   pages={895--934}
}

\bib{Wahl}{article}{
 author = {Wahl, Jonathan},
 journal = {Ann. Math.},
 pages = {325--356},
 publisher = {Ann. Math.},
 title = {Equisingular Deformations of Normal Surface Singularities, I},
 volume = {104},
 year = {1976}
}

\bib{Z}{misc}{	
  author       = {Z\'u\~niga, Juan Pablo},
  title        = {Program to compute all M-resolutions and N-resolutions of a cyclic quotient singularity, \url{https://colab.research.google.com/drive/19aOVSpunJPOiY4TTCvt8Q5D2KQOr8_t7?usp=sharing}},
  year         = {2022}}



\end{thebibliography}
\end{document}